\documentclass[11pt]{article}
\usepackage[latin1]{inputenc}
\usepackage{amsmath,amsthm,amssymb}
\usepackage{amsfonts}
\usepackage{amsmath,amsthm,amssymb,amscd}
\usepackage{latexsym}
\usepackage{color}
\usepackage{graphicx}
\usepackage{mathrsfs}
\usepackage{cite}

\usepackage{color,enumitem,graphicx}
\usepackage[colorlinks=true,urlcolor=black,
citecolor=black,linkcolor=black,linktocpage,pdfpagelabels,
bookmarksnumbered,bookmarksopen]{hyperref}

\makeatletter \@addtoreset{equation}{section} \makeatother

\newtheorem{theorem}{Theorem}[section]
\newtheorem{definition}{Definition}[section]
\newtheorem{proposition}{Proposition}[section]
\newtheorem{lemma}{Lemma}[section]
\newtheorem{remark}{Remark}[section]

\allowdisplaybreaks

\begin{document}
\title{Normalized ground states for a biharmonic Choquard system in $\mathbb{R}^4$}

\author{Wenjing Chen\footnote{Corresponding author.}\ \footnote{E-mail address:\, {\tt wjchen@swu.edu.cn} (W. Chen), {\tt zxwangmath@163.com} (Z. Wang).}\  \ and Zexi Wang\\
\footnotesize  School of Mathematics and Statistics, Southwest University,
Chongqing, 400715, P.R. China}

\date{ }
\maketitle

\begin{abstract}
{In this paper, we study the existence of normalized ground state solutions for the following biharmonic Choquard system
\begin{align*}
  \begin{split}
  \left\{
  \begin{array}{ll}
   \Delta^2u=\lambda_1 u+(I_\mu*F(u,v))F_u (u,v),
    \quad\mbox{in}\ \ \mathbb{R}^4,  \\
    \Delta^2v=\lambda_2 v+(I_\mu*F(u,v)) F_v(u,v),
    \quad\mbox{in}\ \ \mathbb{R}^4,  \\
    \displaystyle\int_{\mathbb{R}^4}|u|^2dx=a^2,\quad \displaystyle\int_{\mathbb{R}^4}|v|^2dx=b^2,\quad u,v\in H^2(\mathbb{R}^4),\\
    \end{array}
    \right.
  \end{split}
  \end{align*}
where $a,b>0$ are prescribed, $\lambda_1,\lambda_2\in \mathbb{R}$, $I_\mu=\frac{1}{|x|^\mu}$ with $\mu\in (0,4)$, $F_u,F_v$ are partial derivatives of $F$ and $F_u,F_v$ have exponential subcritical or critical growth in the sense of the Adams inequality.
By using a minimax principle and analyzing the behavior of the ground
state energy with respect to the prescribed mass,
we obtain the existence of ground state solutions for the above problem. }



\end{abstract}

 \section{Introduction and statement of main results}
This work is devoted to the following biharmonic nonlinear Schr\"{o}dinger system with  a general Choquard nonlinear term
\begin{align}\label{back}
  \begin{split}
  \left\{
  \begin{array}{ll}
   i\frac{\partial\psi_1}{\partial t}-\Delta^2\psi_1+(I_\mu*F(\psi_1,\psi_2))F_{\psi_1}(\psi_1,\psi_2)
   =0,\quad \text{in $\mathbb{R}^N\times\mathbb{R}$},\\
   i\frac{\partial\psi_2}{\partial t}-\Delta^2\psi_2+(I_\mu*F(\psi_1,\psi_2))F_{\psi_2}(\psi_1,\psi_2)=0,\quad \text{in $\mathbb{R}^N\times\mathbb{R}$},
    \end{array}
    \right.
  \end{split}
  \end{align}
where $N\geq1$, $i$ denotes the imaginary unit, $I_\mu=\frac{1}{|x|^\mu}$ with $\mu\in [0,N)$, $\Delta^2$ denotes the biharmonic operator, $F_{\psi_1},F_{\psi_2}$ are partial derivatives of $F$. System \eqref{back} is derived in \cite{FIP,K,KS} to reveal the effects of a small fourth-order dispersion term in the Schr\"{o}dinger equation. For convenience, we write $z=(z_1,z_2)$, and $F$ satisfies:

$(H_1)$ For $j=1,2$, $ F_{z_j}(z)\in \mathbb{R}$ for $z_j\in \mathbb{R}$,
$F_{z_j}(e^{i\theta_1}z_1,e^{i\theta_2}z_2)
=e^{i\theta_j}F_{z _j}(z_1,z_2)$
for
any
$\theta_j\in \mathbb{R}$, $z_j\in \mathbb{C}$;

$(H_2)$ $\displaystyle F(z_1,z_2)=\int_{0}^{|z_1|}\int_{0}^{|z_2|}\frac{\partial^2F(s,t)}{\partial s\partial t}dsdt$ for any $z_1,z_2\in \mathbb{C}$.\\

When looking for standing waves solutions of \eqref{back}, that is, solutions for \eqref{back} of the form $\psi_1(t,x)=e^{-i\lambda_1 t}u(x)$ and $\psi_2(t,x)=e^{-i\lambda_2 t}v(x)$ with $\lambda_1,\lambda_2\in \mathbb{R}$ and $u,v\in H^2(\mathbb{R}^N)$ are time-independent real valued functions. Then $u,v$ satisfy
\begin{align}\label{e1.1}
  \begin{split}
  \left\{
  \begin{array}{ll}
   \Delta^2u=\lambda_1 u+(I_\mu*F(u))F_u(u,v),\quad \text{in $\mathbb{R}^N$},\\
   \Delta^2v=\lambda_2 v+(I_\mu*F(u))F_v(u,v),\quad \text{in $\mathbb{R}^N$}.
    \end{array}
    \right.
  \end{split}
  \end{align}
If $\lambda_1, \lambda_2\in \mathbb{R}$ are fixed parameters, there are many results for \eqref{e1.1} by variational methods, see e.g. \cite{CLZ,FS,MSV,Sa,WLW}.

Another interesting way to find solutions of \eqref{e1.1} is to search for solutions
with prescribed mass, and $\lambda_1,\lambda_2\in \mathbb{R}$ arise as Lagrange multipliers. This type of solution is called normalized solution, and this approach is particularly meaningful from the physical point of view, since in addition to there is a conservation of mass by $(H_1)$ and $(H_2)$, the mass has often an important physical meaning, e.g. it shows the power supply in nonlinear optics, or the total number of atoms in Bose-Einstein condensation. In this case, particular attention is also devoted to the least energy solutions which are also called ground state solutions, namely solutions minimizing the associated energy functional among all nontrivial solutions, and the associated energy is called ground state energy.

For the nonlinear Schr\"{o}dinger equation with a normalization constraint
\begin{align}\label{e1.3}
  \begin{split}
  \left\{
  \begin{array}{ll}
   -\Delta u=\lambda u+f(u),
    \quad\mbox{in}\ \ \mathbb{R}^N,  \\
    \displaystyle\int_{\mathbb{R}^N}|u|^2dx=a^2,\quad u\in H^1(\mathbb{R}^N).
    \end{array}
    \right.
  \end{split}
  \end{align}
If $f(u)=|u|^{p-2}u$, the associated energy functional of \eqref{e1.3} is given by $$
\mathcal{J}_1(u)=\frac{1}{2}\int_{\mathbb{R}^N}|\nabla u|^2dx-\frac{1}{p}\int_{\mathbb{R}^N}|u|^pdx.
$$
From the variational point of view,
$\mathcal{J}_1$ is bounded from below on $
\widehat{S}(a)=\{u\in H^1(\mathbb{R}^N):\int_{\mathbb{R}^N}|u|^2dx=a^2\}
$ for $p\in(2,2+\frac{4}{N})$ ($L^2$-subcritical case), here $2+\frac{4}{N}$ is called the $L^2$-critical exponent, which comes from the Gagliardo-Nirenberg inequality \cite{Nirenberg1}.
Thus, a ground state solution of \eqref{e1.3} can be found as a global minimizer of $\mathcal{J}_1$ on $\widehat{S}(a)$, see e.g. \cite{S1,S2,CL,Lions2,Shibata}. 

If $p\in(2+\frac{4}{N},2^*)$ ($L^2$-supercritical case), on the contrary, $\mathcal{J}_1$ is unbounded from below on $
\widehat{S}(a)$, so it seems impossible to search for a global minimizer to obtain a solution of \eqref{e1.3},  where $2^*=\infty$ if $N\leq2$ and $2^*=\frac{2N}{N-2}$ if $N\geq3$.
Furthermore,  since $\lambda\in \mathbb{R}$ is unknown and $H^1(\mathbb{R}^N)\hookrightarrow L^2(\mathbb{R}^N)$ is not compact (even $H^1_{rad}(\mathbb{R}^N)$), some classical methods cannot be directly used to prove the boundedness and compactness of any $(PS)$ sequence.
Jeanjean \cite{Jeanjean} first showed that a normalized ground state solution for \eqref{e1.3} does exist in this case by showing that the mountain pass geometry of $\mathcal{J}_1|_{\widehat{S}(a)}$ can construct a $(PS)$ sequence related to the Pohozaev identity, which gives the boundedness.
By using a minimax principle based on the homotopy stable family, Bartsch and Soave \cite{BS1,BS2} also presented a
new approach that is based on a natural
constraint associated to the problem and proved the existence of normalized solutions for \eqref{e1.3}.
For more related results of normalized solutions for \eqref{e1.3} in the $L^2$-supercritical case, the reader may refer to
\cite{AJM,BM,Li1,JL,Soave1,Soave2,WW} and references therein.

For the mixed dispersion nonlinear
Schr\"{o}dinger equation with a prescribed $L^2$-norm constraint
\begin{align}\label{e1.4}
  \begin{split}
  \left\{
  \begin{array}{ll}
   \Delta^2u-\beta\Delta u=\lambda u+f(u),
    \quad\mbox{in}\ \ \mathbb{R}^N,  \\
    \displaystyle\int_{\mathbb{R}^N}|u|^2dx=a^2,\quad u\in H^2(\mathbb{R}^N).
    \end{array}
    \right.
  \end{split}
  \end{align}
This kind of problem gives a new $L^2$-critical exponent $2+\frac{8}{N}$. When $\beta>0$, $f(u)=|u|^{q-2}u$ and $q\in(2,2+\frac{8}{N})$,
 Bonheure et al. \cite{BCdN} studied \eqref{e1.4} and established the existence, qualitative properties of minimizers. By using the minimax principle, Bonheure et al. \cite{BCGJ} obtained the existence of ground state solutions, and the multiplicity of radial solutions for \eqref{e1.4} when $q\in(2+\frac{8}{N},4^*)$, where $4^*=\infty$ if $N\leq4$ and $4^*=\frac{2N}{N-4}$ if $N\geq5$. More recently, Fernandez et al. have made some improvements to \cite{BCdN} and \cite{BCGJ}, we refer to \cite{FJMM} for more details. When $\beta<0$, the problem is more involved, see \cite{LZZ,BFJ} for $q\in(2,2+\frac{8}{N})$ and \cite{LY2} for $q\in(2+\frac{8}{N},4^*)$.
Moreover, Luo and Zhang \cite{LZ1} studied normalized solutions of \eqref{e1.4} with a general nonlinear term $f$, and obtained the existence and orbital stability of minimizers,
when $\beta\in \mathbb{R}$ and $f$ satisfies the suitable $L^2$-subcritical assumptions. For normalized solutions of \eqref{e1.4} with a general Choquard nonlinear term, we refer the readers to \cite{CW, CW1}.

Considering the following system with the constraints
\begin{align}\label{Bose}
  \begin{split}
  \left\{
  \begin{array}{ll}
   -\Delta u=\lambda_1 u+F_u(u,v),\quad \text{in $\mathbb{R}^N$},\\
   -\Delta v=\lambda_2 v+F_v(u,v),\quad \text{in $\mathbb{R}^N$},\\
    \displaystyle\int_{\mathbb{R}^N}|u|^2dx=a^2,\quad \displaystyle\int_{\mathbb{R}^N}|v|^2dx=b^2,\quad u,v\in H^1(\mathbb{R}^N).\\
    \end{array}
    \right.
  \end{split}
  \end{align}
This system has an important physical significance in nonlinear optics and Bose-Einstein condensation. The most famous case is that of coupled Gross-Pitaevskii
equations in dimension $N\leq3$ with $F_u(u,v)=\mu_1|u|^{p-2}u+ r_1\kappa|v|^{r_2}|u|^{r_1-2}u $, $F_v(u,v)=\mu_2|u|^{q-2}u+ r_2\kappa|u|^{r_1}|v|^{r_2-2} v$, $p=q=4$, $r_1=r_2=2$, and $\mu_1,\mu_2,\kappa>0$, which models
Bose-Einstein condensation.
The particular case in $\mathbb{R}^3$ was investigated in the companion paper \cite{BJS}, and has been further developed by many scholars, see \cite{BS0,BS1,BS2,BZZ} for normalized solutions of \eqref{Bose} in $\mathbb{R}^3$, \cite{BJ,BLZ,GJ,LZ,LWYZ,WWW} in $\mathbb{R}^N$, and \cite{NTV} in bounded domain. It is worth pointing out that in \cite{DY}, the authors first considered normalized solutions of \eqref{Bose} with general nonlinear terms involving exponential critical growth in $\mathbb{R}^2$.

The study of normalized solutions for \eqref{Bose} is a hot topic in nonlinear PDEs nowadays. However, as far as we know, there are only a few papers dealing with such problems with general nonlinear terms besides the one already mentioned \cite{DY}.
In this work, we study normalized solutions to the following biharmonic Choquard system with general nonlinear terms in $\mathbb{R}^4$
\begin{align}\label{abs}
  \begin{split}
  \left\{
  \begin{array}{ll}
   \Delta^2u=\lambda_1 u+(I_\mu*F(u)) F_u(u,v),\quad \text{in $\mathbb{R}^4$},\\
   \Delta^2v=\lambda_2 v+(I_\mu*F(u)) F_v(u,v),\quad \text{in $\mathbb{R}^4$},\\
    \displaystyle\int_{\mathbb{R}^4}|u|^2dx=a^2,\quad \displaystyle\int_{\mathbb{R}^4}|v|^2dx=b^2,\quad u,v\in H^2(\mathbb{R}^4),\\
    \end{array}
    \right.
  \end{split}
  \end{align}
where $a,b>0$,  $\lambda_1,\lambda_2\in \mathbb{R}$, $I_\mu=\frac{1}{|x|^\mu}$ with $\mu\in (0,4)$, and $F_u,F_v$ are partial derivatives of $F$ with $F_u,F_v$ have exponential subcritical or critical growth in the sense of the Adams inequality. To our best knowledge, in literature, there is no contribution devoted to the study of coupled Choquard systems with general nonlinear terms involving exponential critical growth in $\mathbb{R}^4$ (even without the $L^2$-norm constraints). Compared with the single equation, there are some difficulties to deal with:

$(i)$ We need to establish the Adams inequality for vector functions. Commonly, using the Adams inequality \cite{MSV}, Young's inequality, we can obtain the Adams inequality for vector functions. However, using this Adams inequality, we cannot get a good  upper bound of the ground state energy (the optimal upper bound is $\frac{8-\mu}{16}$) in the exponential critical growth, which is crucial to
rule out the trivial case for the weak limit of a $(PS)_{E(a,b)}$ sequence, where $E(a,b)$ is the ground state energy defined in \eqref{gse}.
Fortunately, inspired by an useful algebraic inequality given in \cite{DT}, we overcome this difficulty and establish the Adams inequality for vector functions that exactly we need in Lemma \ref{modadams}.

$(ii)$ In \cite{A1,A2,DY}, the authors considered a coupled system involving exponential critical growth in $\mathbb{R}^2$, and  the condition for the estimation on energy level is only given by an algebraic form:
\begin{equation*}
  F(u,v)\geq \kappa|(u,v)|^\sigma\quad \text{for any $(u,v)\in \mathbb{R}^2$ and some $\kappa,\sigma>0.$}
\end{equation*}
In this paper, based on the Adams functions \cite{LY3}, we will give a more natural growth condition in the exponential critical case. However, this adams functions cannot be directly applied to study normalized solutions, because the functions are unknown in some ranges and the $L^2-$ norms of $u,\nabla u, \Delta u$ are given as $O(\frac{1}{\log n})$, after a normalization, it is unfavorable to make a refined estimation. Hence, we use the modified Adams functions introduced in \cite{CW1} to complete the estimation.

Now, we introduce the precise assumptions on what our problems are studied.
Assume
that $F$ satisfies:

$(F_1)$ For $j=1,2$, $F_{z_j} (z)\in C(\mathbb{R}\times\mathbb{R},\mathbb{R})$, and $F_{z_j} (z)=o(|z|^\tau)$ as $|z|\rightarrow 0$ for some $\tau>2-\frac{\mu}{4}$;

$(F_2)$ There exists a constant $\theta>3-\frac{\mu}{4}$ such that
\begin{equation*}
  0<\theta F(z)\leq z\cdot \nabla F(z), \,\,\,\text{for all $z\in (\mathbb{R}\times \mathbb{R})\backslash (0,0)$}, \quad \text{where $\nabla F(z)=(F_{z_1} (z),F_{z_2} (z));$}
\end{equation*}

$(F_3)$ $F_{z_1}(0,z_2)\neq0$ for all $z_2\in \mathbb{R}\backslash\{0\}$ and $F_{z_2}(z_1,0)\neq0$ for all $z_1\in \mathbb{R}\backslash\{0\}$;

$(F_4)$ For any $z\in \mathbb{R}\backslash \{0\}\times \mathbb{R}\backslash \{0\}$, $0< F_{z_j}(z)z_j <(2-\frac{\mu}{4})F(z)$, $j=1,2$; 

$(F_{5})$ For any $z\in \mathbb{R}\backslash \{0\}\times \mathbb{R}\backslash \{0\}$, let $\mathfrak{F}(z)=z\cdot\nabla F(z)-(2-\frac{\mu}{4})F(z)$,
$\nabla \mathfrak{F}(z)$
exists,
and
\begin{equation*}
(3-\frac{\mu}{4})F(\hat{z})\mathfrak{F}(\tilde{z})<
F(\hat{z})\tilde{z}\cdot\nabla\mathfrak{F}(\tilde{z})+\mathfrak{F}(\hat{z})(\mathfrak{F}(\tilde{z})-F(\tilde{z})), \quad\text{for any
$\hat{z},\tilde{z}\in \mathbb{R}\backslash\{0\}\times \mathbb{R}\backslash \{0\}$};
\end{equation*}

Our main results state as follows:
\begin{theorem}\label{th3}
Assume that $F_{z_j}$($j=1,2$) has exponential subcritical growth at $\infty$, that is,
\begin{align*}
  \lim\limits_{|z|\rightarrow+\infty}\frac{|F_{z_j} (z)|}{e^{\alpha |z|^2}}=0,\quad\text{for all $\alpha>0$}.
  \end{align*}
Moreover, assume $F$ satisfies $(F_1)-(F_5)$, then
problem \eqref{abs} admits a ground state solution.
\end{theorem}

\begin{theorem}\label{th5}
Assume that $F_{z_j}$($j=1,2$) has exponential critical growth at $\infty$, that is,
\begin{align*}
  \begin{split}
  \lim\limits_{|z|\rightarrow+\infty}\frac{|F_{z_j} (z)|}{e^{\alpha |z|^2}}=\left\{
  \begin{array}{ll}
  0,&\quad \text {for all} \,\,\,\alpha>32\pi^2,\\
  +\infty,&\quad \text {for all} \,\,\,0<\alpha<32\pi^2.
    \end{array}
    \right.
  \end{split}
  \end{align*}
Moreover, assume $F$ satisfies $(F_1)-(F_5)$, and

$(F_6)$ There exists $\varrho>0$ such that $\liminf\limits_{|z_1|,|z_2|\rightarrow+\infty}\frac{F(z)[z\cdot \nabla F(z)]}{e^{64\pi^2 |z|^2}}\geq \varrho$,
\\then
 problem \eqref{abs} admits a ground state solution.
\end{theorem}

\begin{remark}
{\rm By $(F_2)$ and $(F_4)$, we have $3-\frac{\mu}{4}<\theta<4-\frac{\mu}{2}$.}
\end{remark}

\begin{remark}
{\rm

(i) $(F_3)$ is a technical condition to rule out the semitrivial case for the weak limit of a $(PS)_{E(a,b)}$ sequence.

(ii) By using $(F_4)$, we can prove that the Lagrange multipliers sequences $\{\lambda_{1,n}\}$ and $\{\lambda_{2,n}\}$ are bounded in $\mathbb{R}$, and up to a subsequence, $\lambda_{1,n}\rightarrow \lambda_{1}<0$, $\lambda_{2,n}\rightarrow \lambda_{2}<0$ as $n\rightarrow\infty$.

(iii) Under the condition $(F_{5})$, we can prove that there exists a unique $s_{(u,v)}\in \mathbb{R}$ such that $\mathcal{H}((u,v),s_{(u,v)})\in \mathcal{P}(a,b)$ for any $(u,v)\in \mathcal{S}$, where $\mathcal{H}((u,v),s)$, $\mathcal{P}(a,b)$ and $\mathcal{S}$ are defined later.}
\end{remark}

Our paper is arranged as follows. Section \ref{sec preliminaries} contains some preliminary results.  In Section \ref{fra}, we give the variational framework of problem \eqref{abs}.
In section \ref{sub}, we study problem \eqref{abs} in the exponential subcritical case. Section \ref{cri} is devoted to the exponential critical case.

\section{Preliminaries}\label{sec preliminaries}
In this section, we give some preliminaries.
For the nonlocal type problems with Riesz potential, an important inequality due to the Hardy-Littlewood-Sobolev inequality will be used in the following.
\begin{proposition}\cite[Theorem 4.3]{LL}\label{HLS}
Assume that $1<r$, $t<\infty$, $0<\mu<4$ and
$
\frac{1}{r}+\frac{\mu}{4}+\frac{1}{t}=2.
$
Then there exists $C(\mu,r,t)>0$ such that
\begin{align}\label{HLSin}
\Big|\int_{\mathbb R^{4}}(I_{\mu}\ast g(x))h(x)dx\Big|\leq
C(\mu,r,t)\|g\|_r
\|h\|_t
\end{align}
for all $g\in L^r(\mathbb{R}^4)$ and $h\in L^t(\mathbb{R}^4)$.
\end{proposition}

\begin{lemma}(Cauchy-Schwarz type inequality) \cite{Matt}
For $g,h\in L_{loc}^1(\mathbb{R}^4)$, there holds
\
\begin{equation}\label{CS}
  \int_{\mathbb R^{4}}(I_{\mu}\ast |g(x)|)|h(x)|dx\leq \Big(\int_{\mathbb R^{4}}(I_{\mu}\ast |g(x)|)|g(x)|dx\Big)^{\frac{1}{2}}\Big(\int_{\mathbb R^{4}}(I_{\mu}\ast |h(x)|)|h(x)|dx\Big)^{\frac{1}{2}}.
\end{equation}
\end{lemma}

\begin{lemma}(Gagliardo-Nirenberg inequality) \cite{Nirenberg1}
For any $u\in H^2(\mathbb{R}^4)$ and $p\geq2$, then it holds
\begin{equation}\label{GNinequality}
  \|u\|_p\leq B_{p}\|\Delta u\|_2^{\frac{p-2}{p}}\| u\|_2^{\frac{2}{p}},
\end{equation}
where $B_{p}$ is a constant depending on $p$.
\end{lemma}

\begin{lemma}\label{adams}
(i) \cite[Theorem 1.2]{Y} If $\alpha>0$ and $u\in H^2(\mathbb{R}^4)$, then
\begin{equation*}
  \int_{\mathbb{R}^4}(e^{\alpha u^2}-1)dx<+\infty;
\end{equation*}
(ii) \cite[Proposition 7]{MSV} There exists a constant $C>0$ such that
\begin{equation*}
  \sup\limits_{u\in H^2(\mathbb{R}^4), \|\Delta u\|_2\leq1}\int_{\mathbb{R}^4}(e^{\alpha u^2}-1)dx\leq C
\end{equation*}
for all $0<\alpha\leq 32\pi^2$.
\end{lemma}

\begin{lemma}\cite[Lemma 2.3]{DT}\label{alge}
Suppose that $a_1, a_2, \ldots, a_k\geq 0$ with $a_1+a_2+\cdots +a_k<1$, then there exist $p_1, p_2, \ldots, p_k>1$ satisfying $\frac{1}{p_1}+\frac{1}{p_2}+\cdots +\frac{1}{p_k}=1$ such that $p_ia_i<1$ for any $i=1, 2, \ldots, k$. Moreover, if $a_1, a_2, \ldots, a_k\geq 0$ satisfying $a_1+a_2+\cdots +a_k=1$, then we can take $p_i=\frac{1}{a_i}$ such that $\frac{1}{p_1}+\frac{1}{p_2}+\cdots +\frac{1}{p_k}=1$ and $p_ia_i=1$ for any $i=1, 2, \ldots, k$.
\end{lemma}
Let $\mathcal{X}=H^2(\mathbb{R}^4)\times H^2(\mathbb{R}^4)$ with the norm
\begin{equation*}
  \|(u,v)\|:=(\|u\|^2+\|v\|^2)^{\frac{1}{2}}=\Big(\int_{\mathbb{R}^4}|\Delta u|^2+|\Delta v|^2+|u|^2+|v|^2dx\Big)^{\frac{1}{2}},
\end{equation*}
Similar to $H^2(\mathbb{R}^4)$, $\mathcal{X}$ is a Hilbert space and satisfies
\begin{equation}\label{embedding}
  \mathcal{X}\hookrightarrow L^p(\mathbb{R}^4,\mathbb{R}^2):=L^p(\mathbb{R}^4)\times L^p(\mathbb{R}^4)\quad\text{and}\quad \mathcal{X}\hookrightarrow \hookrightarrow L^q_{Loc}(\mathbb{R}^4,\mathbb{R}^2):=L^q_{Loc}(\mathbb{R}^4)\times L^q_{Loc}(\mathbb{R}^4)
\end{equation}
for any $p\geq2$, $q\geq1$, where $\displaystyle\|(u,v)\|_{p}:=\|(u,v)\|_{{L^p(\mathbb{R}^4,\mathbb{R}^2)}}=(\|u\|_p^p+\|v\|_p^p)^{\frac{1}{p}}=\Big(\int_{\mathbb{R}^4}(|u|^p+|v|^p)dx\Big)^{\frac{1}{p}}$. Moreover, we set
\begin{equation*}
  \mathcal{S}:=\Big\{(u,v)\in \mathcal{X}:\int_{\mathbb{R}^4}|u|^2dx=a^2, \int_{\mathbb{R}^4}|v|^2dx=b^2\Big\}.
\end{equation*}

Applying Lemmas \ref{adams} and \ref{alge}, we immediately have the following lemma.
\begin{lemma}\label{modadams}
(i) If $\alpha>0$ and $(u,v)\in \mathcal{X}$, then
\begin{equation*}
  \int_{\mathbb{R}^4}(e^{\alpha |(u,v)|^2}-1)dx<+\infty;
\end{equation*}
(ii) There exists a constant $C>0$ such that
\begin{equation*}
  \sup\limits_{(u,v)\in \mathcal{X}, \|\Delta (u,v)\|_2\leq1} \int_{\mathbb{R}^4}(e^{\alpha |(u,v)|^2}-1)dx\leq C
\end{equation*}
for all $0<\alpha\leq32\pi^2$, where $\|\Delta (u,v)\|_2=\|(\Delta u,\Delta v)\|_2=(\int_{\mathbb{R}^4}(|\Delta u|^2+|\Delta v|^2)dx)^{\frac{1}{2}}$.
\end{lemma}
\begin{proof}
$(i)$ By Lemma \ref{adams} and Young's inequality, for any $t>1$ and $t'=\frac{t}{t-1}$, we have
\begin{equation*}
 \int_{\mathbb{R}^4}(e^{\alpha |(u,v)|^2}-1)dx=\int_{\mathbb{R}^4}(e^{\alpha (u^2+v^2)}-1)dx\leq \frac{1}{t}\int_{\mathbb{R}^4}(e^{t\alpha u^2}-1)dx + \frac{1}{t'}\int_{\mathbb{R}^4}(e^{t'\alpha v^2}-1)dx<\infty.
\end{equation*}
$(ii)$
Using Lemma \ref{adams} with $k=2$, we know, for any $a_1+a_2\leq 1$, there exist $p_1, p_2>1$ satisfying $\frac{1}{p_1}+\frac{1}{p_2}=1$ such that $p_1a_1\leq1,p_2a_2\leq1$.

For any $(u,v)\in \mathcal{X}$ with $\|\Delta (u,v)\|_2\leq1$, there exist $p_1,p_2>1$ satisfying $\frac{1}{p_1}+\frac{1}{p_2}=1$ such that 
$p_1\|\Delta u\|_2^2\leq1$, $p_2\|\Delta v\|_2^2\leq1$. 
By Young's inequality, we obtain
\begin{equation*}
  \int_{\mathbb{R}^4}(e^{\alpha |(u,v)|^2}-1)dx\leq \frac{1}{p_1}\int_{\mathbb{R}^4}(e^{ p_1\alpha \|\Delta u\|_2^2(\frac{u}{\|\Delta u\|_2})^2}-1)dx+\frac{1}{p_2}\int_{\mathbb{R}^4}(e^{p_2\alpha \|\Delta v\|_2^2(\frac{v}{\|\Delta v\|_2})^2}-1)dx\leq C.
\end{equation*}
\end{proof}

\begin{definition}\cite[Definition 3.1]{G}
Let $B$ be a closed subset of $X$. A class of $\mathcal{G}$ of compact subsets of $X$ is a homotopy stable family with boundary $B$ provided

$(i)$ Every set in $\mathcal{G}$ contains $B$;

$(ii)$ For any set $A\in \mathcal{G}$ and any $\eta\in C([0,1]\times X,X)$ satisfying $\eta(t,w)=w$ for all $(t,w)\in (\{0\}\times X)\cup([0,1]\times B)$, one has $\eta(\{1\}\times A)\in\mathcal{G}$.
\end{definition}
\begin{lemma}\cite[Theorem 3.2]{G}\label{Ghouss}
Let $\psi$ be a $C^1$-functional on a complete connected $C^{1}$-Finsler manifold $X$ (without boundary), and consider a homotopy stable family $\mathcal{G}$ with a closed boundary $B$. Set $\tilde{c}=\tilde{c}(\psi,\mathcal{G})=\inf\limits_{A\in \mathcal{G}}\max\limits_{v\in A}\psi(w)$ and suppose that
\begin{equation*}
  \sup\psi(B)<\tilde{c}.
\end{equation*}
Then, for any sequence of sets $\{A_n\} \subset\mathcal{G}$ satisfying $\lim\limits_{n\rightarrow\infty}\sup\limits_{w\in A_n}\psi(w)=\tilde{c}$, there exists a sequence $\{w_n\}\subset X$ such that

$(i)$ $\lim\limits_{n\rightarrow\infty}\psi(w_n)=\tilde{c}$;

$(ii)$ $\lim\limits_{n\rightarrow\infty}\|\psi'(w_n)\|_*=0$;

$(iii)$ $\lim\limits_{n\rightarrow\infty}dist(w_n,A_n)=0$.
\end{lemma}

We observe that $B=\emptyset$ is admissible, it is sufficient to follow the usual convention of defining $\sup\psi(\emptyset)=-\infty$. Since $G_1(u):=\|u\|_2^2-a^2$, $G_2(v):=\|v\|_2^2-b^2$ are of class $C^1$ and for any $(u,v)\in \mathcal{S}$, we have $\langle G'_1(u),u\rangle=2a^2>0$, $\langle G'_2(v),v\rangle=2b^2>0$. Therefore, by the implicit function theorem, $\mathcal{S}$ is a $C^{1}$-Finsler manifold.

\begin{lemma}\cite[Lemma 4.8]{Kav}\label{weakcon}
Let $\Omega\subseteq\mathbb R^{4}$ be any open set. For $1<s<\infty$, let $\{u_n\}$ be bounded in $L^s(\Omega)$ and $u_n(x)\rightarrow u(x)$ a.e. in $\Omega$. Then $u_n(x)\rightharpoonup u(x)$  in $L^s(\Omega)$.
\end{lemma}

\section{Variational framework}\label{fra}

It is standard to see that any critical point of the energy functional
\begin{align*}
\mathcal J(u,v)
=\frac{1}{2}\int_{\mathbb{R}^4}(|\Delta u|^2+|\Delta v|^2)dx-\frac{1}{2}\int_{\mathbb{R}^4}(I_\mu*F(u,v))F(u,v)dx,
\end{align*}
restricted to $\mathcal{S}$
corresponds to a solution of \eqref{abs}.

If $F$ satisfies $(F_1)$ and $F_{z_j}$($j=1,2$) has exponential subcritical growth at $\infty$, fix $q>2$, for any $\xi>0$ and $\alpha>0$, there exists a constant $C_\xi>0$
 such that
\begin{equation*}\label{fcondition1}
  | F_{z_1}(z)|,| F_{z_2}(z)|\leq \xi|z|^\tau+C_\xi|z|^q(e^{\alpha |z|^2}-1)\quad\text{for all $z=(z_1,z_2)\in \mathbb{R}\times \mathbb{R}$},
\end{equation*}
and using $(F_2)$, we have
\begin{equation}\label{fcondition02}
  |F(z)|\leq \xi|z|^{\tau+1}+C_\xi|z|^{q+1}(e^{\alpha |z|^2}-1)\quad\text{for all $z\in \mathbb{R}\times \mathbb{R}$}.
\end{equation}
By  \eqref{fcondition02}, using the Hardy-Littlewood-Sobolev inequality \cite{LL} and Lemma \ref{modadams}, we obtain $\mathcal{J}$ is well defined in $\mathcal{X}$ and $\mathcal{J}\in C^1(\mathcal{X},\mathbb{R})$ with
\begin{align*}
\langle\mathcal J'(u,v),(\varphi,\psi)\rangle=&
\int_{\mathbb{R}^4}(\Delta u \Delta \varphi+\Delta v\Delta \psi)dx
-\int_{\mathbb{R}^4}(I_\mu*F(u,v)) F_u(u,v)\varphi dx-\int_{\mathbb{R}^4}(I_\mu*F(u,v))F_v(u,v)\psi dx,
\end{align*}
for any $(u,v), (\varphi,\psi)\in \mathcal{X}$. The very same discussion applies to that $F$ satisfies $(F_1)$ and $F_{z_j}$($j=1,2$) has exponential critical growth at $\infty$.

To understand the geometry of $\mathcal J|_{\mathcal{S}}$, for any $s\in \mathbb{R}$ and $u\in H^2(\mathbb{R}^4)$, we define
\begin{equation*}
  \mathcal{H}(u,s)(x)=e^{2s}u(e^sx).
\end{equation*}
One can easily check that $\| \mathcal{H}(u,s)\|_2=\|u\|_2$ for any $s\in \mathbb{R}$. As a consequence, for any $(u,v)\in \mathcal{S}$, it holds that $\mathcal{H}((u,v),s):=(\mathcal{H}(u,s),\mathcal{H}(v,s))\in \mathcal{S}$ for any $s\in \mathbb{R}$, and
$\mathcal{H}((u,v),s_1+s_2)=\mathcal{H}(\mathcal{H}((u,v),s_1),s_2)=\mathcal{H}(\mathcal{H}((u,v),s_2),s_1)$ for any $s_1,s_2\in \mathbb{R}$.
By Lemma \ref{minimax01} (or Lemma \ref{minimax1}), we find that $\mathcal{J}$ is unbounded from below on  $\mathcal{S}$. It is well known that solutions of \eqref{abs} satisfy the Pohozaev identity
 \begin{align*}
  P(u,v)
  =\int_{\mathbb{R}^4}(|\Delta u |^2+|\Delta v |^2) dx-
  \int_{\mathbb{R}^4}(I_\mu*F(u,v))\mathfrak{F}(u,v)dx,
\end{align*}
where $\mathfrak{F}(z)=z\cdot\nabla F(z)-(2-\frac{\mu}{4})F(z)$.
This enlightens us to consider the minimization on the natural constraint
\begin{equation*}
\mathcal{P}(a,b)=\Big\{(u,v)\in \mathcal{S}:P(u,v)=0\Big\},
\end{equation*}
 i.e.,
\begin{equation}\label{gse}
  E(a,b)=\inf\limits_{(u,v)\in \mathcal{P}(a,b)}\mathcal J(u,v).
\end{equation}
As will be show in Lemma \ref{equi0} (or Lemma \ref{equi}), $\mathcal{P}(a,b)$ is nonempty, and from Lemma \ref{Pohozaev}, we can see that any critical point of $\mathcal{J}|_{\mathcal{S}}$ stays in $\mathcal{P}(a,b)$, thus any critical point $u$ of $\mathcal{J}|_{\mathcal{S}}$ with $\mathcal{J}(u)=E(a,b)$ is a ground state solution of \eqref{abs}.

\begin{lemma}\label{biancon}
Assume that $u_n\rightarrow u$ in $H^2(\mathbb{R}^4)$, and $s_n\rightarrow s$ in $\mathbb{R}$, then $\mathcal{H}(u_n,s_n)\rightarrow \mathcal{H}(u,s)$ in $H^2(\mathbb{R}^4)$ as $n\rightarrow\infty$.
\end{lemma}
\begin{proof}
The proof can be seen in Lemma $2.6$ in \cite{CW1}, we omit it here.
\end{proof}

\begin{lemma}\label{Pohozaev}
If $(u,v)\in \mathcal{X}$ is a critical point of $\mathcal J|_{\mathcal{S}}$, then $(u,v)\in \mathcal{P}(a,b)$.
\end{lemma}
\begin{proof}
If $(u,v)\in \mathcal{X}$ is a critical point of $\mathcal J|_{\mathcal{S}}$,
there exist $\lambda_1,\lambda_2\in \mathbb{R}$ such that
\begin{align}\label{critical}
&\int_{\mathbb{R}^4}(\Delta u \Delta \varphi+\Delta v\Delta \psi)dx-\int_{\mathbb{R}^4}(I_\mu*F(u,v)) F_{u}(u,v)\varphi dx-\int_{\mathbb{R}^4}(I_\mu*F(u,v))F_{v}(u,v)\psi dx=\int_{\mathbb{R}^4}(\lambda_{1}u\varphi+\lambda_{2}v\psi)dx.
\end{align}
Testing \eqref{critical} with $(\varphi,\psi)=(u,v)$,
 we have
\begin{align}\label{po2}
 \int_{\mathbb{R}^4}(|\Delta u |^2+|\Delta v |^2) dx
=\int_{\mathbb{R}^4} (\lambda _1|u|^2+\lambda _2|v|^2) dx+\int_{\mathbb{R}^4} (I_\mu*F(u,v))[(u,v)\cdot\nabla F(u,v)]dx.
\end{align}
On the other hand, consider a cut-off function $\varphi\in C_0^\infty(\mathbb{R}^4,[0,1])$ such that $\varphi(x)=1$ if $|x|\leq1$, $\varphi(x)=0$ if $|x|\geq2$.
For any fixed $\rho>0$, set $(\widetilde{u}_\rho(x),0)=(\varphi(\rho x)x\cdot\nabla u(x),0)$ and $(0,\widetilde{v}_\rho(x))=(0,\varphi(\rho x)x\cdot\nabla v(x))$ as test functions for \eqref{critical}, we obtain
\begin{equation}\label{test1}
  \int_{\mathbb{R}^4}\Delta u \Delta\widetilde{u}_\rho dx=\lambda_1 \int_{\mathbb{R}^4} u\widetilde{u}_\rho dx+\int_{\mathbb{R}^4} (I_\mu*F(u,v))F_u(u,v)\widetilde{u}_\rho dx
\end{equation}
and
\begin{equation}\label{test2}
  \int_{\mathbb{R}^4}\Delta v \Delta\widetilde{v}_\rho dx=\lambda_2 \int_{\mathbb{R}^4} v\widetilde{v}_\rho dx+\int_{\mathbb{R}^4} (I_\mu*F(u,v))F_v(u,v)\widetilde{v}_\rho dx.
\end{equation}
Moreover,
\begin{equation*}
  \int_{\mathbb{R}^4} (I_\mu*F(u,v))F_u(u,v)\widetilde{u}_\rho dx+\int_{\mathbb{R}^4} (I_\mu*F(u,v))F_v(u,v)\widetilde{v}_\rho dx=\int_{\mathbb{R}^4} (I_\mu*F(u,v))\varphi(\rho x)[x\cdot\nabla _xF(u,v)]dx.
\end{equation*}
By \cite{MS1}, we know
\begin{equation*}
  \lim_{\rho\rightarrow0}\int_{\mathbb{R}^4} u\widetilde{u}_\rho dx=-2\int_{\mathbb{R}^4}|u|^2 dx,\quad\lim_{\rho\rightarrow0}\int_{\mathbb{R}^4} v\widetilde{v}_\rho dx=-2\int_{\mathbb{R}^4} |v|^2 dx,
\end{equation*}
and a direct computation shows that
\begin{equation*}
  \lim_{\rho\rightarrow0}\int_{\mathbb{R}^4}\Delta u \Delta\widetilde{u}_\rho dx=0,\quad \lim_{\rho\rightarrow0}\int_{\mathbb{R}^4}\Delta v \Delta\widetilde{v}_\rho dx=0.
\end{equation*}
We claim that
\begin{equation*}
  \lim_{\rho\rightarrow0}\int_{\mathbb{R}^4} (I_\mu*F(u,v))\varphi(\rho x)[x\cdot\nabla _xF(u,v)]dx=-(4-\frac{\mu}{2})\int_{\mathbb{R}^4} (I_\mu*F(u,v))F(u,v)dx.
\end{equation*}
Indeed, integrating by parts, we find that
\begin{align*}
                   &\int_{\mathbb{R}^4} (I_\mu*F(u,v))\varphi(\rho x)[x\cdot\nabla _xF(u,v)]dx\\
                   =&\int_{\mathbb{R}^4}\int_{\mathbb{R}^4} |x-y|^{-\mu}F(u(y),v(y))\varphi(\rho x)[x\cdot\nabla _xF(u(x),v(x))]dxdy\\
                   =&\frac{1}{2}\int_{\mathbb{R}^4}\int_{\mathbb{R}^4}\Big( |x-y|^{-\mu}F(u(y),v(y))\varphi(\rho x)[x\cdot\nabla _x F(u(x),v(x))]\\&+|x-y|^{-\mu}F(u(x),v(x))\varphi(\rho y)[y\cdot\nabla_y F(u(y),v(y))]\Big)dxdy\\
                   =&-\frac{1}{2}\int_{\mathbb{R}^4}\int_{\mathbb{R}^4}\Big(-\mu|x-y|^{-\mu}F(u(y),v(y)) F(u(x),v(x))\frac{x(x-y)\varphi(\rho x)}{|x-y|^2}\\
                   &-\mu|x-y|^{-\mu}F(u(x),v(x)) F(u(y),v(y))\frac{-y(x-y)\varphi(\rho y)}{|x-y|^2}\\
                   &+\rho |x-y|^{-\mu}F(u(y),v(y)) [x\cdot\nabla_x\varphi(\rho x)]F(u(x),v(x))\\&+\rho|x-y|^{-\mu}F(u(x),v(x)) [y\cdot\nabla_y\varphi(\rho y)]F(u(y),v(y))\\
                   &+4|x-y|^{-\mu}F(u(y),v(y))\varphi(\rho x)F(u(x),v(x))\\&+4|x-y|^{-\mu}F(u(x),v(x))\varphi(\rho y)F(u(y),v(y))\Big)dxdy.
                 \end{align*}
Using the Lebesgue dominated convergence theorem, we prove the claim. 

Adding \eqref{test1} and \eqref{test2}, we obtain
\begin{equation*}\label{po1}
  \int_{\mathbb{R}^4}( \lambda_1|u|^2+\lambda_2|v|^2) dx+(2-\frac{\mu}{4}) \int_{\mathbb{R}^4} (I_\mu*F(u,v))F(u,v)dx=0.
\end{equation*}
This together with \eqref{po2} yields that \begin{align*}
  &\int_{\mathbb{R}^4}(|\Delta u |^2+|\Delta v |^2)dx-\int_{\mathbb{R}^4} (I_\mu*F(u,v))\mathfrak{F}(u,v) dx=0,
\end{align*}
where $\mathfrak{F}(u,v)=(u,v)\cdot\nabla F(u,v)-(2-\frac{\mu}{4})F(u,v)$. Thus, $(u,v)\in \mathcal{P}(a,b)$.
\end{proof}

 \section{Exponential subcritical case}\label{sub}
In this section, we will study problem \eqref{abs} in the exponential subcritical case, and establish the existence of ground state solutions.

\begin{lemma}\label{strong0}
Assume that $F$ satisfies $(F_1)$, $(F_2)$ and $F_{z_j}$($j=1,2$) has exponential subcritical growth at $\infty$, let $\{(u_n,v_n)\}\subset\mathcal{S}$ be a bounded sequence in $\mathcal{X}$,
up to a subsequence, if $(u_n,v_n)\rightharpoonup (u,v)$ in $\mathcal{X}$, for any $\varphi\in C_0^\infty(\mathbb{R}^4)$, we have
\begin{equation}\label{strong01}
  \int_{\mathbb{R}^4}\Delta u_n\Delta\varphi dx \rightarrow  \int_{\mathbb{R}^4}\Delta u\Delta\varphi dx,\quad \int_{\mathbb{R}^4}\Delta v_n\Delta\varphi dx \rightarrow  \int_{\mathbb{R}^4}\Delta v\Delta\varphi dx,\,\,\,\text{as}\,\,\, n\rightarrow\infty,
\end{equation}
\begin{equation}\label{strong03}
  \int_{\mathbb{R}^4}u_n\varphi dx \rightarrow  \int_{\mathbb{R}^4}u\varphi dx, \quad\int_{\mathbb{R}^4}v_n\varphi dx \rightarrow  \int_{\mathbb{R}^4}v\varphi dx,\,\,\,\text{as}\,\,\, n\rightarrow\infty,
\end{equation}
and
\begin{equation}\label{strong04}
  \int_{\mathbb{R}^4}(I_\mu*F(u_n,v_n))F_{u_n} (u_n,v_n)\varphi dx\rightarrow \int_{\mathbb{R}^4}(I_\mu*F(u,v))F_{u} (u,v)\varphi dx,\,\,\,\text{as}\,\,\, n\rightarrow\infty,
\end{equation}
\begin{equation}\label{strong05}
  \int_{\mathbb{R}^4}(I_\mu*F(u_n,v_n))F_{v_n} (u_n,v_n)\varphi dx\rightarrow \int_{\mathbb{R}^4}(I_\mu*F(u,v)) F_{v} (u,v)\varphi dx,\,\,\,\text{as}\,\,\, n\rightarrow\infty.
\end{equation}
\end{lemma}

\begin{proof}
For any fixed $v\in H^2(\mathbb{R}^4)$, define
\begin{equation*}
  f_v(u):=\int_{\mathbb{R}^4}\Delta u\Delta v dx,\quad h_v(u):=\int_{\mathbb{R}^4}u v dx,\quad \text{for every $u \in H^2(\mathbb{R}^4)$}.
\end{equation*}
Then, by the H\"{o}lder inequality, we have
\begin{equation*}
  |f_v(u)|,|h_v(u)|\leq\|v\|\|u\|,
\end{equation*}
this yields that $f_v$ and $h_v$ are continuous linear functionals on $H^2(\mathbb{R}^4)$. Thus, by $u_n\rightharpoonup u$,  $v_n\rightharpoonup v$ in $H^2(\mathbb{R}^4)$, and $C_0^\infty(\mathbb{R}^4)$ is dense in $H^2(\mathbb{R}^4)$ , we have that \eqref{strong01} and \eqref{strong03} hold.
Next, we prove \eqref{strong04}-\eqref{strong05}.
Since $\{(u_n,v_n)\}$ is bounded in $\mathcal{X}$, we can fix $\alpha>0$ close to $0$ and $t>1$ close to $1$ such that
 \begin{equation*}
 \sup\limits_{n\in \mathbb{N}^+}\frac{4\alpha t\|\Delta(u_n,v_n)\|_2^2}{4-\mu}\leq 32\pi^2.
\end{equation*}
Using Lemma \ref{modadams} and Young's inequality, it yields that
\begin{equation*}
 \sup\limits_{n\in \mathbb{N}^+} \int_{\mathbb{R}^4}(e^{\alpha |(u_n,v_n)|^2}-1)^tdx\leq \sup\limits_{n\in \mathbb{N}^+} \int_{\mathbb{R}^4}(e^{\alpha t\|\Delta(u_n,v_n)\|_2^2( \frac{|(u_n,v_n)|}{\|\Delta(u_n,v_n)\|_2})^2}-1)dx\leq C.
\end{equation*}
Therefore, $(e^{\alpha |(u_n,v_n)|^2}-1)$ is uniformly bounded in $L^t(\mathbb{R}^4)$. Since $(e^{\alpha |(u_n,v_n)|^2}-1)\rightarrow (e^{\alpha |(u,v)|^2}-1)$ a.e. in $\mathbb{R}^4$, by Lemma \ref{weakcon}, we obtain $(e^{\alpha |(u_n,v_n)|^2}-1)\rightharpoonup (e^{\alpha |(u,v)|^2}-1)$ in $L^t(\mathbb{R}^4)$.

By corollary $2.1$ of \cite{CW}, we know
\begin{equation*}
  \|I_\mu\ast\varphi\|_{\frac{4s}{4-(4-\mu) s}}\leq C
  \|\varphi\|_s
\end{equation*}
for any $\varphi\in L^s(\mathbb{R}^4)$ and $s\in(1,\frac{4}{4-\mu})$.
Hence, let $s\rightarrow \frac{4}{4-\mu}$, then $\frac{4s}{4-(4-\mu)s}\rightarrow\infty$ and
\begin{equation*}
  \|I_\mu\ast F(u_n,v_n)\|_\infty\leq C\|F(u_n,v_n)\|_{\frac{4}{4-\mu}}\quad\text{for all $n\in \mathbb{N}^+$},
\end{equation*}
and by \eqref{fcondition02},
\begin{equation*}
  |F(u_n,v_n)|\leq \xi|(u_n,v_n)|^{\tau+1}+C_\xi|(u_n,v_n)|^{q+1}(e^{\alpha |(u_n,v_n)|^2}-1)\quad\text{for all $n\in \mathbb{N}^+$}.
\end{equation*}
For $t'=\frac{t}{t-1}$, using Lemma \ref{modadams}, the H\"{o}lder inequality, and the Sobolev inequality, we have
\begin{align*}
  \|F(u_n,v_n)\|_{\frac{4}{4-\mu}}&\leq\bigg(\int_{\mathbb{R}^4}\Big(\xi|(u_n,v_n)|^{\tau+1}+C_\xi|(u_n,v_n)|^{q+1}(e^{\alpha |(u_n,v_n)|^2}-1)\Big)^{\frac{4}{4-\mu}}dx\bigg)^{\frac{4-\mu}{4}}\nonumber\\
  &\leq\bigg(\int_{\mathbb{R}^4}\Big(C|(u_n,v_n)|^{\frac{4(\tau+1)}{4-\mu}}+C|(u_n,v_n)|^{\frac{4(q+1)}{4-\mu}}(e^{\frac{4\alpha|(u_n,v_n)|^2}{4-\mu}}-1)\Big)dx\bigg)^{\frac{4-\mu}{4}}\nonumber\\
  &\leq C\|(u_n,v_n)\|_{\frac{4(\tau+1)}{4-\mu}}^{\tau+1}+C\bigg(\int_{\mathbb{R}^4}|(u_n,v_n)|^{\frac{4(q+1)}{4-\mu}}(e^{\frac{4\alpha |(u_n,v_n)|^2}{4-\mu}}-1)dx\bigg)^{\frac{4-\mu}{4}}\nonumber\\
  &\leq C\|(u_n,v_n)\|_{\frac{4(\tau+1)}{4-\mu}}^{\tau+1}+C\Big(\int_{\mathbb{R}^4}|(u_n,v_n)|^{\frac{4(q+1)t'}{4-\mu}}dx\Big)^{\frac{4-\mu}{4t'}}\Big(\int_{\mathbb{R}^4}(e^{\frac{4\alpha t|(u_n,v_n)|^2}{4-\mu}}-1)dx\Big)^{\frac{4-\mu}{4t}}\nonumber\\
  &\leq C\|(u_n,v_n)\|_{\frac{4(\tau+1)}{4-\mu}}^{\tau+1}+C\|(u_n,v_n)\|_{\frac{4(q+1)t'}{4-\mu}}^{q+1}\Big(\int_{\mathbb{R}^4}(e^{\frac{4\alpha t\|\Delta (u_n,v_n)\|_2^2 }{4-\mu}(\frac{|(u_n,v_n)|}{\|\Delta (u_n,v_n)\|_2})^2}-1)dx\Big)^{\frac{4-\mu}{4t}}\nonumber\\
  \\& \leq C\|(u_n,v_n)\|^{\tau+1}+C\|(u_n,v_n)\|^{q+1}\leq C,
\end{align*}
where we have used the fact that $\tau>2-\frac{\mu}{4}>1$, $q>2$, and $t'$ large enough.

Hence, for any $\varphi\in C_0^\infty(\mathbb{R}^4)$, we have
\begin{equation*}
  (I_\mu*F(u_n,v_n))F_{u_n}(u_n,v_n)\varphi\rightarrow (I_\mu*F(u,v))F_u(u,v)\varphi\quad \text{a.e. in $\mathbb{R}^4$},
\end{equation*}
\begin{equation*}
  |(I_\mu*F(u_n,v_n))F_{u_n}(u_n,v_n)\varphi|\leq C|F_{u_n}(u_n,v_n)||\varphi|\leq C|(u_n,v_n)|^{\tau}|\varphi|+C|(u_n,v_n)|^{q}|\varphi|(e^{\alpha |(u_n,v_n)|^2}-1)
\end{equation*}
for all $n\in \mathbb{N}^+$, and
\begin{equation*}
  |(u_n,v_n)|^{\tau}|\varphi|+|(u_n,v_n)|^{q}|\varphi|(e^{\alpha |(u_n,v_n)|^2}-1)\rightarrow |(u,v)|^{\tau}|\varphi|+|(u,v)|^{q}|\varphi|(e^{\alpha |(u,v)|^2}-1)\quad\text{a.e. in $\mathbb{R}^4$}.
\end{equation*}
Denote $\Omega:=supp \varphi$. In the following, we prove that
 \begin{equation}\label{strong06}
  \int_{\Omega}|(u_n,v_n)|^{\tau}|\varphi| dx\rightarrow \int_{\Omega}|(u,v)|^{\tau}|\varphi| dx,\,\,\,\text{as}\,\,\, n\rightarrow\infty
\end{equation}
and
\begin{equation}\label{strong07}
  \int_{\Omega}|(u_n,v_n)|^{q}|\varphi|(e^{\alpha |(u_n,v_n)|^2}-1)dx\rightarrow \int_{\Omega}|(u,v)|^{q}|\varphi|(e^{\alpha |(u,v)|^2}-1)dx,\,\,\,\text{as}\,\,\, n\rightarrow\infty.
\end{equation}
Now, we show that
\begin{equation}\label{lp}
  |(u_n,v_n)|^\tau\rightarrow |(u,v)|^\tau \quad\text{in $L^2(\Omega)$}\quad \text{and}\quad |(u_n,v_n)|^q\rightarrow |(u,v)|^q \quad\text{in $L^{t'}(\Omega)$}.
\end{equation}
Since $\tau>2-\frac{\mu}{4}>1$, we have $u_n\rightarrow u$ and $v_n\rightarrow v$ in $L^{2\tau}(\Omega)$, thus $|u_n|^2\rightarrow |u|^2$, $|v_n|^2\rightarrow |v|^2$ in $L^{\tau}(\Omega)$, and
\begin{equation*}
  \||(u_n,v_n)|^2-|(u,v)|^2\|_{L^\tau(\Omega)}=\||u_n|^2-|u|^2+|v_n|^2-|v|^2\|_{L^\tau(\Omega)}\leq \||u_n|^2-|u|^2\|_{L^\tau(\Omega)}+\||v_n|^2-|v|^2\|_{L^\tau(\Omega)},
\end{equation*}
which implies $|(u_n,v_n)|^2\rightarrow |(u,v)|^2 $ in $L^\tau(\Omega)$, and we have $|(u_n,v_n)|^\tau\rightarrow |(u,v)|^\tau $ in $L^2(\Omega)$. Similarly, we can prove that $|(u_n,v_n)|^q\rightarrow |(u,v)|^q$ in $L^{t'}(\Omega)$.
By the definition of weak convergence, we obtain \eqref{strong06}, and
\begin{align*}
  &\Big|\int_{\Omega}|(u_n,v_n)|^{q}|\varphi|(e^{\alpha |(u_n,v_n)|^2}-1)dx-\int_{\Omega}|(u,v)|^{q}|\varphi|(e^{\alpha |(u,v)|^2}-1)dx\Big|\\
  &\leq\|\varphi\|_\infty\int_{\Omega}\Big||(u_n,v_n)|^{q}-|(u,v)|^{q}\Big|(e^{\alpha |(u_n,v_n)|^2}-1)dx\\&\quad+\|\varphi\|_\infty\int_{\Omega}|(u,v)|^{q}\Big|(e^{\alpha |(u_n,v_n)|^2}-1)-(e^{\alpha |(u,v)|^2}-1)\Big|dx\\
  &\leq \|\varphi\|_\infty \Big(\int_{\Omega}\Big||(u_n,v_n)|^{q}-|(u,v)|^{q}\Big|^{t'}dx\Big)^{\frac{1}{t'}}\Big(\int_{\Omega}(e^{\alpha |(u_n,v_n)|^2}-1)^tdx\Big)^{\frac{1}{t}}\\
  &\quad+\|\varphi\|_\infty\int_{\Omega}|(u,v)|^{q}\Big|(e^{\alpha |(u_n,v_n)|^2}-1)-(e^{\alpha |(u,v)|^2}-1)\Big|dx\rightarrow 0,\,\,\,\text{as}\,\,\, n\rightarrow\infty.
\end{align*}
Applying a variant of the Lebesgue dominated convergence theorem, we obtain \eqref{strong04}. Similar argument shows that  \eqref{strong05} holds.
\end{proof}

\subsection{The behaviour of $E(a,b)$}

The goal of this subsection is to characterize the behaviour of $E(a,b)$ as $a,b>0$ vary. In particular, we will prove the continuity of the function $(a,b)\mapsto E_r(a,b)$, the monotonicity of the functions $a\mapsto E(a,b)$ and $b\mapsto E(a,b)$, and the limit behaviour of $E(a,b)$ as $(a,b)\rightarrow(0,0)$, where $E_r(a,b)$ is the radial ground state energy defined in \eqref{rgse}.
\begin{lemma}\label{minimax01}
Assume that $F$ satisfies $(F_1)$, $(F_2)$ and $F_{z_j}$($j=1,2$) has exponential subcritical growth at $\infty$, then for any fixed $(u,v)\in \mathcal{S}$, we have

$(i)$ $\mathcal{J}(\mathcal{H}((u,v),s))\rightarrow0^+$ as $s\rightarrow-\infty$;

$(ii)$ $\mathcal{J}(\mathcal{H}((u,v),s))\rightarrow-\infty$ as $s\rightarrow+\infty$.
\end{lemma}
\begin{proof} A straightforward calculation shows that for any $q>2$,
\begin{equation*}
\|\mathcal{H}(u,s)\|_2=a,\quad \|\mathcal{H}(v,s)\|_2=b,
\end{equation*}
\begin{equation*}
  \quad \|\Delta \mathcal{H}(u,s)\|_2=e^{2s}\|\Delta u\|_2,\quad \|\Delta \mathcal{H}(v,s)\|_2=e^{2s}\|\Delta v\|_2,
\end{equation*}
and
\begin{equation*}
  \|\mathcal{H}((u,v),s)\|_q=e^{\frac{2(q-2)s}{q}}\|(u,v)\|_q.
\end{equation*}
So there exists $s_1<<0$ such that
\begin{equation*}
  \|\Delta \mathcal{H}((u,v),s)\|_2^2= \|\Delta \mathcal{H}(u,s)\|_2^2+ \|\Delta\mathcal{H}(v,s)\|_2^2\leq C
\end{equation*}
for all $s\leq s_1$.
Fix $\alpha>0$ close to $0$ and $t>1$ close to $1$ such that
\begin{equation*}
\frac{8\alpha t\|\Delta \mathcal{H}((u,v),s)\|_2^2}{8-\mu}\leq 32\pi^2.
\end{equation*}
For $t'=\frac{t}{t-1}$, using Lemma \ref{modadams}, the H\"{o}lder inequality, and the Sobolev inequality, we have
\begin{align}\label{close01}
 &\int_{\mathbb{R}^4}(I_\mu*F(\mathcal{H}((u,v),s)))F(\mathcal{H}((u,v),s))dx\leq \|F(\mathcal{H}((u,v),s))\|_{\frac{8}{8-\mu}}^2\nonumber\\&\leq\bigg(\int_{\mathbb{R}^4}\Big(\xi|\mathcal{H}((u,v),s)|^{\tau+1}+C_\xi|\mathcal{H}((u,v),s)|^{q+1}(e^{\alpha |\mathcal{H}((u,v),s)|^2}-1)\Big)^{\frac{8}{8-\mu}}dx\bigg)^{\frac{8-\mu}{4}}\nonumber\\
  &\leq\bigg(\int_{\mathbb{R}^4}\Big(C|\mathcal{H}((u,v),s)|^{\frac{8(\tau+1)}{8-\mu}}+C|\mathcal{H}((u,v),s)|^{\frac{8(q+1)}{8-\mu}}(e^{\frac{8\alpha |\mathcal{H}((u,v),s)|^2}{8-\mu}}-1)\Big)dx\bigg)^{\frac{8-\mu}{4}}\nonumber\\
  &\leq C\|\mathcal{H}((u,v),s)\|_{\frac{8(\tau+1)}{8-\mu}}^{2(\tau+1)}+C\bigg(\int_{\mathbb{R}^4}|\mathcal{H}((u,v),s)|^{\frac{8(q+1)}{8-\mu}}(e^{\frac{8\alpha |\mathcal{H}((u,v),s)|^2}{8-\mu}}-1)dx\bigg)^{\frac{8-\mu}{4}}\nonumber\\
  &\leq C\|\mathcal{H}((u,v),s)\|_{\frac{8(\tau+1)}{8-\mu}}^{2(\tau+1)}+C\Big(\int_{\mathbb{R}^4}|\mathcal{H}((u,v),s)|^{\frac{8(q+1)t'}{8-\mu}}dx\Big)^{\frac{8-\mu}{4t'}}\Big(\int_{\mathbb{R}^4}(e^{\frac{8\alpha t|\mathcal{H}((u,v),s)|^2}{8-\mu}}-1)dx\Big)^{\frac{8-\mu}{4t}}\nonumber\\
  &\leq C\|\mathcal{H}((u,v),s)\|_{\frac{8(\tau+1)}{8-\mu}}^{2(\tau+1)}+C\|\mathcal{H}((u,v),s)\|_{\frac{8(q+1)t'}{8-\mu}}^{2(q+1)}\Big(\int_{\mathbb{R}^4}(e^{\frac{8\alpha t\|\Delta\mathcal{H}((u,v),s)\|_2^2}{8-\mu}(\frac{|\mathcal{H}((u,v),s)|}{\|\Delta\mathcal{H}((u,v),s)\|_2})^2}-1)dx\Big)^{\frac{8-\mu}{4t}}\nonumber\\
  &\leq C\|\mathcal{H}((u,v),s)\|_{\frac{8(\tau+1)}{8-\mu}}^{2(\tau+1)}+C\|\mathcal{H}((u,v),s)\|_{\frac{8(q+1)t'}{8-\mu}}^{2(q+1)}\nonumber\\
  &=Ce^{(4\tau+\mu-4)s}\|(u,v)\|_{\frac{8(\tau+1)}{8-\mu}}^{2(\tau+1)}+Ce^{(4q+4-\frac{8-\mu}{t'})s}\|(u,v)\|_{\frac{8(q+1)t'}{8-\mu}}^{2(q+1)}\nonumber\\
  &=Ce^{(4\tau+\mu-4)s}\Big(\|u\|_{\frac{8(\tau+1)}{8-\mu}}^{\frac{8(\tau+1)}{8-\mu}}+\|v\|_{\frac{8(\tau+1)}{8-\mu}}^{\frac{8(\tau+1)}{8-\mu}}\Big)^{\frac{8-\mu}{4}}
  +Ce^{(4q+4-\frac{8-\mu}{t'})s}\Big(\|u\|_{\frac{8(q+1)t'}{8-\mu}}^{\frac{8(q+1)t'}{8-\mu}}+\|v\|_{\frac{8(q+1)t'}{8-\mu}}^{\frac{8(q+1)t'}{8-\mu}}\Big)^{\frac{8-\mu}{4t'}},
\end{align}
for all $s\leq s_1$.
Since $\tau>2-\frac{\mu}{4}$, $q>2$ and $t'$ large enough, it follows from \eqref{close01} that
\begin{align*}
  \mathcal{J}(\mathcal{H}((u,v),s))=& \frac{1}{2}e^{4s}(\|\Delta u\|_2^2+\|\Delta v\|_2^2)-\int_{\mathbb{R}^4}(I_\mu*F(\mathcal{H}((u,v),s)))F(\mathcal{H}((u,v),s))dx
  \rightarrow 0^+,
\end{align*}
as $s\rightarrow-\infty$.

For any fixed $s>>0$,
set
\begin{equation*}
 \mathcal{M}(t)=\frac{1}{2}\int_{\mathbb{R}^4}(I_\mu*F(tu,tv))F(tu,tv)dx,\quad\text{for $t>0$}.
\end{equation*}
It follows from $(F_2)$ that
\begin{equation*}
 \frac{\frac{d\mathcal{M}(t)}{dt}}{\mathcal{M}(t)}>\frac{2\theta}{t},\quad\text{for $t>0$}.
\end{equation*}
Thus, integrating this over $[1,e^{2s}]$, we have
\begin{equation}\label{imp}
  \frac{1}{2}\int_{\mathbb{R}^4}(I_\mu*F(e^{2s}u,e^{2s}v))F(e^{2s}u,e^{2s}v)dx\geq \frac{e^{4\theta s}}{2}\int_{\mathbb{R}^4}(I_\mu*F(u,v))F(u,v)dx.
\end{equation}
Therefore,
\begin{align*}
\mathcal J(\mathcal{H}((u,v),s))
\leq&
\frac{ e^{4s}}{2}\int_{\mathbb{R}^4}(|\Delta u|^2+|\Delta v|^2)dx-\frac{e^{(4\theta+\mu-8) s}}{2}\int_{\mathbb{R}^4}(I_\mu*F(u,v)F(u,v)dx.
\end{align*}
Since $\theta>3-\frac{\mu}{4}$, the above inequality yields that $\mathcal{J}(\mathcal{H}((u,v),s))\rightarrow-\infty$ as $s\rightarrow+\infty$.
\end{proof}

\begin{lemma}\label{equi0}
Assume that $F$ satisfies $(F_1)$, $(F_2)$, $(F_{5})$ and $F_{z_j}$($j=1,2$) has exponential subcritical growth at $\infty$, then for any fixed $(u,v)\in \mathcal{S}$,
the following statements hold.

$(i)$ The function ${\mathcal J}(\mathcal{H}((u,v),s))$ achieves its maximum with positive level at a unique point $s_{(u,v)}\in \mathbb{R}$ such that $\mathcal{H}((u,v),s_{(u,v)}) \in \mathcal{P}(a,b)$.

$(ii)$ The mapping $(u,v)\mapsto s_{(u,v)}$ is continuous in $(u,v)\in \mathcal{S}$.
\end{lemma}
\begin{proof}
$(i)$
By Lemma \ref{minimax01}, we have
\begin{equation*}
  \lim\limits_{s\rightarrow-\infty}\mathcal{J}(\mathcal{H}((u,v),s))=0^+\quad\text{and}\lim\limits_{s\rightarrow+\infty}\mathcal{J}(\mathcal{H}((u,v),s))=-\infty.
\end{equation*}
Therefore, there exists $s_{(u,v)}\in \mathbb{R}$ such that $P(\mathcal{H}({(u,v)},s_{(u,v)}))=\frac{1}{2}\frac{d}{ds}\mathcal{J}(\mathcal{H}((u,v),s))|_{s=s_{(u,v)}}=0$, and $\mathcal{J}(\mathcal{H}((u,v),s_{(u,v)}))>0$.
In the following, we prove the uniqueness of $s_{(u,v)}$ for any $(u,v)\in \mathcal{S}$.

Taking into account that $\frac{d}{ds}\mathcal{J}(\mathcal{H}((u,v),s))|_{s=s_{(u,v)}}=0$, using $(F_{5})$, we deduce that
\begin{align*}
  &\frac{d^2}{ds^2}\mathcal{J}(\mathcal{H}((u,v),s))\Big|_{s=s_{(u,v)}}\\
  &=
8 e^{4s_{(u,v)}}\int_{\mathbb{R}^4}(|\Delta u|^2+|\Delta v|^2)dx-\frac{(8-\mu)^2}{2}e^{(\mu-8)s_{(u,v)}}\int_{\mathbb{R}^4}(I_\mu*F(e^{2s_{(u,v)}}u,e^{2s_{(u,v)}}v))F(e^{2s_{(u,v)}}u,e^{2s_{(u,v)}}v)dx
\\&\quad+(28-4\mu)e^{(\mu-8)s_{(u,v)}}\int_{\mathbb{R}^4}(I_\mu*F(e^{2s_{(u,v)}}u,e^{2s_{(u,v)}}v))\\
&\quad\times\Big[(e^{2s_{(u,v)}}u,e^{2s_{(u,v)}}v)\cdot (\frac{\partial F(e^{2s_{(u,v)}}u,e^{2s_{(u,v)}}v) }{\partial(e^{2s_{(u,v)}}u)},\frac{\partial F(e^{2s_{(u,v)}}u,e^{2s_{(u,v)}}v) }{\partial(e^{2s_{(u,v)}}v)} )\Big] dx\\
&\quad-4e^{(\mu-8)s_{(u,v)}}\int_{\mathbb{R}^4}\bigg(I_\mu* \Big[(e^{2s_{(u,v)}}u,e^{2s_{(u,v)}}v)\cdot (\frac{\partial F(e^{2s_{(u,v)}}u,e^{2s_{(u,v)}}v) }{\partial(e^{2s_{(u,v)}}u)},\frac{\partial F(e^{2s_{(u,v)}}u,e^{2s_{(u,v)}}v) }{\partial(e^{2s_{(u,v)}}v)} )\Big]\bigg)\\&\quad\times\Big[(e^{2s_{(u,v)}}u,e^{2s_{(u,v)}}v)\cdot (\frac{\partial F(e^{2s_{(u,v)}}u,e^{2s_{(u,v)}}v) }{\partial(e^{2s_{(u,v)}}u)},\frac{\partial F(e^{2s_{(u,v)}}u,e^{2s_{(u,v)}}v) }{\partial(e^{2s_{(u,v)}}v)} )\Big]dx\\&\quad
-4e^{(\mu-8)s_{(u,v)}}\int_{\mathbb{R}^4}(I_\mu*F(e^{2s_{(u,v)}}u,e^{2s_{(u,v)}}v))\Big[(e^{2s_{(u,v)}}u)^2\frac{\partial^2 F(e^{2s_{(u,v)}}u,e^{2s_{(u,v)}}v)}{\partial (e^{2s_{(u,v)}}u)^2}\\
&\quad+(e^{2s_{(u,v)}}v)^2\frac{\partial^2 F(e^{2s_{(u,v)}}u,e^{2s_{(u,v)}}v)}{\partial (e^{2s_{(u,v)}}v)^2}\Big]dx\\
&=-(8-\mu)(6-\frac{\mu}{2})\int_{\mathbb{R}^4}(I_\mu*F(\mathcal{H}((u,v),s_{(u,v)})))F(\mathcal{H}((u,v),s_{(u,v)}))dx\\
&\quad+(36-4\mu)\int_{\mathbb{R}^4}(I_\mu*F(\mathcal{H}((u,v),s_{(u,v)}))) \Big[\mathcal{H}((u,v),s_{(u,v)})\cdot (\frac{\partial F(\mathcal{H}((u,v),s_{(u,v)}))}{\partial (\mathcal{H}(u,s_{(u,v)}))},\frac{\partial F(\mathcal{H}((u,v),s_{(u,v)}))}{\partial (\mathcal{H}(v,s_{(u,v)}))})\Big]dx\\
&\quad-4\int_{\mathbb{R}^4}\bigg(I_\mu* \Big[\mathcal{H}((u,v),s_{(u,v)})\cdot (\frac{\partial F(\mathcal{H}((u,v),s_{(u,v)}))}{\partial (\mathcal{H}(u,s_{(u,v)}))},\frac{\partial F(\mathcal{H}((u,v),s_{(u,v)}))}{\partial (\mathcal{H}(v,s_{(u,v)}))})\Big]\bigg)\\
&\quad\times\Big[\mathcal{H}((u,v),s_{(u,v)})\cdot (\frac{\partial F(\mathcal{H}((u,v),s_{(u,v)}))}{\partial (\mathcal{H}(u,s_{(u,v)}))},\frac{\partial F(\mathcal{H}((u,v),s_{(u,v)}))}{\partial (\mathcal{H}(v,s_{(u,v)}))})\Big]dx\\&\quad
-4\int_{\mathbb{R}^4}(I_\mu*F(\mathcal{H}((u,v),s_{(u,v)})))\Big[(\mathcal{H}(u,s_{(u,v)}))^2\frac{\partial^2 F(\mathcal{H}((u,v),s_{(u,v)}))}{\partial (\mathcal{H}(u,s_{(u,v)}))^2}+(\mathcal{H}(v,s_{(u,v)}))^2\frac{\partial^2 F(\mathcal{H}((u,v),s_{(u,v)}))}{\partial (\mathcal{H}(u,s_{(u,v)}))^2}\Big]dx \\
&=4\int_{\mathbb{R}^4}\int_{\mathbb{R}^4}\frac{A}{|x-y|^\mu}dxdy<0,
\end{align*}
where
\begin{align*}
  A=&(3-\frac{\mu}{4})F(\mathcal{H}((u(y),v(y)),s_{(u(y),v(y))}))\mathfrak{F}(\mathcal{H}((u(x),v(x)),s_{(u(x),v(x))}))\\&
  -F(\mathcal{H}((u(y),v(y)),s_{(u(y),v(y))}))\\
  &\times\Big[\mathcal{H}((u(x),v(x)),s_{(u(x),v(x))})\cdot (\frac{\partial \mathfrak{F}(\mathcal{H}((u(x),v(x)),s_{(u(x),v(x))}))}{\partial (\mathcal{H}(u(x),s_{(u(x),v(x))}))},\frac{\partial \mathfrak{F}(\mathcal{H}((u(x),v(x)),s_{(u(x),v(x))}))}{\partial (\mathcal{H}(v(x),s_{(u(x),v(x))}))})\Big]\\&
  -\mathfrak{F}(\mathcal{H}((u(y),v(y)),s_{(u(y),v(y))}))[\mathfrak{F}(\mathcal{H}((u(x),v(x)),s_{(u(x),v(x))}))-F(\mathcal{H}((u(x),v(x)),s_{(u(x),v(x))}))]<0,
\end{align*}
this prove the uniqueness of $s_{(u,v)}$.

$(ii)$ By $(i)$, the mapping $(u,v)\mapsto s_{(u,v)}$ is well defined. Let $\{(u_n,v_n)\}\subset \mathcal{S}$ be any sequence such that $(u_n,v_n)\rightarrow (u,v)$ in $\mathcal{X}$ as $n\rightarrow\infty$. We only need to prove that up to a subsequence, $s_{(u_n,v_n)}\rightarrow s_{(u,v)}$ in $\mathbb{R}$ as $n\rightarrow\infty$.

We first show that $\{s_{(u_n,v_n)}\}$ is bounded in $\mathbb{R}$. If up to a subsequence, $s_{(u_n,v_n)}\rightarrow+\infty$ as $n\rightarrow\infty$, then by \eqref{imp}, $(F_2)$ and $(u_n,v_n)\rightarrow (u,v)\neq(0,0)$ in $\mathcal{X}$ as $n\rightarrow\infty$, we have
\begin{align*}
  0&\leq \lim\limits_{n\rightarrow\infty}e^{-4s_{(u_n,v_n)}}\mathcal J(\mathcal{H}((u_n,v_n),s_{(u_n,v_n)}))\\
  &\leq \lim\limits_{n\rightarrow\infty}\frac{1}{2}\Big[\int_{\mathbb{R}^4}(|\Delta u_n|^2+|\Delta v_n|^2)dx-e^{(4\theta+\mu-12)s_{(u_n,v_n)}}\int_{\mathbb{R}^4}(I_\mu*F(u_n,v_n))F(u_n,v_n)dx\Big]=-\infty,
\end{align*}
which is a contradiction. Therefore, $\{s_{(u_n,v_n)}\}$ is bounded from above. On the other hand, by Lemma \ref{biancon}, $\mathcal{H}((u_n,v_n),s_{(u,v)})\rightarrow \mathcal{H}((u,v),s_{(u,v)})$ in $\mathcal{X}$ as $n\rightarrow\infty$, it follows from $(i)$ that
\begin{equation*}
  \mathcal J(\mathcal{H}((u_n,v_n),s_{(u_n,v_n)}))\geq \mathcal J(\mathcal{H}((u_n,v_n),s_{(u,v)}))=\mathcal J(\mathcal{H}((u,v),s_{(u,v)}))+o_n(1)
\end{equation*}
and thus
\begin{equation*}
  \liminf\limits_{n\rightarrow\infty}\mathcal J(\mathcal{H}((u_n,v_n),s_{(u_n,v_n)}))\geq \mathcal J(\mathcal{H}((u,v),s_{(u,v)}))>0.
\end{equation*}
If up to a subsequence, $s_{(u_n,v_n)}\rightarrow-\infty$ as $n\rightarrow\infty$,
using $(F_2)$,
 we get
\begin{equation*}
  \mathcal J(\mathcal{H}((u_n,v_n),s_{(u_n,v_n)}))\leq \frac{ e^{4s_{(u_n,v_n)}}}{2}(\|\Delta u_n\|_2^2+ \|\Delta v_n\|_2^2)
  \rightarrow0, \,\,\text{as}\,\,\, n\rightarrow\infty,
\end{equation*}
which is an absurd. Therefore, $\{s_{(u_n,v_n)}\}$ is bounded from below.
Up to a subsequence, we assume that $s_{(u_n,v_n)}\rightarrow s_*$ as $n\rightarrow\infty$. Recalling that $(u_n,v_n)\rightarrow (u,v)$ in $\mathcal{X}$ as $n\rightarrow\infty$, then $\mathcal{H}((u_n,v_n),s_{(u_n,v_n)})\rightarrow \mathcal{H}((u,v),s_{*})$ in $\mathcal{X}$ as $n\rightarrow\infty$. Since $P(\mathcal{H}((u_n,v_n),s_{(u_n,v_n)}))=0$ for any $n\in \mathbb{N}^+$, it follows that $P(\mathcal{H}((u,v),s_{*}))=0$. By the uniqueness of $s_{(u,v)}$, we get $s_{(u,v)}=s_*$, thus $(ii)$ is proved.
\end{proof}

\begin{lemma}\label{continuous0}
Assume that $F$ satisfies $(F_1)$, $(F_2)$ and $F_{z_j}$($j=1,2$) has exponential subcritical growth at $\infty$, then there exists $\delta>0$ small enough such that
\begin{equation*}
  \mathcal  J(u,v)
\geq\frac{1}{4}(\|\Delta u\|_2^2+\|\Delta v\|_2^2)
\end{equation*}
and
\begin{equation*}
  P(u,v)
\geq\|\Delta u\|_2^2+\|\Delta v\|_2^2
\end{equation*}
for all $(u,v)\in \mathcal{S}$ satisfying $\|\Delta u\|_2^2+\|\Delta v\|_2^2\leq \delta$.
\end{lemma}
\begin{proof}
Arguing as \eqref{close01}, for any $\alpha\rightarrow0^+$ and $t\rightarrow1^+$, using \eqref{GNinequality}, we have
$\frac{8\alpha t\|\Delta(u,v)\|_2^2}{8-\mu}\leq 32\pi^2$ and
\begin{align*}
  &\int_{\mathbb{R}^4}(I_\mu*F(u,v))F(u,v)dx\\
 &\leq
  C\Big(\|u\|_{\frac{8(\tau+1)}{8-\mu}}^{\frac{8(\tau+1)}{8-\mu}}+\|v\|_{\frac{8(\tau+1)}{8-\mu}}^{\frac{8(\tau+1)}{8-\mu}}\Big)^{\frac{8-\mu}{4}}+C\Big(\|u\|_{\frac{8(q+1)t'}{8-\mu}}^{\frac{8(q+1)t'}{8-\mu}}+\|v\|_{\frac{8(q+1)t'}{8-\mu}}^{\frac{8(q+1)t'}{8-\mu}}\Big)^{\frac{8-\mu}{4t'}}\\
  &\leq C\Big(a^2\|\Delta u\|_2^{\frac{8\tau-8+2\mu}{8-\mu}}+b^2\|\Delta v\|_2^{\frac{8\tau-8+2\mu}{8-\mu}}\Big)^{\frac{8-\mu}{4}}\\
  &\quad+
  C\Big(a^2\|\Delta u\|_2^{\frac{8(q+1)t'-16+2\mu}{8-\mu}}+b^2\|\Delta v\|_2^{\frac{8(q+1)t'-16+2\mu}{8-\mu}}\Big)^{\frac{8-\mu}{4t'}}\\
  &\leq  C(\|\Delta u\|_2^{\frac{8\tau-8+2\mu}{4}}+\|\Delta v\|_2^{\frac{8\tau-8+2\mu}{4}})+C(\|\Delta u\|_2^{\frac{8(q+1)t'-16+2\mu}{4t'}}+\|\Delta v\|_2^{\frac{8(q+1)t'-16+2\mu}{4t'}})
  \\ &\leq\Big [C\delta^{\tau+\frac{\mu}{4}-2}+C \delta^{q-\frac{8-\mu}{4t'}}\Big]\|\Delta u\|_2^2+\Big [C\delta^{\tau+\frac{\mu}{4}-2}+C  \delta^{q-\frac{8-\mu}{4t'}}\Big]\|\Delta v\|_2^2.
\end{align*}
Since $\tau>2-\frac{\mu}{4}$, $q>2$ and $t'=\frac{t}{t-1}$ large enough, choosing $\delta>0$ small enough, we conclude the result. Similarly, we can prove the other.
\end{proof}

\begin{lemma}\label{inf0}
Assume that $F$ satisfies $(F_1)$, $(F_2)$, $(F_{5})$ and $F_{z_j}$($j=1,2$) has exponential subcritical growth at $\infty$, then we have
\begin{equation*}
\inf\limits_{(u,v)\in \mathcal{P}(a,b)}(\|\Delta u\|_2^2+\|\Delta v\|_2^2)>0\quad \text{and} \quad E(a,b)>0.
\end{equation*}
\end{lemma}
\begin{proof}
By Lemma \ref{equi0}, we obtain $\mathcal{P}(a,b)\neq\emptyset$. Suppose that there exists a sequence $\{(u_n,v_n)\}\subset \mathcal{P}(a,b)$ such that $(\|\Delta u_n\|_2^2+\|\Delta v_n\|_2^2)\rightarrow0$ as $n\rightarrow\infty$, then by Lemma \ref{continuous0}, up to a subsequence,
\begin{equation*}
  0=P(u_n,v_n)\geq \int_{\mathbb{R}^4}(|\Delta u_n|^2+|\Delta v_n|^2)dx\geq0,
\end{equation*}
which implies that $\int_{\mathbb{R}^4}|\Delta u_n|^2dx=\int_{\mathbb{R}^4}|\Delta v_n|^2dx=0$ for any $n\in \mathbb{N}^+$. Hence, by $(F_2)$ and $P(u_n,v_n)=0$, we have
\begin{align*}
  0&=\int_{\mathbb{R}^4}(I_\mu*F(u_n,v_n))\Big(\frac{8-\mu}{4}F(u_n,v_n)-(u_n,v_n)\cdot \nabla F(u_n,v_n)\Big)dx\\
  &\leq \int_{\mathbb{R}^4}(I_\mu*F(u_n,v_n))(\frac{8-\mu}{4\theta}-1)[(u_n,v_n)\cdot \nabla F(u_n,v_n)]dx\leq0.
\end{align*}
So $u_n,v_n\rightarrow0$ a.e. in $\mathbb{R}^4$, which is contradict to $a,b>0$.

For any $(u,v)\in \mathcal{P}(a,b)$, by Lemma \ref{equi0}, we have
\begin{equation*}
  \mathcal{J}(u,v)=\mathcal{J}(\mathcal{H}((u,v),0))\geq \mathcal{J}(\mathcal{H}((u,v),s))\quad\text{for all $s\in \mathbb{R}$}.
\end{equation*}
Let $\delta>0$ be the number given by Lemma \ref{continuous0} and $4s:=\ln\frac{\delta}{(\|\Delta u\|_2^2+\|\Delta v\|_2^2)}$. Then $e^{4s}(\|\Delta u\|_2^2+\|\Delta v\|_2^2)=\delta$, by Lemma \ref{continuous0}, we deduce that
\begin{equation*}
  \mathcal{J}(u,v)\geq \mathcal{J}(\mathcal{H}((u,v),s))\geq \frac{ e^{4s}}{4}(\|\Delta u\|_2^2+\|\Delta v\|_2^2)=\frac{ \delta}{4}>0.
\end{equation*}
By the arbitrariness of $(u,v)\in \mathcal{P}(a,b)$, we derive the conclusion.
\end{proof}
Denote $\mathcal{X}_r=H^2_{rad}(\mathbb{R}^4)\times H^2_{rad}(\mathbb{R}^4)$,  $\mathcal{S}_r=\mathcal{S}\cap\mathcal{X}_r$, $\mathcal{P}_r(a,b)=\mathcal{P}(a,b)\cap \mathcal{X}_r$, $\displaystyle S(c)=\Big\{u\in H^2(\mathbb{R}^4):\int_{\mathbb{R}^4}|u|^2dx=c\Big\}$, $S_r(c)=S(c)\cap H^2_{rad}(\mathbb{R}^4)$ for any $c>0$, and the radial ground state energy
\begin{equation}\label{rgse}
 E_r(a,b)=\inf\limits_{(u,v)\in \mathcal{P}_r(a,b)}\mathcal J(u,v).
\end{equation}
Here $H^2_{rad}(\mathbb{R}^4)$ denotes the space of radially symmetric functions in $H^2(\mathbb{R}^4)$.
The next lemma is focused on showing the continuity of $E_r(a,b)$ in the exponential subcritical case.
\begin{lemma}\label{continuous0'}
Assume that $F$ satisfies $(F_1)$, $(F_2)$, $(F_{5})$ and $F_{z_j}$($j=1,2$) has exponential subcritical growth at $\infty$, then the function $(a,b)\mapsto E_{r}(a,b)$ is continuous.
\end{lemma}
\begin{proof}
For any given $a,b>0$, let $\{a_n\}$, $\{b_n\}$ be any positive sequence such that $a_n\rightarrow a$ and $b_n\rightarrow b$ as $n\rightarrow\infty$, we only need to prove that $E_{r}(a_n,b_n)\rightarrow E_{r}(a,b)$ as $n\rightarrow\infty$.

On the one hand, for any $(u,v)\in \mathcal{P}_{r}(a,b)$, then $$(\frac{a_n}{a}u,\frac{b_n}{b}v)\in S_{r}(a_n)\times S_{r}(b_n)$$ for $n\in \mathbb{N}^+$. Since $(\frac{a_n}{a}u,\frac{b_n}{b}v)\rightarrow (u,v)$ in $\mathcal{X}_{r}$ as $n\rightarrow\infty$, by Lemmas  \ref{biancon} and \ref{equi0}, we have $s_{(\frac{a_n}{a}u,\frac{b_n}{b}v)}\rightarrow s_{(u,v)}=0$ in $\mathbb{R}$ as $n\rightarrow\infty$ and
\begin{equation*}
  \mathcal{H}((\frac{a_n}{a}u,\frac{b_n}{b}v),s_{(\frac{a_n}{a}u,\frac{b_n}{b}v)})\rightarrow \mathcal{H}((u,v),s_{(u,v)})=(u,v) \quad\text{in $\mathcal{X}_r$ as $n\rightarrow\infty$}.
\end{equation*}
As a consequence, $\limsup\limits_{n\rightarrow\infty}E_{r}(a_n,b_n)\leq \limsup\limits_{n\rightarrow\infty} \mathcal J(\mathcal{H}((\frac{a_n}{a}u,\frac{b_n}{b}v),s_{(\frac{a_n}{a}u,\frac{b_n}{b}v)}))=\mathcal J(u,v)$. By the arbitrariness of $(u,v)\in \mathcal{P}_{r}(a,b)$, we get $\limsup \limits_{n\rightarrow\infty}E_{r}(a_n,b_n)\leq E_{r}(a,b)$.

On the other hand, for any $n\in \mathbb{N}^+$, by the definition of $E_{r}(a_n,b_n)$, there exists $(\hat{u}_n,\hat{v}_n)\in \mathcal{P}_r(a_n,b_n)$ such that
\begin{equation*}
  \mathcal J(\hat{u}_n,\hat{v}_n)\leq E_{r}(a_n,b_n)+\frac{1}{n},
\end{equation*}
thus $\limsup\limits_{n\rightarrow\infty}\mathcal J(\hat{u}_n,\hat{v}_n)\leq E_{r}(a,b)$. This with $P(\hat{u}_n,\hat{v}_n)=0$ and $(F_2)$ yields that $\{(\hat{u}_n,\hat{v}_n)\}$ is bounded in $\mathcal{X}_r$.
Setting
\begin{equation*}
  s_n:=\Big(\frac{a}{a_n}\Big)^{\frac{1}{2}},\quad t_n:=\Big(\frac{b}{b_n}\Big)^{\frac{1}{2}},
  \quad(\tilde{u}_n,\tilde{v}_n):=(\hat{u}_n(\frac{x}{s_n}),\hat{v}_n(\frac{x}{t_n}))\in \mathcal{S}_r,
\end{equation*}
then by Lemma \ref{equi0}, we have
\begin{align*}
   E_r(a,b)&\leq \mathcal J(\mathcal{H}((\tilde{u}_n,\tilde{v}_n),s_{(\tilde{u}_n,\tilde{v}_n)}))\\&\leq \mathcal J(\mathcal{H}((\hat{u}_n,\hat{v}_n),s_{(\tilde{u}_n,\tilde{v}_n)}))+\Big|\mathcal J(\mathcal{H}((\tilde{u}_n,\tilde{v}_n),s_{(\tilde{u}_n,\tilde{v}_n)}))-\mathcal J(\mathcal{H}((\hat{u}_n,\hat{v}_n),s_{(\tilde{u}_n,\tilde{v}_n)}))\Big|\\
   &\leq \mathcal J((\hat{u}_n,\hat{v}_n))+\Big|\mathcal J(\mathcal{H}((\tilde{u}_n,\tilde{v}_n),s_{(\tilde{u}_n,\tilde{v}_n)}))-\mathcal J(\mathcal{H}((\hat{u}_n,\hat{v}_n),s_{(\tilde{u}_n,\tilde{v}_n)}))\Big|\\
   &\leq E_r(a_n,b_n)+\frac{1}{n}+\Big|\mathcal J(\mathcal{H}((\tilde{u}_n,\tilde{v}_n),s_{(\tilde{u}_n,\tilde{v}_n)}))-\mathcal J(\mathcal{H}((\hat{u}_n,\hat{v}_n),s_{(\tilde{u}_n,\tilde{v}_n)}))\Big|\\
   &=:E_r(a_n,b_n)+\frac{1}{n}+A(n).
\end{align*}
We complete the proof if we can prove that $\limsup\limits_{n\rightarrow\infty} A(n)=0$. Indeed,
\begin{align*}
  A(n)
& =\Big|\frac{1}{2}\int_{\mathbb{R}^4}(I_\mu*F(\mathcal{H}((\hat{u}_n,\hat{v}_n),s_{(\tilde{u}_n,\tilde{v}_n)})))F(\mathcal{H}((\hat{u}_n,\hat{v}_n),s_{(\tilde{u}_n,\tilde{v}_n)}))dx\\
&\quad-\frac{1}{2}\int_{\mathbb{R}^4}(I_\mu*F(\mathcal{H}((\tilde{u}_n,\tilde{v}_n),s_{(\tilde{u}_n,\tilde{v}_n)})))F(\mathcal{H}((\tilde{u}_n,\tilde{v}_n),s_{(\tilde{u}_n,\tilde{v}_n)}))dx \Big|.
\end{align*}
We claim that
$\limsup\limits_{n\rightarrow\infty} s_{(\tilde{u}_n,\tilde{v}_n)}<+\infty$, otherwise, $\limsup\limits_{n\rightarrow\infty} s_{(\tilde{u}_n,\tilde{v}_n)}=+\infty$.

First, we prove that: up to a subsequence and up to translations in $\mathbb{R}^4$, there exists $(\tilde{u},\tilde{v})\in \mathcal{X}_r$ such that $(\tilde{u}_n,\tilde{v}_n)\rightharpoonup (\tilde{u},\tilde{v})\neq(0,0)$ in $\mathcal{X}_r$.
Since $s_n,t_n\rightarrow1$ as $n\rightarrow\infty$, and $\{(\hat{u}_n,\hat{v}_n)\}$ is bounded in $\mathcal{X}_r$, we know $\{(\tilde{u}_n,\tilde{v}_n)\}$ is bounded in $\mathcal{X}_r$, up to a subsequence, we assume that $(\hat{u}_n,\hat{v}_n)\rightharpoonup (\hat{u},\hat{v})$ and $(\tilde{u}_n,\tilde{v}_n)\rightharpoonup (\tilde{u},\tilde{v})$ in $\mathcal{X}_r$. By the uniqueness of limit, we have $(\hat{u},\hat{v})=(\tilde{u},\tilde{v})$.
For any $r>0$, let
\begin{equation*}
  \varrho:=\limsup\limits_{n\rightarrow\infty}\Big(\sup\limits_{y\in \mathbb{R}^4}\int_{B(y,r)}(|\tilde{u}_n|^2+|\tilde{v}_n|^2)dx\Big).
\end{equation*}

If $\varrho>0$, there exists $\{y_n\}\subset \mathbb{R}^4$  such that $\int_{B(y_n,1)}(|\tilde{u}_n|^2+|\tilde{v}_n|^2)dx>\frac{\varrho}{2}$, i.e., $\int_{B(0,1)}(|\tilde{u}_n(x-y_n)|^2+|\tilde{v}_n(x-y_n)|^2)dx>\frac{\varrho}{2}$. Up to a subsequence, and up to translations in $\mathbb{R}^4$, $(\tilde{u}_n,\tilde{v}_n)\rightharpoonup (\tilde{u},\tilde{v})\neq(0,0)$ in $\mathcal{X}_r$.

If $\varrho=0$, using the Lions lemma \cite[Lemma 1.21]{Wil}, $\tilde{u}_n,\tilde{v}_n\rightarrow0$ in $L^p(\mathbb{R}^4)$ for any $p>2$. As a consequence, arguing as \eqref{close01},
 for any $\alpha\rightarrow 0^+$, $t\rightarrow1^+$, by $\tau>2-\frac{\mu}{4}>1$, $q>2$ and $t'=\frac{t}{t-1}$ large enough, we have $\frac{8\alpha t \|\Delta(\hat{u}_n,\hat{v}_n)\|_2^2}{8-\mu}\leq 32 \pi^2$ and
 \begin{align*}
 & \int_{\mathbb{R}^4}(I_\mu*F(\hat{u}_n,\hat{v}_n)[(\hat{u}_n,\hat{v}_n)\cdot \nabla F(\hat{u}_n,\hat{v}_n) ]dx\\
\leq &
  C\Big(\|\hat{u}_n\|_{\frac{8(\tau+1)}{8-\mu}}^{\frac{8(\tau+1)}{8-\mu}}+\|\hat{v}_n\|_{\frac{8(\tau+1)}{8-\mu}}^{\frac{8(\tau+1)}{8-\mu}}\Big)^{\frac{8-\mu}{4}}
  +C\Big(\|\hat{u}_n\|_{\frac{8(q+1)t'}{8-\mu}}^{\frac{8(q+1)t'}{8-\mu}}+\|\hat{v}_n\|_{\frac{8(q+1)t'}{8-\mu}}^{\frac{8(q+1)t'}{8-\mu}}\Big)^{\frac{8-\mu}{4t'}}\\
= & C\Big(s_n^{-4}\|\tilde{u}_n\|_{\frac{8(\tau+1)}{8-\mu}}^{\frac{8(\tau+1)}{8-\mu}}+t_n^{-4}\|\tilde{v}_n\|_{\frac{8(\tau+1)}{8-\mu}}^{\frac{8(\tau+1)}{8-\mu}}\Big)^{\frac{8-\mu}{4}}
+C\Big(s_n^{-4}\|\tilde{u}_n\|_{\frac{8(q+1)t'}{8-\mu}}^{\frac{8(q+1)t'}{8-\mu}}+t_n^{-4}\|\tilde{v}_n\|_{\frac{8(q+1)t'}{8-\mu}}^{\frac{8(q+1)t'}{8-\mu}}\Big)^{\frac{8-\mu}{4t'}}\rightarrow0,
\end{align*}
as $n\rightarrow\infty$.
From the above equality and $P(\hat{u}_n,\hat{v}_n)=0$, using $(F_2)$, we deduce that $\|\Delta \hat{u}_n\|_2^2+\|\Delta \hat{v}_n\|_2^2\rightarrow 0$ as $n\rightarrow\infty$. In view of Lemma \ref{continuous0}, up to a subsequence, we obtain
\begin{equation*}
  0=P(\hat{u}_n,\hat{v}_n)\geq \int_{\mathbb{R}^4}(|\Delta \hat{u}_n|^2+|\Delta \hat{v}_n|^2)dx\geq0,
\end{equation*}
which implies that $\int_{\mathbb{R}^4}|\Delta \hat{u}_n|^2dx=\int_{\mathbb{R}^4}|\Delta \hat{v}_n|^2dx=0$ for any $n\in \mathbb{N}^+$. Hence, by $(F_2)$ and $P(\hat{u}_n,\hat{v}_n)=0$, we have
\begin{align*}
  0&=\int_{\mathbb{R}^4}(I_\mu*F(\hat{u}_n,\hat{v}_n))\Big(\frac{8-\mu}{4}F(\hat{u}_n,\hat{v}_n)-(\hat{u}_n,\hat{v}_n)\cdot \nabla F(\hat{u}_n,\hat{v}_n) \Big)dx\\&\leq \int_{\mathbb{R}^4}(I_\mu*F(\hat{u}_n,\hat{v}_n))(\frac{8-\mu}{4\theta}-1)[(\hat{u}_n,\hat{v}_n)\cdot \nabla F(\hat{u}_n,\hat{v}_n)] dx\leq0,
\end{align*}
So $\hat{u}_n,\hat{v}_n\rightarrow0$ a.e. in $\mathbb{R}^4$, which is contradict to $a_n,b_n>0$.

If $\limsup\limits_{n\rightarrow\infty} s_{(\tilde{u}_n,\tilde{v}_n)}=+\infty$, using \eqref{imp}, $(F_2)$ and $(\tilde{u}_n,\tilde{v}_n)\rightarrow(\tilde{u},\tilde{v})\neq(0,0)$ a.e. in $\mathbb{R}^4$, we have 
\begin{align*}
  0&\leq \liminf\limits_{n\rightarrow\infty}e^{-4s_{(\tilde{u}_n,\tilde{v}_n)}}\mathcal J(\mathcal{H}((\tilde{u}_n,\tilde{v}_n),s_{(\tilde{u}_n,\tilde{v}_n)}))\\
  &\leq \liminf\limits_{n\rightarrow\infty}\frac{1}{2}\Big[\int_{\mathbb{R}^4}(|\Delta \tilde{u}_n|^2+|\Delta \tilde{v}_n|^2)dx-e^{(4\theta+\mu-12)s_{(\tilde{u}_n,\tilde{v}_n)}}\int_{\mathbb{R}^4}(I_\mu*F(\tilde{u}_n,\tilde{v}_n)F(\tilde{u}_n,\tilde{v}_n)dx\Big]=-\infty,
\end{align*}
which is an absurd. Hence, $\limsup\limits_{n\rightarrow\infty} s_{(\tilde{u}_n,\tilde{v}_n)}<+\infty$, up to a subsequence, we assume that $\sup\limits_{n\in \mathbb{N}^+} s_{(\tilde{u}_n,\tilde{v}_n)}<+\infty$.

According to $\{(\hat{u}_n,\hat{v}_n)\}$ and $\{(\tilde{u}_n,\tilde{v}_n)\}$ are bounded in $\mathcal{X}_r$, using $\sup\limits_{n\in \mathbb{N}^+} s_{(\tilde{u}_n,\tilde{v}_n)}<+\infty$, we obtain $$\sup\limits_{n\in \mathbb{N}^+} \|\Delta\mathcal{H}((\hat{u}_n,\hat{v}_n),s_{(\tilde{u}_n,\tilde{v}_n)})\|_2^2<+\infty,\quad\sup\limits_{n\in \mathbb{N}^+} \|\Delta\mathcal{H}((\tilde{u}_n,\tilde{v}_n),s_{(\tilde{u}_n,\tilde{v}_n)})\|_2^2<+\infty.$$  Fix $\alpha\rightarrow 0^+$ and $t\rightarrow1^+$ such that
\begin{equation*}
  \sup\limits_{n\in \mathbb{N}^+} \alpha t\|\Delta(\hat{u}_n,\hat{v}_n)\|_2^2\leq 32\pi^2,\quad \sup\limits_{n\in \mathbb{N}^+} \alpha t\|\Delta(\tilde{u}_n,\tilde{v}_n)\|_2^2\leq 32\pi^2,
\end{equation*}
and
\begin{equation*}
   \sup\limits_{n\in \mathbb{N}^+} \frac{4\alpha t\|\Delta\mathcal{H}((\hat{u}_n,\hat{v}_n),s_{(\tilde{u}_n,\tilde{v}_n)})\|_2^2}{4-\mu}\leq 32\pi^2,\quad \sup\limits_{n\in \mathbb{N}^+} \frac{4\alpha t\|\Delta\mathcal{H}((\tilde{u}_n,\tilde{v}_n),s_{(\tilde{u}_n,\tilde{v}_n)})\|_2^2}{4-\mu}\leq 32\pi^2.
\end{equation*}
With a similar proof of Lemma \ref{strong0}, we can prove that, for any $\alpha\rightarrow 0^+$ and $t\rightarrow1^+$,
\begin{equation*}
  (e^{\alpha |(\hat{u}_n,\hat{v}_n)|^2}-1)\rightharpoonup (e^{\alpha |(\tilde{u},\tilde{v})|^2}-1),\quad (e^{\alpha |(\tilde{u}_n,\tilde{v}_n)|^2}-1)\rightharpoonup (e^{\alpha |(\tilde{u},\tilde{v})|^2}-1),\quad\text{in $L^t(\mathbb{R}^4)$}
\end{equation*}
and
\begin{equation*}
  \|(I_\mu*F(\mathcal{H}((\hat{u}_n,\hat{v}_n),s_{(\tilde{u}_n,\tilde{v}_n)})))\|_\infty\leq C,\quad\|(I_\mu*F(\mathcal{H}((\tilde{u}_n,\tilde{v}_n),s_{(\tilde{u}_n,\tilde{v}_n)})))\|_\infty\leq C.
\end{equation*}
Hence,
\begin{align*}
  &(I_\mu*F(\mathcal{H}((\hat{u}_n,\hat{v}_n),s_{(\tilde{u}_n,\tilde{v}_n)})))F(\mathcal{H}((\hat{u}_n,\hat{v}_n),s_{(\tilde{u}_n,\tilde{v}_n)}))\\
  &-(I_\mu*F(\mathcal{H}((\tilde{u}_n,\tilde{v}_n),s_{(\tilde{u}_n,\tilde{v}_n)})))F(\mathcal{H}((\tilde{u}_n,\tilde{v}_n),s_{(\tilde{u}_n,\tilde{v}_n)}))\rightarrow 0\quad \text{a.e. in $\mathbb{R}^4$},
\end{align*}
\begin{align*}
  &|(I_\mu*F(\mathcal{H}((\hat{u}_n,\hat{v}_n),s_{(\tilde{u}_n,\tilde{v}_n)})))F(\mathcal{H}((\hat{u}_n,\hat{v}_n),s_{(\tilde{u}_n,\tilde{v}_n)}))\\
-&(I_\mu*F(\mathcal{H}((\tilde{u}_n,\tilde{v}_n),s_{(\tilde{u}_n,\tilde{v}_n)})))F(\mathcal{H}((\tilde{u}_n,\tilde{v}_n),s_{(\tilde{u}_n,\tilde{v}_n)}))|
\\ \leq& C(|(\hat{u}_n,\hat{v}_n)|^{\tau+1}+|(\tilde{u}_n,\tilde{v}_n)|^{\tau+1})+C[|(\hat{u}_n,\hat{v}_n)|^{q+1}(e^{\alpha |(\hat{u}_n,\hat{v}_n)|^2}-1)+|(\tilde{u}_n,\tilde{v}_n)|^{q+1}(e^{\alpha |(\tilde{u}_n,\tilde{v}_n)|^2}-1)]
\end{align*}
for all $n\in \mathbb{N}^+$,
and
\begin{align*}
  &|(\hat{u}_n,\hat{v}_n)|^{\tau+1}+|(\tilde{u}_n,\tilde{v}_n)|^{\tau+1}+|(\hat{u}_n,\hat{v}_n)|^{q+1}(e^{\alpha |(\hat{u}_n,\hat{v}_n)|^2}-1)+|(\tilde{u}_n,\tilde{v}_n)|^{q+1}(e^{\alpha |(\tilde{u}_n,\tilde{v}_n)|^2}-1)\\
\rightarrow& 2|(\tilde{u},\tilde{v})|^{\tau+1}+2|(\tilde{u},\tilde{v})|^{q+1}(e^{\alpha |(\tilde{u},\tilde{v})|^2}-1)\quad\text{a.e. in $\mathbb{R}^4$}.
\end{align*}
In the following, we prove that
\begin{equation}\label{e3.10}
  \int_{\mathbb{R}^4}(|(\hat{u}_n,\hat{v}_n)|^{\tau+1}+|(\tilde{u}_n,\tilde{v}_n)|^{\tau+1})dx\rightarrow 2\int_{\mathbb{R}^4}|(\tilde{u},\tilde{v})|^{\tau+1}dx,\,\,\,\text{as}\,\,\, n\rightarrow\infty
\end{equation}
and
\begin{align}\label{e3.11}
  &\int_{\mathbb{R}^4}[|(\hat{u}_n,\hat{v}_n)|^{q+1}(e^{\alpha |(\hat{u}_n,\hat{v}_n)|^2}-1)+|(\tilde{u}_n,\tilde{v}_n)|^{q+1}(e^{\alpha |(\tilde{u}_n,\tilde{v}_n)|^2}-1)]dx\nonumber\\
  &\rightarrow 2\int_{\mathbb{R}^4}|(\tilde{u},\tilde{v})|^{q+1}(e^{\alpha |(\tilde{u},\tilde{v})|^2}-1)dx,\,\,\,\text{as}\,\,\, n\rightarrow\infty.
\end{align}
 Using the compact embedding $H^2_{rad}(\mathbb{R}^4)\hookrightarrow L^p(\mathbb{R}^4)$ for any $p>2$, arguing as \eqref{lp}, we obtain \eqref{e3.10} and $|(\hat{u}_n,\hat{v}_n)|^{q+1} \rightarrow |(\tilde{u},\tilde{v})|^{q+1}$, $|(\tilde{u}_n,\tilde{v}_n)|^{q+1} \rightarrow |(\tilde{u},\tilde{v})|^{q+1}$ in $L^{t'}(\mathbb{R}^4)$, where $t'=\frac{t}{t-1}$. By the definition of weak convergence, using the H\"{o}lder inequality, we infer that
\begin{align*}
  &\Big|\int_{\mathbb{R}^4}[|(\hat{u}_n,\hat{v}_n)|^{q+1}(e^{\alpha |(\hat{u}_n,\hat{v}_n)|^2}-1)+|(\tilde{u}_n,\tilde{v}_n)|^{q+1}(e^{\alpha |(\tilde{u}_n,\tilde{v}_n)|^2}-1)]dx-\int_{\mathbb{R}^4}2|(\tilde{u},\tilde{v})|^{q+1}(e^{\alpha |(\tilde{u},\tilde{v})|^2}-1)dx\Big|\\
  &\leq \Big|\int_{\mathbb{R}^4}|(\hat{u}_n,\hat{v}_n)|^{q+1}(e^{\alpha |(\hat{u}_n,\hat{v}_n)|^2}-1)dx-\int_{\mathbb{R}^4}|(\tilde{u},\tilde{v})|^{q+1}(e^{\alpha |(\tilde{u},\tilde{v})|^2}-1)dx\Big|\\
  &\quad+\Big|\int_{\mathbb{R}^4}|(\tilde{u}_n,\tilde{v}_n)|^{q+1}(e^{\alpha |(\tilde{u}_n,\tilde{v}_n)|^2}-1)dx-\int_{\mathbb{R}^4}|(\tilde{u},\tilde{v})|^{q+1}(e^{\alpha |(\tilde{u},\tilde{v})|^2}-1)dx\Big|\\
&\leq\int_{\mathbb{R}^4}\Big||(\hat{u}_n,\hat{v}_n)|^{q+1}-|(\tilde{u},\tilde{v})|^{q+1}\Big|(e^{\alpha |(\hat{u}_n,\hat{v}_n)|^2}-1)dx+\int_{\mathbb{R}^4}|(\tilde{u},\tilde{v})|^{q+1}\Big|(e^{\alpha |(\hat{u}_n,\hat{v}_n)|^2}-1)-(e^{\alpha |(\tilde{u},\tilde{v})|^2}-1)\Big|dx\\
&\quad+\int_{\mathbb{R}^4}\Big||(\tilde{u}_n,\tilde{v}_n)|^{q+1}-|(\tilde{u},\tilde{v})|^{q+1}\Big|(e^{\alpha |(\tilde{u}_n,\tilde{v}_n)|^2}-1)dx+\int_{\mathbb{R}^4}|(\tilde{u},\tilde{v})|^{q+1}\Big|(e^{\alpha |(\tilde{u}_n,\tilde{v}_n)|^2}-1)-(e^{\alpha |(\tilde{u},\tilde{v})|^2}-1)\Big|dx\\
&\leq  \Big(\int_{\mathbb{R}^4}\Big||(\hat{u}_n,\hat{v}_n)|^{q+1}-|(\tilde{u},\tilde{v})|^{q+1}\Big|^{t'}dx\Big)^{\frac{1}{t'}}\Big(\int_{\mathbb{R}^4}(e^{\alpha |(\hat{u}_n,\hat{v}_n)|^2}-1)^tdx\Big)^{\frac{1}{t}}
 \\&\quad +\int_{\mathbb{R}^4}|(\tilde{u},\tilde{v})|^{q+1}\Big|(e^{\alpha |(\hat{u}_n,\hat{v}_n)|^2}-1)-(e^{\alpha |(\tilde{u},\tilde{v})|^2}-1)\Big|dx\\
  &\quad+\Big(\int_{\mathbb{R}^4}\Big||(\tilde{u}_n,\tilde{v}_n)|^{q+1}-|(\tilde{u},\tilde{v})|^{q+1}\Big|^{t'}dx\Big)^{\frac{1}{t'}}\Big(\int_{\mathbb{R}^4}(e^{\alpha |(\tilde{u}_n,\tilde{v}_n)|^2}-1)^tdx\Big)^{\frac{1}{t}}
 \\&\quad +\int_{\mathbb{R}^4}|(\tilde{u},\tilde{v})|^{q+1}\Big|(e^{\alpha |(\tilde{u}_n,\tilde{v}_n)|^2}-1)-(e^{\alpha |(\tilde{u},\tilde{v})|^2}-1)\Big|dx\rightarrow 0,\quad\text{as $n\rightarrow\infty$}.
\end{align*}
 Applying a variant of the Lebesgue dominated convergence theorem, we get $\limsup\limits_{n\rightarrow\infty} A(n)=0$.
Thus $E_r(a_n,b_n)\rightarrow E_r(a,b)$ as $n\rightarrow\infty$.
\end{proof}

\begin{lemma}\label{nonincreasing0}
Assume that $F$ satisfies $(F_1)$, $(F_2)$, $(F_{5})$ and $F_{z_j}$($j=1,2$) has exponential subcritical growth at $\infty$, then the functions $a\mapsto E(a,b)$ and $b\mapsto E(a,b)$ are non-increasing.
\end{lemma}
\begin{proof}
For any given $a,b>0$, if $\tilde{a}>a$ and $\tilde{b}>b$, we shall prove that $E(\tilde{a},b)\leq E(a,b)$ and $E(a,\tilde{b})\leq E(a,b)$. By the definition of $E(a,b)$, for any $\varepsilon>0$, there exists $(u,v)\in \mathcal{P}(a,b)$ such that
 \begin{equation}\label{nonin00}
 \mathcal J(u,v)\leq E(a,b)+\frac{\varepsilon}{3}.
\end{equation}
Consider a cut-off function $\varphi\in C_0^\infty(\mathbb{R}^4,[0,1])$ such that $\varphi(x)=1$ if $|x|\leq1$, $\varphi(x)=0$ if $|x|\geq2$. For any small $\delta>0$, define $u_\delta(x)=\varphi(\delta x)u(x)\in H^2(\mathbb{R}^4) \backslash\{0\}$, then $(u_\delta(x),v(x))\rightarrow (u,v)$ in $\mathcal{X}$ as $\delta\rightarrow0^+$. By Lemmas \ref{biancon} and \ref{equi0}, we have $s_{(u_\delta,v)}\rightarrow s_{(u,v)}=0$ in $\mathbb{R}$ as $\delta\rightarrow0^+$ and
\begin{equation*}
  \mathcal{H}((u_\delta,v),s_{(u_\delta,v)})\rightarrow \mathcal{H}((u,v),s_{(u,v)})=(u,v) \quad\text{in $\mathcal{X}$ as $\delta\rightarrow0^+$}.
\end{equation*}
Fix $\delta_0>0$ small enough such that
\begin{equation}\label{nonin01}
  \mathcal J(\mathcal{H}((u_{\delta_0},v),s_{(u_{\delta_0},v)}))\leq \mathcal J(u,v)+\frac{\varepsilon}{3}.
\end{equation}
Let $\psi(x)\in C_0^\infty(\mathbb{R}^4)$ satisfy $supp(\psi)\subset B_{1+\frac{4}{\delta_0}}(0)\backslash B_{\frac{4}{\delta_0}}(0)$, and set
\begin{equation*}
  \hat{u}_{\delta_0}=\frac{\tilde{a}^2-\|u_{\delta_0}\|_2^2}{\|\psi\|_2^2}\psi.
\end{equation*}
Define $\hat{u}_\lambda=u_{\delta_0}+\mathcal{H}(\hat{u}_{\delta_0},\lambda)$ for any $\lambda<0$.
Since $$dist(u_{\delta_0},\mathcal{H}(\hat{u}_{\delta_0},\lambda))\geq \frac{2}{\delta_0}>0
,$$ we have $\|\hat{u}_\lambda\|_2^2=\tilde{a}^2$, i.e., $(\hat{u}_\lambda,v) \in S(\tilde{a})\times S(b)$.

We claim that $s_{(\hat{u}_\lambda,v)}$ is bounded from above as $\lambda\rightarrow-\infty$. Otherwise, by \eqref{imp}, $(F_2)$ and $(\hat{u}_\lambda,v)\rightarrow (u_{\delta_0},v)\neq(0,0)$ a.e. in $\mathbb{R}^4$ as $\lambda\rightarrow-\infty$, we have
\begin{align*}
  0&\leq \lim\limits_{n\rightarrow\infty}e^{-4s_{(\hat{u}_\lambda,v)}}\mathcal J(\mathcal{H}((\hat{u}_\lambda,v),s_{(\hat{u}_\lambda,v)}))\\
  &\leq \lim\limits_{n\rightarrow\infty}\frac{1}{2}\Big[\int_{\mathbb{R}^4}(|\Delta \hat{u}_\lambda|^2+|\Delta v|^2)dx-e^{(4\theta+\mu-12)s_{(\hat{u}_\lambda,v)}}\int_{\mathbb{R}^4}(I_\mu*F(\hat{u}_\lambda,v))F(\hat{u}_\lambda,v)dx\Big]=-\infty,
\end{align*}
which is a contradiction. Thus $s_{(\hat{u}_\lambda,v)}+\lambda\rightarrow-\infty$ as $\lambda\rightarrow-\infty$, by $(F_2)$, we get
\begin{equation}\label{nonin02}
  \mathcal J(\mathcal{H}((\hat{u}_{\delta_0},0),s_{(\hat{u}_\lambda,v)}+\lambda))\leq \frac{ e^{4(s_{(\hat{u}_\lambda,v)}+\lambda)}}{2}\|\Delta \hat{u}_{\delta_0}\|_2^2
  \rightarrow0, \quad\text{as}\,\,\, \lambda\rightarrow-\infty.
\end{equation}
Now, using Lemma \ref{equi0}, $\eqref{nonin00}-\eqref{nonin02}$, we obtain
\begin{align*}
  E(\tilde{a},b)\leq \mathcal J(\mathcal{H}((\hat{u}_\lambda,v),s_{(\hat{u}_\lambda,v)}))&=\mathcal J(\mathcal{H}((u_{\delta_0},v),s_{(\hat{u}_\lambda,v)}))+\mathcal J(\mathcal{H}((\mathcal{H}(\hat{u}_{\delta_0},0),\lambda),s_{(\hat{u}_\lambda,v)}))\\
  &=\mathcal J(\mathcal{H}((u_{\delta_0},v),s_{(\hat{u}_\lambda,v)}))+\mathcal J(\mathcal{H}((\hat{u}_{\delta_0},0),s_{(\hat{u}_\lambda,v)}+\lambda))\\
  &\leq\mathcal J(\mathcal{H}((u_{\delta_0},v),s_{(u_{\delta_0},v)}))+\mathcal J(\mathcal{H}((\hat{u}_{\delta_0},0),s_{(\hat{u}_\lambda,v)}+\lambda))\leq E(a,b)+\varepsilon.
\end{align*}
By the arbitrariness of $\varepsilon>0$, we deduce that $E(\tilde{a},b)\leq E(a,b)$. Similarly, we prove $E(a,\tilde{b})\leq E(a,b)$.
\end{proof}
\begin{lemma}\label{close to0}
Assume that $F$ satisfies $(F_1)$, $(F_2)$, $(F_{5})$ and $F_{z_j}$($j=1,2$) has exponential subcritical growth at $\infty$. Suppose that \eqref{abs} possesses a ground state solution $(u,v)$ with  $\lambda_1,\lambda_2<0$, then $E(a',b)<E(a,b)$ for any $a'>a$ close to $a$ and $E(a,b')<E(a,b)$ for any $b'>b$ close to $b$.
\end{lemma}

\begin{proof}
For any $t>0$ and $s\in \mathbb{R}$, we know $\mathcal{H}((tu,v),s)\in S(ta)\times S(b)$ and
\begin{align*}
\mathcal J(\mathcal{H}((tu,v),s))=&\frac{  e^{4s}}{2}\int_{\mathbb{R}^4}(t^2|\Delta u|^2+|\Delta v|^2)dx-\frac{e^{(\mu-8)s}}{2}\int_{\mathbb{R}^4}(I_\mu*F(te^{2s}u,e^{2s}v))F(te^{2s}u,e^{2s}v)dx.
\end{align*}
Denote $\alpha(t,s):=\mathcal J(\mathcal{H}((tu,v),s))$, then
\begin{align*}
  \frac{\partial \alpha(t,s) }{\partial t}=& t e^{4s}\int_{\mathbb{R}^4}|\Delta u|^2dx-e^{(\mu-8)s}\int_{\mathbb{R}^4}(I_\mu*F(te^{2s}u,e^{2s}v))\frac{\partial F(te^{2s}u,e^{2s}v)}{\partial (te^{2s}u) } e^{2s}udx:=\frac{B}{t},
\end{align*}
where
\begin{align*}
  B=&\langle\mathcal{J}'(\mathcal{H}((tu,v),s)),\mathcal{H}((tu,v),s)\rangle- e^{4s}\int_{\mathbb{R}^4}|\Delta v|^2dx-e^{(\mu-8)s}\int_{\mathbb{R}^4}(I_\mu*F(te^{2s}u,e^{2s}v))\frac{\partial F(te^{2s}u,e^{2s}v)}{\partial (e^{2s}v) } e^{2s}vdx
\end{align*}
By Lemma \ref{biancon}, $\mathcal{H}((tu,v),s)\rightarrow (u,v)$ in $\mathcal{X}$ as $(t,s)\rightarrow (1,0)$, and since $\lambda_1<0$, we have
\begin{equation*}
  \langle\mathcal{J}'(u,v),(u,v)\rangle- \int_{\mathbb{R}^4}|\Delta v|^2dx-\int_{\mathbb{R}^4}(I_\mu*F(u,v))F_v(u,v) vdx=\lambda_1 \|u\|_2^2=\lambda_1 a^2<0.
\end{equation*}
Hence, one can fix $\delta>0$ small enough such that
\begin{equation*}
   \frac{\partial \alpha(t,s) }{\partial t}<0\quad\text{for any $(t,s)\in (1,1+\delta]\times [-\delta,\delta]$}.
\end{equation*}
For any $t\in (1,1+\delta]$ and $s\in [-\delta,\delta]$, using the mean value theorem, we obtain
\begin{equation*}
  \alpha(t,s)=\alpha(1,s)+(t-1) \frac{\partial \alpha(t,s) }{\partial t}\Big|_{t=\xi}<\alpha(1,s)
\end{equation*}
for some $\xi\in (1,t)$. By Lemma \ref{equi0}, $s_{(tu,v)}\rightarrow s_{(u,v)}=0$  in $\mathbb{R}$ as $t \rightarrow1^+$. For any $a'>a$ close to $a$, let $t_0=\frac{a'}{a}$, then
\begin{equation*}
  t_0\in (1,1+\delta],\quad s_{(t_0u,v)}\in [-\delta,\delta]
\end{equation*}
and thus, using Lemma \ref{equi0} again,
\begin{equation*}
  E(a',b)\leq \alpha(t_0,s_{(t_0u,v)})<\alpha(1,s_{(t_0u,v)})=\mathcal J(\mathcal{H}((u,v),s_{(t_0u,v)}))\leq \mathcal J(u,v)=E(a,b).
\end{equation*}
Analogously, we can prove that $E(a,b')<E(a,b)$ for any $b'>b$ close to $b$.
\end{proof}
From Lemmas \ref{nonincreasing0} and \ref{close to0}, we directly obtain
\begin{lemma}\label{conclusion0}
Assume that $F$ satisfies $(F_1)$, $(F_2)$, $(F_{5})$ and $F_{z_j}$($j=1,2$) has exponential subcritical growth at $\infty$. Suppose that \eqref{abs} possesses a ground state solution with  $\lambda_1,\lambda_2<0$,
then $E(a',b)<E(a,b)$ for any $a'>a$ and $E(a,b')<E(a,b)$ for any $b'>b$.
\end{lemma}
\begin{lemma}
Assume that $F$ satisfies $(F_1)$, $(F_2)$, $(F_{5})$ and $F_{z_j}$($j=1,2$) has exponential subcritical growth at $\infty$, then $$\lim\limits_{(a,b)\rightarrow(0^+,0^+)}E(a,b)=+\infty.$$
\end{lemma}
\begin{proof}
It is sufficient to prove that for any $\{(u_n,v_n)\}\in \mathcal{X}\backslash (0,0)$ such that $P(u_n,v_n)=0$ and $\lim\limits_{n\rightarrow\infty}(\|u_n\|_2^2+\|v_n\|_2^2)=0$, one has $\lim\limits_{n\rightarrow\infty}\mathcal J(u_n,v_n)=+\infty$.
Setting
 \begin{equation*}
4s_n:=\ln(\|\Delta u_n\|_2^2+\|\Delta v_n\|_2^2)\quad\text{ and} \quad
(\hat{u}_n,\hat{v}_n):=\mathcal{H}((u_n,v_n),-s_n),
\end{equation*}
then $\|\Delta \hat{u}_n\|_2^2+\|\Delta \hat{v}_n\|_2^2=1$ and $\|\hat{u}_n\|_2^2+\|\hat{v}_n\|_2^2=\|u_n\|_2^2+\|v_n\|_2^2$.
 Since $\lim\limits_{n\rightarrow\infty}(\|u_n\|_2^2+\|v_n\|_2^2)=0$, we have
\begin{equation*}
  \limsup\limits_{n\rightarrow\infty}\Big(\sup\limits_{y\in \mathbb{R}^4}\int_{B(y,r)}(|\hat{u}_n|^2+|\hat{v}_n|^2)dx\Big)=0,
\end{equation*}
for any $r>0$. Using the Lions lemma \cite[Lemma 1.21]{Wil}, $\hat{u}_n,\hat{v}_n\rightarrow0$ in $L^p(\mathbb{R}^4)$ for any $p>2$. As a consequence, arguing as \eqref{close01}, for any fixed $s>0$, as $\alpha\rightarrow0^+$ and $t\rightarrow 1^+$, by $\tau>2-\frac{\mu}{4}>1$, $q>2$, and $t'=\frac{t}{t-1}$ large enough, we have $\frac{8\alpha t\|\Delta\mathcal{H}((\hat{u}_n,\hat{v}_n),s)\|_2^2}{8-\mu}\leq 32\pi^2$ and
 \begin{align*}
  &\int_{\mathbb{R}^4}(I_\mu*F(\mathcal{H}((\hat{u}_n,\hat{v}_n),s)))F(\mathcal{H}((\hat{u}_n,\hat{v}_n),s))dx\\&\leq
  Ce^{(4\tau+\mu-4)s}\Big(\|\hat{u}_n\|_{\frac{8(\tau+1)}{8-\mu}}^{\frac{8(\tau+1)}{8-\mu}}+\|\hat{v}_n\|_{\frac{8(\tau+1)}{8-\mu}}^{\frac{8(\tau+1)}{8-\mu}}\Big)^{\frac{8-\mu}{4}}+Ce^{(4q+4-\frac{8-\mu}{t'})s}
  \Big(\|\hat{u}_n\|_{\frac{8(q+1)t'}{8-\mu}}^{\frac{8(q+1)t'}{8-\mu}}+\|\hat{v}_n\|_{\frac{8(q+1)t'}{8-\mu}}^{\frac{8(q+1)t'}{8-\mu}}\Big)^{\frac{8-\mu}{4t'}}
\end{align*}
Combing this and $P(\mathcal{H}((\hat{u}_n,\hat{v}_n),s_n))=P(u_n,v_n)=0$, using Lemma \ref{equi0}, we derive that
\begin{equation*}
  \mathcal J(u_n,v_n)=\mathcal J(\mathcal{H}((\hat{u}_n,\hat{v}_n),s_n))\geq \mathcal J(\mathcal{H}(\mathcal{H}((\hat{u}_n,\hat{v}_n),s_n),s))=
  \mathcal J(\mathcal{H}((\hat{u}_n,\hat{v}_n),s))\geq  e^{4s}+o_n(1).
\end{equation*}
By the arbitrariness of $s\in \mathbb{R}$, we deduce that $\lim\limits_{n\rightarrow\infty}\mathcal J(u_n,v_n)=+\infty$.
\end{proof}

\subsection{Palais-Smale sequence}\label{appro0}
In this section, using the minimax principle based on the homotopy stable family of compact subsets of $\mathcal{S}$, we will construct a $(PS)_{E(a,b)}$ sequence on $\mathcal{P}(a,b)$ for $\mathcal J$.
\begin{proposition}\label{pssequencehomo0}
Assume that $F$ satisfies $(F_1)$, $(F_2)$, $(F_{5})$ and $F_{z_j}$($j=1,2$) has exponential subcritical growth at $\infty$, then there exists a $(PS)_{E(a,b)}$ sequence $\{(u_n,v_n)\}\subset \mathcal{P}(a,b)$ for $\mathcal J$.
\end{proposition}
Following  by \cite{Wil}, we recall that the tangent space of $\mathcal{S}$ at $(u,v)$ is defined by
\begin{equation*}
  T_{(u,v)}\mathcal{S}=\Big\{(\varphi,\psi)\in \mathcal{X}:\int_{\mathbb{R}^4} (u\varphi+v\psi) dx=0\Big\}.
\end{equation*}
To prove Proposition \ref{pssequencehomo0}, we borrow some arguments from \cite{BS1,BS2} and consider the functional $\mathcal{S}\rightarrow\mathbb{R}$ defined by
\begin{equation*}
 \mathcal{I}(u,v):=\mathcal J(\mathcal{H}((u,v),s_{(u,v)})),
\end{equation*}
where $s_{(u,v)}\in \mathbb{R}$ is the unique number obtained in Lemma \ref{equi0} for any $(u,v)\in \mathcal{S}$. By Lemma \ref{equi0}, we know that $s_{(u,v)}$ is continuous as a mapping of $(u,v)\in \mathcal{S}$. However, it remains unknown that whether $s_{(u,v)}$ is of class $C^1$. Inspired by \cite[Proposition 2.9]{SW}, we observe that
\begin{lemma}\label{C0^1}
Assume that $F$ satisfies $(F_1)$, $(F_2)$, $(F_{5})$ and $F_{z_j}$($j=1,2$) has exponential subcritical growth at $\infty$, then $\mathcal{I}:\mathcal{S}\rightarrow\mathbb{R}$ is of class $C^1$ and
\begin{align*}
  \langle \mathcal{I}'(u,v),(\varphi,\psi)\rangle=& e^{4s_{(u,v)}}\int_{\mathbb{R}^4}(\Delta u\Delta \varphi+\Delta v\Delta\psi) dx-e^{(\mu-8)s_{(u,v)}}\int_{\mathbb{R}^4}(I_\mu*F(e^{2s_{(u,v)}}u,e^{2s_{(u,v)}}v))\\
  &\times\Big[(e^{2s_{(u,v)}}\varphi,e^{2s_{(u,v)}}\psi)\cdot(\frac{\partial F(e^{2s_{(u,v)}}u,e^{2s_{(u,v)}}v)}{\partial (e^{2s_{(u,v)}}u)},\frac{\partial F(e^{2s_{(u,v)}}u,e^{2s_{(u,v)}}v)}{\partial (e^{2s_{(u,v)}}v)}) \Big] dx\\
  =&\langle\mathcal J'(\mathcal{H}((u,v),s_{(u,v)})),\mathcal{H}((\varphi,\psi),s_{(u,v)})\rangle
\end{align*}
for any $(u,v)\in \mathcal{S}$ and $(\varphi,\psi)\in T_{(u,v)}\mathcal{S}$.
\end{lemma}
\begin{proof}
Let $(u,v)\in \mathcal{S}$ and $(\varphi,\psi)\in T_{(u,v)}\mathcal{S}$, then for any $|t|$ small enough, by Lemma \ref{equi0}, we obtain
\begin{align*}
  &\mathcal{I}(u+t\varphi,v+t\psi)-\mathcal{I}(u,v)\\=&\mathcal J(\mathcal{H}((u+t\varphi,v+t\psi),s_{(u+t\varphi,v+t\psi)}))-\mathcal J(\mathcal{H}((u,v),s_{(u,v)}))\\
\leq& \mathcal J(\mathcal{H}((u+t\varphi,v+t\psi),s_{(u+t\varphi,v+t\psi)}))-\mathcal J(\mathcal{H}((u,v),s_{(u+t\varphi,v+t\psi)}))\\
=&\frac{ e^{4s_{(u+t\varphi,v+t\psi)}}}{2}\int_{\mathbb{R}^4}\Big[|\Delta (u+t\varphi)|^2-|\Delta u|^2+|\Delta (v+t\psi)|^2-|\Delta v|^2\Big]dx\\ & -\frac{e^{(\mu-8)s_{(u+t\varphi,v+t\psi)}}}{2}\int_{\mathbb{R}^4}\Big[(I_\mu*F(e^{2s_{(u+t\varphi,v+t\psi)}}(u+t\varphi),e^{2s_{(u+t\varphi,v+t\psi)}}(v+t\psi)))\\&
  \times F(e^{2s_{(u+t\varphi,v+t\psi)}}(u+t\varphi),e^{2s_{(u+t\varphi,v+t\psi)}}(v+t\psi))\\
  &-(I_\mu*F(e^{2s_{(u+t\varphi,v+t\psi)}}u,e^{2s_{(u+t\varphi,v+t\psi)}}v))F(e^{2s_{(u+t\varphi,v+t\psi)}}u,e^{2s_{(u+t\varphi,v+t\psi)}}v)\Big] dx\\
=&\frac{ e^{4s_{(u+t\varphi,v+t\psi)}}}{2}\int_{\mathbb{R}^4}\Big(t^2|\Delta u|^2+t^2|\Delta v|^2+2t\Delta u \Delta \varphi+2t\Delta v\Delta\psi\Big)dx\\
  & -\frac{e^{(\mu-8)s_{(u+t\varphi,v+t\psi)}}}{2}\int_{\mathbb{R}^4}(I_\mu*F(e^{2s_{(u+t\varphi,v+t\psi)}}(u+t\varphi),e^{2s_{(u+t\varphi,v+t\psi)}}(v+t\psi)))\\&
  \times \Big[(e^{2s_{(u+t\varphi,v+t\psi)}}t\varphi,e^{2s_{(u+t\varphi,v+t\psi)}}t\psi)\cdot (F_{z_1}|_{z_1=e^{2s_{(u+t\varphi,v+t\psi)}}(u+\xi_tt\varphi)},F_{z_2}|_{z_2=e^{2s_{(u+t\varphi,v+t\psi)}}(v+\xi_tt\psi)} ) \Big] dx\\&  -\frac{e^{(\mu-8)s_{(u+t\varphi,v+t\psi)}}}{2}\int_{\mathbb{R}^4}(I_\mu*F(e^{2s_{(u+t\varphi,v+t\psi)}}u,e^{2s_{(u+t\varphi,v+t\psi)}}v))\\&
  \times \Big[(e^{2s_{(u+t\varphi,v+t\psi)}}t\varphi,e^{2s_{(u+t\varphi,v+t\psi)}}t\psi)\cdot (F_{z_1}|_{z_1=e^{2s_{(u+t\varphi,v+t\psi)}}(u+\xi_tt\varphi)},F_{z_2}|_{z_2=e^{2s_{(u+t\varphi,v+t\psi)}}(v+\xi_tt\psi)} ) \Big] dx,
\end{align*}
where $\xi_t\in (0,1)$. Analogously, we have
\begin{align*}
  &\mathcal{I}(u+t\varphi,v+t\psi)-\mathcal{I}(u,v)
  \\ \geq& \frac{ e^{4s_{(u,v)}}}{2}\int_{\mathbb{R}^4}\Big(t^2|\Delta u|^2+t^2|\Delta v|^2+2t\Delta u \Delta \varphi+2t\Delta v\Delta\psi\Big)dx\\
  &-\frac{e^{(\mu-8)s_{(u,v)}}}{2}\int_{\mathbb{R}^4}(I_\mu*F(e^{2s_{(u,v)}}(u+t\varphi),e^{2s_{(u,v)}}(v+t\psi)))\\&
  \times \Big[(e^{2s_{(u,v)}}t\varphi,e^{2s_{(u,v)}}t\psi)\cdot (F_{z_1}|_{z_1=e^{2s_{(u,v)}}(u+\zeta_tt\varphi)},F_{z_2}|_{z_2=e^{2s_{(u,v)}}(v+\zeta_tt\psi)} ) \Big] dx\\&-\frac{e^{(\mu-8)s_{(u,v)}}}{2}\int_{\mathbb{R}^4}(I_\mu*F(e^{2s_{(u,v)}}u,e^{2s_{(u,v)}}v))\\&
  \times \Big[(e^{2s_{(u,v)}}t\varphi,e^{2s_{(u,v)}}t\psi)\cdot (F_{z_1}|_{z_1=e^{2s_{(u,v)}}(u+\zeta_tt\varphi)},F_{z_2}|_{z_2=e^{2s_{(u,v)}}(v+\zeta_tt\psi)} ) \Big] dx,
\end{align*}
where $\zeta_t\in (0,1)$. By Lemma \ref{equi0}, $\lim\limits_{t\rightarrow0}s_{(u+t\varphi,v+t\psi)}=s_{(u,v)}$, from the above inequalities, we conclude
\begin{align*}
  &\lim\limits_{t\rightarrow0}\frac{\mathcal{I}(u+t\varphi,v+t\psi)-\mathcal{I}(u,v)}{t}\\
=& e^{4s_{(u,v)}}\int_{\mathbb{R}^4}(\Delta u\Delta \varphi+\Delta v\Delta\psi) dx-e^{(\mu-8)s_{(u,v)}}\int_{\mathbb{R}^4}(I_\mu*F(e^{2s_{(u,v)}}u,e^{2s_{(u,v)}}v))\\
  &  \times\Big[(e^{2s_{(u,v)}}\varphi,e^{2s_{(u,v)}}\psi)\cdot(\frac{\partial F(e^{2s_{(u,v)}}u,e^{2s_{(u,v)}}v)}{\partial (e^{2s_{(u,v)}}u)},\frac{\partial F(e^{2s_{(u,v)}}u,e^{2s_{(u,v)}}v)}{\partial (e^{2s_{(u,v)}}v)}) \Big] dx.
\end{align*}
Using Lemma \ref{equi0}, we find that the G\^{a}teaux derivative of $\mathcal{I}$ is continuous linear in $(\varphi,\psi)$ and continuous in $(u,v)$. Therefore, by \cite[Proposition 1.3]{Wil}, we get $\mathcal{I}$ is of class $C^1$. Changing variables in the integrals, we prove the rest.
\end{proof}

\begin{lemma}\label{pssequencehomo02}
Assume that $F$ satisfies $(F_1)$, $(F_2)$, $(F_{5})$ and $F_{z_j}$($j=1,2$) has exponential subcritical growth at $\infty$. Let $\mathcal{G}$  be a homotopy stable family of compact subsets of $\mathcal{S}$ without boundary (i.e., $B=\emptyset$) and set
\begin{equation*}
  e_{\mathcal{G}}:=\inf\limits_{A\in \mathcal{G}}\max\limits_{(u,v)\in A}\mathcal I(u,v).
\end{equation*}
If $e_{\mathcal{G}}>0$, then there exists a $(PS)_{e_{\mathcal{G}}}$ sequence $\{(u_n,v_n)\}\subset \mathcal{P}(a,b)$ for $\mathcal J$.
\end{lemma}
\begin{proof}
Let $\{A_n\}\subset \mathcal{G}$ be a minimizing sequence of $e_{\mathcal{G}}$. We define the mapping
\begin{equation*}
  \eta:[0,1]\times \mathcal{S}\rightarrow \mathcal{S}, \quad\eta(t,(u,v))=\mathcal{H}((u,v),ts_{(u,v)}).
\end{equation*}
By Lemma \ref{equi0}, $\eta(t,(u,v))$ is continuous in $[0,1]\times \mathcal{S}$, and satisfies $\eta(t,(u,v))=(u,v)$ for all $(t,(u,v))\in \{0\}\times \mathcal{S}$. Thus by the definition of $\mathcal{G}$, one has
\begin{equation*}
  D_n:=\eta(1,A_n)=\{\mathcal{H}((u,v),s_{(u,v)}):(u,v)\in A_n\}\subset \mathcal{G}.
\end{equation*}
Obviously, $D_n\subset \mathcal{P}(a,b)$ for any $n\in \mathbb{N}^+$. Since $\mathcal{I}(\mathcal{H}((u,v),s))=\mathcal{I}(u,v)$ for any $(u,v)\in \mathcal{S}$ and $s\in \mathbb{R}$, then
\begin{equation*}
  \max\limits_{(u,v)\in D_n}\mathcal{I}(u,v)=\max\limits_{(u,v)\in A_n}\mathcal{I}(u,v)\rightarrow e_{\mathcal{G}},\quad\text{as $n\rightarrow\infty$.}
\end{equation*}
which implies that $\{D_n\}\subset \mathcal{G}$ is another minimizing sequence of $e_{\mathcal{G}}$. By Lemma \ref{Ghouss}, we obtain a $(PS)_{e_{\mathcal{G}}}$ sequence $\{(\hat{u}_n,\hat{v}_n)\}\subset \mathcal{S}$ for $\mathcal I$ such that $\lim\limits_{n\rightarrow\infty}dist((\hat{u}_n,\hat{v}_n),D_n)=0$.
Let
\begin{equation*}
  (u_n,v_n):=\mathcal{H}((\hat{u}_n,\hat{v}_n),s_{(\hat{u}_n,\hat{v}_n)}),
\end{equation*}
then we prove that $\{(u_n,v_n)\}\subset \mathcal{P}(a,b)$ is the desired sequence.

We claim that there exists $C>0$ such that $e^{-4s_{(\hat{u}_n,\hat{v}_n)}}\leq C$ for any $n\in \mathbb{N}^+$. Indeed, we observe that
\begin{equation*}
  e^{-4s_{(\hat{u}_n,\hat{v}_n)}}=\frac{\|\Delta\hat{u}_n\|_2^2+\|\Delta\hat{v}_n\|_2^2}{\|\Delta u_n\|_2^2+\|\Delta v_n\|_2^2}.
\end{equation*}
Since $\{(u_n,v_n)\}\subset \mathcal{P}(a,b)$, by Lemma \ref{inf0}, there exists a constant $\widetilde{C}>0$ such that $\|\Delta u_n\|_2^2+\|\Delta v_n\|_2^2\geq \widetilde{C}$ for any $n\in \mathbb{N}^+$. Regarding the term of $\{(\hat{u}_n,\hat{v}_n)\}$, since $D_n\subset \mathcal{P}(a,b)$ for any $n\in \mathbb{N}^+$ and for any $(u,v)\in \mathcal{P}(a,b)$, one has $\mathcal{J}(u,v)=\mathcal{I}(u,v)$, thus
\begin{equation*}
   \max\limits_{(u,v)\in D_n}\mathcal{J}(u,v)= \max\limits_{(u,v)\in D_n}\mathcal{I}(u,v)\rightarrow e_{\mathcal{G}},\quad\text{as $n\rightarrow\infty$.}
\end{equation*}
This together with  $D_n\subset \mathcal{P}(a,b)$ and $(F_2)$ yields $\{D_n\}$ is uniformly bounded in $\mathcal{X}$, thus from
$$\lim\limits_{n\rightarrow\infty}dist((\hat{u}_n,\hat{v}_n),D_n)=0,
$$
we obtain $\sup\limits_{n\geq1}\|(\hat{u}_n,\hat{v}_n)\|^2<\infty$. This prove the claim.

From $\{(u_n,v_n)\}\subset \mathcal{P}(a,b)$, it follows that
\begin{equation*}
  \mathcal{J}(u_n,v_n)=\mathcal{I}(u_n,v_n)=\mathcal{I}(\hat{u}_n,\hat{v}_n)\rightarrow e_{\mathcal{G}},\quad\text{as $n\rightarrow\infty$.}
\end{equation*}
For any $(\varphi,\psi)\in T_{(u_n,v_n)}\mathcal{S}$, we have
\begin{align*}
  &\int_{\mathbb{R}^4}\Big(\hat{u}_n e^{-2s_{(\hat{u}_n,\hat{v}_n)}}\varphi(e^{-s_{(\hat{u}_n,\hat{v}_n)}}x) +\hat{v}_n e^{-2s_{(\hat{u}_n,\hat{v}_n)}}\psi(e^{-s_{(\hat{u}_n,\hat{v}_n)}}x)\Big) dx\\
=&\int_{\mathbb{R}^4}\Big(\hat{u}_n(e^{s_{(\hat{u}_n,\hat{v}_n)}}y) e^{2s_{(\hat{u}_n,\hat{v}_n)}}\varphi(y) +\hat{v}_n (e^{s_{(\hat{u}_n,\hat{v}_n)}}y) e^{2s_{(\hat{u}_n,\hat{v}_n)}}\psi(y)\Big)dy=\int_{\mathbb{R}^4}(u_n\varphi+v_n\psi)dx=0,
\end{align*}
which means $\mathcal{H}((\varphi,\psi),-s_{(\hat{u}_n,\hat{v}_n)})\in T_{(\hat{u}_n,\hat{v}_n)}\mathcal{S}$. Also, we have
\begin{align*}
 &\|(e^{-2s_{(\hat{u}_n,\hat{v}_n)}}\varphi(e^{-s_{(\hat{u}_n,\hat{v}_n)}}x) ,e^{-2s_{(\hat{u}_n,\hat{v}_n)}}\psi(e^{-s_{(\hat{u}_n,\hat{v}_n)}}x) )\|^2\\
=&e^{-4s_{(\hat{u}_n,\hat{v}_n)}}(\|\Delta \varphi\|_2^2+\|\Delta \psi\|_2^2)+2e^{-2s_{(\hat{u}_n,\hat{v}_n)}}(\|\nabla \varphi\|_2^2+\|\nabla \psi\|_2^2)+(\|\varphi\|_2^2+\|\psi\|_2^2)\\
\leq& C(\|\Delta \varphi\|_2^2+\|\Delta \psi\|_2^2)+2\sqrt{C}(\|\nabla \varphi\|_2^2+\|\nabla \psi\|_2^2)+(\|\varphi\|_2^2+\|\psi\|_2^2)\leq \max\{1,C\}\|(\varphi,\psi)\|^2.
\end{align*}
By Lemma \ref{C0^1}, for any $(\varphi,\psi)\in T_{(u_n,v_n)}\mathcal{S}$, we deduce that
\begin{align*}
  |\langle\mathcal{J}'(u_n,v_n),(\varphi,\psi)\rangle|&=|\langle\mathcal{J}'(\mathcal{H}((\hat{u}_n,\hat{v}_n),s_{(\hat{u}_n,\hat{v}_n)})),\mathcal{H}(\mathcal{H}((\varphi,\psi),-s_{(\hat{u}_n,\hat{v}_n)}),s_{(\hat{u}_n,\hat{v}_n)})\rangle|\\
  &=|\langle\mathcal{I}'(\hat{u}_n,\hat{v}_n),\mathcal{H}((\varphi,\psi),-s_{(\hat{u}_n,\hat{v}_n)})\rangle|\\
  &\leq \|\mathcal{I}'(\hat{u}_n,\hat{v}_n)\|_*\cdot \|\mathcal{H}((\varphi,\psi),-s_{(\hat{u}_n,\hat{v}_n)})\|\\&\leq \max\{1,\sqrt{C}\}\|\mathcal{I}'(\hat{u}_n,\hat{v}_n)\|_*\cdot  \|(\varphi,\psi)\|,
\end{align*}
which implies that $\|\mathcal{J}'(u_n,v_n)\|_*\leq \max\{1,\sqrt{C}\}\|\mathcal{I}'(\hat{u}_n,\hat{v}_n)\|_*\rightarrow0$ as $n\rightarrow\infty$. Thus $\{(u_n,v_n)\}\subset \mathcal{P}(a,b)$ is a $(PS)_{e_{\mathcal{G}}}$ sequence for $\mathcal J$.
\end{proof}
\noindent
{\bf Proof of Proposition \ref{pssequencehomo0}:} Noted that the class $\mathcal{G}$ of all singletons included in  $\mathcal{S}$ is a homotopy stable family of compact subsets of $\mathcal{S}$ without boundary. By Lemma \ref{pssequencehomo02}, if $e_{\mathcal{G}}>0$, then there exists a $(PS)_{e_{\mathcal{G}}}$ sequence $\{(u_n,v_n)\}\subset \mathcal{P}(a,b)$ for $\mathcal J$. By Lemma \ref{inf0}, we know $E(a,b)>0$, so if we can prove that $e_{\mathcal{G}}=E(a,b)$, then we complete the proof.

In fact, by the definition of $\mathcal{G}$, we have
\begin{equation*}
  e_{\mathcal{G}}=\inf\limits_{A\in \mathcal{G}}\max\limits_{(u,v)\in A}\mathcal I(u,v)=\inf\limits_{(u,v)\in \mathcal{S}}\mathcal I(u,v)=\inf\limits_{(u,v)\in \mathcal{S}}\mathcal I(\mathcal{H}((u,v),s_{(u,v)}))=\inf\limits_{(u,v)\in \mathcal{S}}\mathcal J(\mathcal{H}((u,v),s_{(u,v)})).
\end{equation*}
For any $(u,v)\in \mathcal{S}$, it follows from $\mathcal{H}((u,v),s_{(u,v)})\in \mathcal{P}(a,b)$ that $\mathcal J(\mathcal{H}((u,v),s_{(u,v)}))\geq E(a,b)$, by the arbitrariness of $(u,v)\in \mathcal{S}$, we get $e_{\mathcal{G}}\geq E(a,b)$. On the other hand, for any $(u,v)\in \mathcal{P}(a,b)$, by Lemma \ref{equi0}, we deduce that $s_{(u,v)}=0$ and $\mathcal{J}(u,v)=\mathcal J(\mathcal{H}((u,v),0))\geq \inf\limits_{(u,v)\in \mathcal{S}}\mathcal J(\mathcal{H}((u,v),s_{(u,v)}))$, by the arbitrariness of $(u,v)\in \mathcal{P}(a,b)$, we have $E(a,b)\geq e_{\mathcal{G}}$.
\qed

For the sequence $\{(u_n,v_n)\}$ obtained in Proposition \ref{pssequencehomo0}, by $(F_2)$, we know that $\{(u_n,v_n)\}$ is bounded in $\mathcal{X}$, up to a subsequence, we assume that $(u_n,v_n)\rightharpoonup (u_{a},v_{b})$ in $\mathcal{X}$. 
Furthermore, by ${\mathcal J}'|_{\mathcal{S}}(u_n,v_n)\rightarrow0$ as $n\rightarrow\infty$ and Lagrange multiplier rule, there exist two sequences $\{\lambda_{1,n}\}, \{\lambda_{2,n}\}\subset \mathbb{R}$ such that
\begin{align}\label{lagrange0}
&\int_{\mathbb{R}^4}(\Delta u_n \Delta \varphi+\Delta v_n\Delta \psi)dx-\int_{\mathbb{R}^4}(I_\mu*F(u_n,v_n)) F_{u_n}(u_n,v_n)\varphi dx-\int_{\mathbb{R}^4}(I_\mu*F(u_n,v_n))F_{v_n}(u_n,v_n)\psi dx \nonumber \\
=&\int_{\mathbb{R}^4}(\lambda_{1,n}u_n\varphi+\lambda_{2,n}v_n\psi)dx+o_n(1)\|(\varphi,\psi)\|
\end{align}
for every $(\varphi,\psi)\in \mathcal{X}$.
\begin{lemma}\label{not0}
Assume that $F$ satisfies $(F_1)-(F_3)$, $(F_{5})$ and $F_{z_j}$($j=1,2$) has exponential subcritical growth at $\infty$, then up to a subsequence and up to translations in $\mathbb{R}^4$, $u_a\neq0$ and $v_b\neq0$.
\end{lemma}
\begin{proof}
For any $r>0$, let
\begin{equation*}
  \varrho:=\limsup\limits_{n\rightarrow\infty}\Big(\sup\limits_{y\in \mathbb{R}^4}\int_{B(y,r)}(|u_n|^2+|v_n|^2)dx\Big).
\end{equation*}

If $\varrho>0$, then there exists $\{y_n\}\subset \mathbb{R}^4$  such that $\int_{B(y_n,1)}(|u_n|^2+|v_n|^2)dx>\frac{\varrho}{2}$, i.e., $\int_{B(0,1)}(|u_n(x-y_n)|^2+|v_n(x-y_n)|^2)dx>\frac{\varrho}{2}$. Up to a subsequence, and up to translations in $\mathbb{R}^4$, $(u_n,v_n)\rightharpoonup (u_a,v_b)\neq(0,0)$ in $\mathcal{X}$. From \eqref{lagrange0} and Lemma \ref{strong0}, we can see that $(u_a,v_b)$ is a weak solution of \eqref{e1.1} with $N=4$. Assume that $u_{a}=0$, then by $(F_2)$ and $(F_3)$, we know that $v_{b}=0$. Similarly, $v_{b}=0$ implies $u_{a}=0$. This is impossible, since $(u_a,v_b)\neq(0,0)$.

If $\varrho=0$, using the Lions lemma \cite[Lemma 1.21]{Wil}, $u_n,v_n\rightarrow0$ in $L^p(\mathbb{R}^4)$ for any $p>2$. As a consequence, arguing as \eqref{close01}, since $\{(u_n,v_n)\}$ is bounded in $\mathcal{X}$, for any $\alpha\rightarrow 0^+$ and $t\rightarrow1^+$, by $\tau>2-\frac{\mu}{4}>1$, $q>2$, and $t'=\frac{t}{t-1}$ large enough, we have $\sup\limits_{n\in \mathbb{N}^+}\frac{8 \alpha t\|\Delta (u_n,v_n)\|_2^2}{8-\mu}\leq 32\pi^2$ and
\begin{align*}
 & \int_{\mathbb{R}^4}(I_\mu*F(u_n,v_n))[(u_n,v_n)\cdot \nabla F(u_n,v_n) ] dx\\
\leq &
 C\Big(\|u_n\|_{\frac{8(\tau+1)}{8-\mu}}^{\frac{8(\tau+1)}{8-\mu}}+\|v_n\|_{\frac{8(\tau+1)}{8-\mu}}^{\frac{8(\tau+1)}{8-\mu}}\Big)^{\frac{8-\mu}{4}}+C\Big(\|u_n\|_{\frac{8(q+1)t'}{8-\mu}}^{\frac{8(q+1)t'}{8-\mu}}+\|v_n\|_{\frac{8(q+1)t'}{8-\mu}}^{\frac{8(q+1)t'}{8-\mu}}\Big)^{\frac{8-\mu}{4t'}}
 \rightarrow0,\quad\text{as $n\rightarrow\infty$}.
\end{align*}
From the above equality and $P(u_n,v_n)=0$, we deduce that $\|\Delta u_n\|_2^2+\|\Delta v_n\|_2^2\rightarrow 0$ as $n\rightarrow\infty$, hence $(F_2)$ implies $\lim\limits_{n\rightarrow\infty}{\mathcal J}(u_n,v_n)= 0$, which is an absurd, since $E(a,b)>0$.
\end{proof}

\begin{lemma}\label{lam0}
Assume that $F$ satisfies $(F_1)-(F_5)$ and $F_{z_j}$($j=1,2$) has exponential subcritical growth at $\infty$, then $\{\lambda_{1,n}\}$ and $\{\lambda_{2,n}\}$ are bounded in $\mathbb{R}$. Furthermore, up to a subsequence,
$\lambda_{1,n}\rightarrow\lambda_1<0$ and $\lambda_{2,n}\rightarrow\lambda_2<0$ in $\mathbb{R}$ as $n\rightarrow\infty$.
\end{lemma}

\begin{proof}
Using $(u_n, 0)$ and $(0, v_n)$ as test functions in \eqref{lagrange0}, we have
\begin{equation*}\label{lam1}
  \lambda_{1,n}a^2=\int_{\mathbb{R}^4}|\Delta u_n|^2 dx-\int_{\mathbb{R}^4}(I_\mu*F(u_n,v_n)) F_{u_n}(u_n,v_n)u_n dx+o_n(1)\|u_n\|
\end{equation*}
and
\begin{equation*}\label{lam2}
  \lambda_{2,n}b^2=\int_{\mathbb{R}^4}|\Delta v_n|^2 dx-\int_{\mathbb{R}^4}(I_\mu*F(u_n,v_n)) F_{v_n}(u_n,v_n)v_n dx+o_n(1)\|v_n\|.
\end{equation*}
By $P(u_n,v_n)=0$, we obtain
\begin{align*}
  \int_{\mathbb{R}^4}(I_\mu*F(u_n,v_n))[(u_n,v_n) \cdot \nabla F(u_n,v_n)] dx=&\int_{\mathbb{R}^4}(|\Delta u_n |^2+|\Delta v_n |^2 ) dx+\frac{8-\mu}{4}
  \int_{\mathbb{R}^4}(I_\mu*F(u_n,v_n))F(u_n,v_n)dx.
\end{align*}
This together with $(F_4)$ and the boundedness of $\{(u_n,v_n)\}$  yields that $\{\lambda_{1,n}\}$ and $\{\lambda_{2,n}\}$ are bounded in $\mathbb{R}$. Moreover,
\begin{align*}
  -\lambda_{1,n}=&\frac{1}{a^2}\Big(\int_{\mathbb{R}^4}(I_\mu*F(u_n,v_n))(\frac{8-\mu}{4}F(u_n,v_n)-F_{v_n}(u_n,v_n)v_n)dx+\int_{\mathbb{R}^4}|\Delta v_n |^2dx\Big)
\end{align*}
and
\begin{align*}
  -\lambda_{2,n}=&\frac{1}{b^2}\Big(\int_{\mathbb{R}^4}(I_\mu*F(u_n,v_n))(\frac{8-\mu}{4}F(u_n,v_n)-F_{u_n}(u_n,v_n)u_n)dx+\int_{\mathbb{R}^4}|\Delta u_n |^2dx\Big),
\end{align*}
Thanks to $u_a\neq0,v_b\neq0$, using $(F_2)$, $(F_4)$ and Fatou lemma, we obtain $\liminf \limits_{n\rightarrow\infty}-\lambda_{1,n}>0$ and $\liminf\limits_{n\rightarrow\infty}-\lambda_{2,n}>0$, namely, $\limsup \limits_{n\rightarrow\infty}\lambda_{1,n}<0$ and $\limsup\limits_{n\rightarrow\infty}\lambda_{2,n}<0$. Since $\{\lambda_{1,n}\}$ and $\{\lambda_{2,n}\}$ are bounded, up to a subsequence, we assume that $\lambda_{1,n}\rightarrow\lambda_1<0$ and $\lambda_{2,n}\rightarrow\lambda_2<0$ in $\mathbb{R}$ as $n\rightarrow\infty$.
\end{proof}

\subsection{Proof of Theorem \ref{th3}}
\noindent
{\it Proof of Theorem \ref{th3}:}
Under the assumptions of Theorem \ref{th3}, from \eqref{lagrange0} and Lemmas \ref{strong0}, \ref{not0}, \ref{lam0}, we know $(u_a,v_b)$ is a nontrivial weak solution of \eqref{e1.1} with $\lambda_1,\lambda_2<0$, $N=4$,
and $P(u_a,v_b)=0$.
Using the Br\'{e}zis-Lieb lemma \cite[Lemma 1.32]{Wil}, we have
\begin{equation*}
  \|u_{n}\|_2^2=\|u_n-u_{a}\|_2^2+\|u_{a}\|_2^2+o_n(1)\quad\text{and}\quad \|v_{n}\|_2^2=\|v_n-v_{b}\|_2^2+\|v_{b}\|_2^2+o_n(1).
\end{equation*}
Let $a_1:=\|u_{a}\|_2>0$, $b_1:=\|v_{b}\|_2>0$, and $a_{1,n}:=\|u_n-u_{a}\|_2$, $b_{1,n}:=\|v_n-v_{b}\|_2$, then
$a^2=a_1^2+a_{1,n}^2+o_n(1)$ and $b^2=b_1^2+b_{1,n}^2+o_n(1)$. Since $P(u_a,v_b)=0$, using $(F_2)$ and Fatou lemma, we have
\begin{align*}
  \mathcal{J}(u_a,v_b)&={\mathcal J}(u_a,v_b)-\frac{1}{2}P(u_a,v_b)\\&=\frac{1}{2}\int_{\mathbb{R}^4}(I_\mu*F(u_a,v_b))\Big[ (u_a,v_b)\cdot\nabla F(u_a,v_b )-(3-\frac{\mu}{4})F(u_a,v_b)\Big] dx\nonumber\\
  &\leq\frac{1}{2}\liminf\limits_{n\rightarrow\infty}\int_{\mathbb{R}^4}(I_\mu*F(u_n,v_n))\Big[ (u_n,v_n)\cdot \nabla F(u_n,v_n)-(3-\frac{\mu}{4})F(u_n,v_n)\Big] dx\nonumber\\
  &=\liminf\limits_{n\rightarrow\infty}({\mathcal J}(u_n,v_n)-\frac{1}{2}P(u_n,v_n))=E(a,b).
\end{align*}
On the other hand, it follows from Lemma \ref{nonincreasing0} that $\mathcal{J}(u_a,v_b)\geq E(a_1,b_1)\geq E(a,b)$. Thus $\mathcal{J}(u_a,v_b)= E(a_1,b_1)= E(a,b)$, and it follows from lemmas \ref{conclusion0}, \ref{lam0} that $\|u_{a}\|_2=a$, $\|v_{b}\|_2=b$. This implies $(u_a,v_b)$ is a ground state solution of \eqref{abs}.
\qed

\section{Exponential critical case} \label{cri}

This section is devoted to study \eqref{abs} in the exponential critical case.
\begin{lemma}\label{strong}
Assume that $F$ satisfies $(F_1)$, $(F_2)$ and $F_{z_j}$($j=1,2$) has exponential critical growth at $\infty$, let $\{(u_n,v_n)\}\subset\mathcal{S}$ be a bounded $(PS)_{E(a,b)}$  sequence of $\mathcal{J}$ in $\mathcal{X}$,
up to a subsequence, if $(u_n,v_n)\rightharpoonup (u,v)$ in $\mathcal{X}$,
and \begin{equation}\label{condition}
  \int_{\mathbb{R}^4} (I_\mu*F(u_n,v_n))[(u_n,v_n)\cdot \nabla F(u_n,v_n)] dx\leq K_0
\end{equation}
for some constant $K_0>0$,
 then for any $\varphi\in C_0^\infty(\mathbb{R}^4)$, we have
\begin{equation}\label{strong1}
  \int_{\mathbb{R}^4}\Delta u_n\Delta\varphi dx \rightarrow  \int_{\mathbb{R}^4}\Delta u\Delta\varphi dx,\quad \int_{\mathbb{R}^4}\Delta v_n\Delta\varphi dx \rightarrow  \int_{\mathbb{R}^4}\Delta v\Delta\varphi dx,\,\,\,\text{as}\,\,\, n\rightarrow\infty,
\end{equation}
\begin{equation}\label{strong3}
  \int_{\mathbb{R}^4}u_n\varphi dx \rightarrow  \int_{\mathbb{R}^4}u\varphi dx, \quad\int_{\mathbb{R}^4}v_n\varphi dx \rightarrow  \int_{\mathbb{R}^4}v\varphi dx,\,\,\,\text{as}\,\,\, n\rightarrow\infty,
\end{equation}
and
\begin{equation}\label{strong4}
  \int_{\mathbb{R}^4}(I_\mu*F(u_n,v_n))F_{u_n} (u_n,v_n)\varphi dx\rightarrow \int_{\mathbb{R}^4}(I_\mu*F(u,v))F_{u} (u,v)\varphi dx,\,\,\,\text{as}\,\,\, n\rightarrow\infty,
\end{equation}
\begin{equation}\label{strong5}
  \int_{\mathbb{R}^4}(I_\mu*F(u_n,v_n))F_{v_n} (u_n,v_n)\varphi dx\rightarrow \int_{\mathbb{R}^4}(I_\mu*F(u,v)) F_{v} (u,v)\varphi dx,\,\,\,\text{as}\,\,\, n\rightarrow\infty,
\end{equation}
\begin{equation}\label{strong6}
  \int_{\mathbb{R}^4}(I_\mu*F(u_n,v_n))F (u_n,v_n) dx\rightarrow \int_{\mathbb{R}^4}(I_\mu*F(u,v)) F (u,v) dx,\,\,\,\text{as}\,\,\, n\rightarrow\infty.
\end{equation}
\end{lemma}

\begin{proof}
With a similar proof of Lemma \ref{strong0}, we have \eqref{strong1}-\eqref{strong3} hold.

We adopt some ideas of Lemma $2.4$ in \cite{ACTY} to prove \eqref{strong4} and \eqref{strong5}. Let $\Omega$ be any compact subset of $\mathbb{R}^4$, $\Omega'\subset\subset \Omega$ and $\phi \in C_0^\infty(\Omega)$ such that $0\leq \phi\leq 1$ and $\phi=1$ in $\Omega'$, then by taking $\phi$ as a text function in the first equation of \eqref{abs} and using H\"{o}lder inequality, we get
\begin{equation*}
  \int_{\Omega'}(I_\mu*F(u_n,v_n))F_{u_n} (u_n,v_n) dx\leq \int_{\Omega}(I_\mu*F(u_n,v_n))F_{u_n} (u_n,v_n)\phi dx\leq C\|u_n\|\|\phi\|\leq C,
\end{equation*}
which implies that $\omega _n:=(I_\mu*F(u_n,v_n))F_{u_n} (u_n,v_n) $ is bounded in $L^1(\Omega)$, up to a subsequence, $\omega_k\rightarrow \omega$ in the weak*-topology as $n\rightarrow\infty$, where $\omega$ denotes a Radon measure. So for any $\varphi\in C_0^\infty(\Omega)$, we get
\begin{equation*}
  \lim\limits_{n\rightarrow\infty}\int_{\Omega}(I_\mu*F(u_n,v_n))F_{u_n} (u_n,v_n)\varphi dx=\int_{\Omega}\varphi d\omega.
\end{equation*}
Now recall that $\{(u_n,v_n)\}\subset\mathcal{S}$ is a $(PS)_{E(a,b)}$  sequence of $\mathcal{J}$, hence, for any $\varphi\in C_0^\infty(\Omega)$,
\begin{equation*}
  \lim\limits_{n\rightarrow\infty}\int_{\mathbb{R}^4}\Delta u_n \Delta\varphi-\lambda_1 u_n \varphi dx=\int_{\Omega}\varphi d\omega,
\end{equation*}
which implies that $\omega$ is absolutely continuous with respect to the Lebesgue measure. Then, by the Radon-Nicodym theorem, there exists a function $g\in L^1(\Omega)$ such that for any $\varphi\in C_0^\infty(\Omega)$,
\begin{equation*}
  \int_{\Omega}\varphi d\omega=\int_{\Omega}\varphi g dx.
\end{equation*}
Since there holds for any compact set $\Omega\subset \mathbb{R}^4$, we have that there exists a function $g\in L^1_{loc}(\mathbb{R}^4)$ such that for any $\varphi\in C_0^\infty(\mathbb{R}^4)$,
\begin{equation*}
  \lim\limits_{n\rightarrow\infty}\int_{\mathbb{R}^4}(I_\mu*F(u_n,v_n))F_{u_n} (u_n,v_n)\varphi dx=\int_{\mathbb{R}^4}\varphi d\omega=\int_{\mathbb{R}^4}(I_\mu*F(u,v))F_{u} (u,v)\varphi dx.
\end{equation*}
So we have \eqref{strong4}. Similarly, we prove \eqref{strong5}.

Next, adopting some ideas of Lemma $3.3$ in \cite{AGMS}, we will prove \eqref{strong6}. By $(F_2)$, \eqref{CS}, \eqref{condition} and Fatou lemma, we have
 \begin{align*}
   &\Big|\int_{\mathbb{R}^4}(I_\mu*F(u_n,v_n))F (u_n,v_n) dx-\int_{\mathbb{R}^4}(I_\mu*F(u,v))F (u,v) dx\Big|\\
   &\leq \Big|\int_{\mathbb{R}^4}(I_\mu*F(u_n,v_n))(F (u_n,v_n)-F(u,v)) dx\Big|+\Big|\int_{\mathbb{R}^4}(I_\mu*F(u,v))(F (u_n,v_n)-F(u,v)) dx\Big|\\
   &\leq \Big(\int_{\mathbb{R}^4}(I_\mu*F(u_n,v_n))F (u_n,v_n) dx\Big)^{\frac{1}{2}}\Big(\int_{\mathbb{R}^4}(I_\mu*(F(u_n,v_n)-F(u,v)))(F (u_n,v_n) -F(u,v))dx\Big)^{\frac{1}{2}}\\
   &\quad+\Big(\int_{\mathbb{R}^4}(I_\mu*F(u,v))F (u,v) dx\Big)^{\frac{1}{2}}\Big(\int_{\mathbb{R}^4}(I_\mu*(F(u_n,v_n)-F(u,v)))(F (u_n,v_n) -F(u,v))dx\Big)^{\frac{1}{2}}\\
   &\leq C\Big(\int_{\mathbb{R}^4}(I_\mu*(F(u_n,v_n)-F(u,v)))(F (u_n,v_n) -F(u,v))dx\Big)^{\frac{1}{2}}.
\end{align*}
Now we only need to prove that
\begin{equation*}
\lim\limits_{n\rightarrow\infty}\int_{\mathbb{R}^4}(I_\mu*(F(u_n,v_n)-F(u,v)))(F (u_n,v_n) -F(u,v))dx=0.
\end{equation*}
For a fixed $S>0$, denote
\begin{equation*}
  A=\Big\{x\in \mathbb{R}^4:|u_n(x)|\leq S\,\,\,\text{and}\,\,\,|v_n(x)|\leq S\Big\},\quad B=\Big\{x\in \mathbb{R}^4:|u(x)|\leq S\,\,\,\text{and}\,\,\,|v(x)|\leq S\Big\},
\end{equation*}
\begin{equation*}
  C=\Big\{x\in \mathbb{R}^4:|u_n(x)|\geq S\,\,\,\text{or}\,\,\,|v_n(x)|\geq S\Big\},\quad D=\Big\{x\in \mathbb{R}^4:|u(x)|\geq S\,\,\,\text{or}\,\,\,|v(x)|\geq S\Big\}.
\end{equation*}
Then
\begin{align*}
  &\int_{\mathbb{R}^4}(I_\mu*(F(u_n,v_n)-F(u,v)))(F (u_n,v_n) -F(u,v))dx\\
\leq& \int_{\mathbb{R}^4}\int_{\mathbb{R}^4}\frac{|F(u_n,v_n)\chi_A(y)-F(u,v)\chi_B(y)| |F(u_n,v_n)\chi_A(x)-F(u,v)\chi_B(x)|}{|x-y|^\mu}dxdy\\
  &\quad + 2\int_{\mathbb{R}^4}\int_{\mathbb{R}^4}\frac{(F(u_n,v_n)\chi_A(y)+F(u,v)\chi_B(y)+F(u,v)\chi_D(y))F(u_n,v_n)\chi_C(x)}{|x-y|^\mu}dxdy\\
   &\quad + 2\int_{\mathbb{R}^4}\int_{\mathbb{R}^4}\frac{(F(u_n,v_n)\chi_A(y)+F(u,v)\chi_B(y))F(u,v)\chi_D(x)}{|x-y|^\mu}dxdy\\
   &\quad + \int_{\mathbb{R}^4}\int_{\mathbb{R}^4}\frac{F(u_n,v_n)\chi_C(y)F(u_n,v_n)\chi_C(x)}{|x-y|^\mu}dxdy+ \int_{\mathbb{R}^4}\int_{\mathbb{R}^4}\frac{F(u,v)\chi_D(y)F(u,v)\chi_D(x)}{|x-y|^\mu}dxdy\\
  :=&I_1+I_2+I_3+I_4+I_5.
\end{align*}
From \eqref{HLSin}, we obtain $I_j=o(1)$ for $j=2,\cdots,5$ when $S$ is large enough. Moreover
\begin{align*}
I_1=&\int_{\Omega}\int_{\mathbb{R}^4}\frac{|F(u_n,v_n)\chi_A(y)-F(u,v)\chi_B(y)| |F(u_n,v_n)\chi_A(x)-F(u,v)\chi_B(x)|}{|x-y|^\mu}dxdy\\
  &\quad+\int_{\mathbb{R}^4\backslash\Omega}\int_{\mathbb{R}^4}\frac{|F(u_n,v_n)\chi_A(y)-F(u,v)\chi_B(y)| |F(u_n,v_n)\chi_A(x)-F(u,v)\chi_B(x)|}{|x-y|^\mu}dxdy\\
:=&I_6+I_7.
\end{align*}
Using \eqref{HLSin} and the Lebesgue dominated convergence theorem, we know $I_6=o(1)$. Furthermore, $I_7=o(1)$ when the measure of $\Omega$ is large enough. This ends the proof.
\end{proof}

\subsection{The behaviour of $E(a,b)$}

\begin{lemma}\label{minimax1}
Assume that $F$ satisfies $(F_1)$, $(F_2)$ and $F_{z_j}$($j=1,2$) has exponential critical growth at $\infty$, then for any fixed $(u,v)\in \mathcal{S}$, we have

$(i)$ $\mathcal{J}(\mathcal{H}((u,v),s))\rightarrow0^+$ as $s\rightarrow-\infty$;

$(ii)$ $\mathcal{J}(\mathcal{H}((u,v),s))\rightarrow-\infty$ as $s\rightarrow+\infty$.
\end{lemma}
\begin{proof}
It is obvious that there exist $s_1<<0$ such that
\begin{equation*}
  \|\Delta \mathcal{H}((u,v),s)\|_2^2= \|\Delta \mathcal{H}(u,s)\|_2^2+ \|\Delta\mathcal{H}(v,s)\|_2^2< \frac{8-\mu}{8}
\end{equation*}
for all $s\leq s_1$.
Fix $\alpha>32\pi^2$ close to $32\pi^2$ and $t>1$ close to $1$ such that
\begin{equation*}
\frac{8\alpha t\|\Delta \mathcal{H}((u,v),s)\|_2^2}{8-\mu}\leq 32\pi^2.
\end{equation*}
Repeating the proof of Lemma \ref{minimax01}, we derive the conclusion. 
\end{proof}

\begin{lemma}\label{equi}
Assume that $F$ satisfies $(F_1)$, $(F_2)$, $(F_{5})$ and $F_{z_j}$($j=1,2$) has exponential critical growth at $\infty$,  then for any fixed $(u,v)\in \mathcal{S}$,
the following statements hold.

$(i)$ The function ${\mathcal J}(\mathcal{H}((u,v),s))$ achieves its maximum with positive level at a unique point $s_{(u,v)}\in \mathbb{R}$ such that $\mathcal{H}((u,v),s_{(u,v)}) \in \mathcal{P}(a,b)$.

$(ii)$ The mapping $(u,v)\mapsto s_{(u,v)}$ is continuous in $(u,v)\in \mathcal{S}$.
\end{lemma}

\begin{proof}
Using Lemma \ref{minimax1}, with a similar proof of Lemma \ref{equi0}, we finish the proof.
\end{proof}

\begin{lemma}\label{continuous}
Assume that $F$ satisfies $(F_1)$, $(F_2)$ and $F_{z_j}$($j=1,2$) has exponential critical growth at $\infty$, then there exists $\delta>0$ small enough such that
\begin{equation*}
  \mathcal  J(u,v)
\geq\frac{1}{4}(\|\Delta u\|_2^2+\|\Delta v\|_2^2)
\end{equation*}
and
\begin{equation*}
  P(u,v)
\geq\|\Delta u\|_2^2+\|\Delta v\|_2^2
\end{equation*}
for all $(u,v)\in \mathcal{S}$ satisfying $\|\Delta u\|_2^2+\|\Delta v\|_2^2\leq \delta$.
\end{lemma}
\begin{proof}
For $\delta<\frac{8-\mu}{8}$, fix $\alpha>32\pi^2$ close to $32\pi^2$ and $t>1$ close to $1$ such that
\begin{equation*}
\frac{8\alpha t\|\Delta (u,v)\|_2^2}{8-\mu}\leq 32\pi^2,
\end{equation*}
with a similar proof of Lemma \ref{continuous0}, we complete the proof.
\end{proof}
\begin{lemma}\label{inf}
Assume that $F$ satisfies $(F_1)$, $(F_2)$, $(F_{5})$ and $F_{z_j}$($j=1,2$) has exponential critical growth at $\infty$, then we have
\begin{equation*}
\inf\limits_{(u,v)\in \mathcal{P}(a,b)}(\|\Delta u\|_2^2+\|\Delta v\|_2^2)>0 \quad \text{and}\quad E(a,b)>0.
\end{equation*}
\end{lemma}
\begin{proof}
Using Lemma \ref{continuous}, a similar argument of Lemma \ref{inf0} completes the proof.
\end{proof}

Using Lemma \ref{equi}, following the arguments of Lemmas \ref{nonincreasing0}-\ref{conclusion0}, we can prove the following lemmas  hold.

\begin{lemma}\label{nonincreasing}
Assume that $F$ satisfies $(F_1)$, $(F_2)$, $(F_{5})$ and $F_{z_j}$($j=1,2$) has exponential critical growth at $\infty$, then the functions $a\mapsto E(a,b)$ and $b\mapsto E(a,b)$ are non-increasing.
\end{lemma}
\begin{lemma}\label{close to}
Assume that $F$ satisfies $(F_1)$, $(F_2)$, $(F_{5})$ and $F_{z_j}$($j=1,2$) has exponential critical growth at $\infty$, suppose that \eqref{abs} possesses a ground state solution  with  $\lambda_1,\lambda_2<0$, then $E(a',b)<E(a,b)$ for any $a'>a$ close to $a$ and $E(a,b')<E(a,b)$ for any $b'>b$ close to $b$.
\end{lemma}
\begin{lemma}\label{conclusion}
Assume that $F$ satisfies $(F_1)$, $(F_2)$, $(F_{5})$ and $F_{z_j}$($j=1,2$) has exponential critical growth at $\infty$, suppose that \eqref{abs} possesses a ground state solution with  $\lambda_1,\lambda_2<0$, then $E(a',b)<E(a,b)$ for any $a'>a$ and $E(a,b')<E(a,b)$ for any $b'>b$.
\end{lemma}

\subsection{The estimation for the upper bound of $E(a,b)$}
In this subsection, we use the modified Adams functions introduced in \cite{CW1} to estimate the upper bound of $E(a,b)$, which is important for exponential critical growth problems.

For any $\varphi(t)\in C_0^\infty([0,\infty),[0,1])$ such that $\varphi(t)=1$ if $0\leq t\leq1$, $\varphi(t)=0$ if $t\geq2$. Define a sequence of functions $\tilde{\omega}_n$ by
\begin{align*}
  \begin{split}
 \tilde{\omega}_n(x)=\left\{
  \begin{array}{ll}
  \sqrt{\frac{\log n}{8\pi^2}}+\frac{1-n^2|x|^2}{\sqrt{32\pi^2\log n}},&\quad \text {for} \,\,\,|x|\leq \frac{1}{n},\\
  -\frac{\log |x|}{\sqrt{8\pi^2\log n}},&\quad \text {for} \,\,\,\frac{1}{n}<|x|\leq1,\\
  -\frac{\varphi(|x|)\log |x|}{\sqrt{8\pi^2\log n}},&\quad \text {for} \,\,\,1<|x|<2,\\
  0,&\quad \text {for} \,\,\,|x|\geq2.
    \end{array}
    \right.
  \end{split}
  \end{align*}
One can check that $\tilde{\omega}_n\in H^2(\mathbb{R}^4)$. A straightforward calculation shows that
\begin{equation*}
  \|\tilde{\omega}_n\|_2^2=\frac{1+32M_1}{128\log n}-\frac{1}{96 n^4}-\frac{1}{192 n^4 \log n}=\frac{1+32M_1}{128\log n}+o(\frac{1}{\log ^4n}),
\end{equation*}
\begin{equation*}
  \|\nabla\tilde{\omega}_n\|_2^2=\frac{1+2M_2}{8\log n}-\frac{1}{12n^2 \log n}=\frac{1+2M_2}{8\log n}+o(\frac{1}{\log ^3n}),
\end{equation*}
and
\begin{equation*}
  \|\Delta\tilde{\omega}_n\|_2^2=1+\frac{4+M_3}{4\log n},
\end{equation*}
where
 \begin{equation*}
 M_1=\int_1^2\varphi^2(r)r^3\log ^2rdr,
\end{equation*}
 \begin{equation*}
 M_2=\int_1^2\Big(\varphi'(r)\log r+\frac{\varphi(r)}{r}\Big)^2r^3dr,
\end{equation*}
and
\begin{equation*}
M_3=\int_1^2\Big(\varphi''(r)\log r+\frac{-\varphi(r)+3\varphi'(r)+2r\varphi'(r)+3r\varphi'(r)\log r}{r^2}\Big)r^3dr.
\end{equation*}
For any $c>0$, let $\omega_n^c=\frac{c\tilde{\omega}_n}{\|\tilde{\omega}_n\|_2}$. Then $\omega_n^c\in S(c)$ and
 \begin{equation}\label{t omega}
 \|\nabla \omega_n^c\|_2^2=\frac{c^2(\frac{1+2M_2}{8\log n}+o(\frac{1}{\log ^3n}))}{\frac{1+32M_1}{128\log n}+o(\frac{1}{\log ^4n})}=\frac{16c^2(1+2M_2)}{1+32M_1}\Big(1+o(\frac{1}{\log ^2n})\Big),
\end{equation}
 \begin{equation}\label{l omega}
 \|\Delta \omega_n^c\|_2^2=\frac{c^2(1+\frac{4+M_3}{4\log n})}{\frac{1+32M_1}{128\log n}+o(\frac{1}{\log ^4n})}=\frac{128c^2}{1+32M_1}\Big(\frac{4+M_3}{4}+\log n+o(\frac{1}{\log ^2n})\Big).
\end{equation}
Furthermore, we have
\begin{align}\label{deomega}
  \begin{split}
 \omega_n^c(x)=\left\{
  \begin{array}{ll}
 \frac{c(1+o(\frac{1}{\log ^3n}))}{\sqrt{\frac{1+32M_1}{128}}}\Big( \frac{\log n}{\sqrt{8\pi^2}}+\frac{1-n^2|x|^2}{\sqrt{32\pi^2}}\Big),&\quad \text {for} \,\,\,|x|\leq \frac{1}{n},\\
  -\frac{c(1+o(\frac{1}{\log ^3n}))}{\sqrt{\frac{1+32M_1}{128}}}\frac{\log |x|}{\sqrt{8\pi^2}},&\quad \text {for} \,\,\,\frac{1}{n}<|x|\leq1,\\
  -\frac{c(1+o(\frac{1}{\log ^3n}))}{\sqrt{\frac{1+32M_1}{128}}}\frac{\varphi(|x|)\log |x|}{\sqrt{8\pi^2}},&\quad \text {for} \,\,\,1<|x|<2,\\
  0,&\quad \text {for} \,\,\,|x|\geq2.
    \end{array}
    \right.
  \end{split}
  \end{align}
For any $t>0$, let
\begin{align*}
  g_n(t):=
\mathcal J(t\omega_n^a(t^{\frac{1}{2}}x),t\omega_n^b(t^{\frac{1}{2}}x))=&\frac{ t^2}{2}\int_{\mathbb{R}^4}(|\Delta \omega_n^a|^2+|\Delta \omega_n^b|^2)dx-\frac{t^{\frac{\mu}{2}-4}}{2}\int_{\mathbb{R}^4}(I_\mu*F(t\omega_n^a,t\omega_n^b))F(t\omega_n^a,t\omega_n^b)dx.
\end{align*}
By Lemmas \ref{equi} and \ref{inf}, we know $E(a,b)=\inf\limits_{(u,v)\in \mathcal{S}}\max\limits_{s\in \mathbb{R}}\mathcal J(\mathcal{H}((u,v),s))>0$, this together with $(\omega_n^a,\omega_n^b)\in \mathcal{S}$ yields that
\begin{equation*}
  0<E(a,b)\leq\max\limits_{s\in \mathbb{R}}\mathcal J(\mathcal{H}((\omega_n^a,\omega_n^b),s))=\max\limits_{t>0}g_n(t).
\end{equation*}
\begin{lemma}\label{attain}
Assume that $F$ satisfies $(F_1)$, $(F_2)$ and $F_{z_j}$($j=1,2$) has exponential critical growth at $\infty$, then for any fixed $n\in \mathbb{N}^+$, $\max\limits_{t>0}g_n(t)$ is attained at some $t_n>0$.
\end{lemma}
\begin{proof}
For any fixed $n\in \mathbb{N}^+$, as $t>0$ small, fix $\alpha>32\pi^2$ close to $32\pi^2$ and $m>1$ close to $1$ such that
\begin{equation*}
  \frac{8\alpha m\| \Delta(t\omega_n^a,t\omega_n^b)\|_2^2}{8-\mu}\leq 32\pi^2.
\end{equation*}
Arguing as \eqref{close01}, for $m'=\frac{m}{m-1}$, we have
\begin{align*}
 & \frac{t^{\frac{\mu}{2}-4}}{2}\int_{\mathbb{R}^4}(I_\mu*F(t\omega_n^a,t\omega_n^b))F(t\omega_n^a,t\omega_n^b)dx\\
\leq&
 Ct^{\frac{\mu}{2}-4}\Big(\|t\omega_n^a\|_{\frac{8(\tau+1)}{8-\mu}}^{\frac{8(\tau+1)}{8-\mu}}+\|t\omega_n^b\|_{\frac{8(\tau+1)}{8-\mu}}^{\frac{8(\tau+1)}{8-\mu}}\Big)^{\frac{8-\mu}{4}}+Ct^{\frac{\mu}{2}-4}
 \Big(\|t\omega_n^a\|_{\frac{8(q+1)m'}{8-\mu}}^{\frac{8(q+1)m'}{8-\mu}}+\|t\omega_n^b\|_{\frac{8(q+1)m'}{8-\mu}}^{\frac{8(q+1)m'}{8-\mu}}\Big)^{\frac{8-\mu}{4m'}} \\
=&Ct^{2\tau-2+\frac{\mu}{2}}\Big(\|\omega_n^a\|_{\frac{8(\tau+1)}{8-\mu}}^{\frac{8(\tau+1)}{8-\mu}}+\|\omega_n^b\|_{\frac{8(\tau+1)}{8-\mu}}^{\frac{8(\tau+1)}{8-\mu}}\Big)^{\frac{8-\mu}{4}}
  +Ct^{2q-2+\frac{\mu}{2}}\Big(\|\omega_n^a\|_{\frac{8(q+1)m'}{8-\mu}}^{\frac{8(q+1)m'}{8-\mu}}+\|\omega_n^b\|_{\frac{8(q+1)m'}{8-\mu}}^{\frac{8(q+1)m'}{8-\mu}}\Big)^{\frac{8-\mu}{4m'}}
\end{align*}
where $\tau>2-\frac{\mu}{4}$ and $q>2$. So $g_n(t)>0$ for $t>0$ small enough.
By \eqref{imp}, we obtain
\begin{equation*}
   \frac{t^{\frac{\mu}{2}-4}}{2}\int_{\mathbb{R}^4}(I_\mu*F(t\omega_n^a,t\omega_n^b))F(t\omega_n^a,t\omega_n^b)dx\geq \frac{t^{2\theta +\frac{\mu}{2}-4 }}{2}\int_{\mathbb{R}^4}(I_\mu*F(\omega_n^a,\omega_n^b))F(\omega_n^a,\omega_n^b)dx.
\end{equation*}
Since $\theta>3-\frac{\mu}{4}$, we have that $g_n(t)<0$ for $t>0$ large enough. Thus $\max\limits_{t>0}g_n(t)$ is attained at some $t_n>0$.
\end{proof}
\begin{lemma}\label{control}
Assume that $F$ satisfies $(F_1)$, $(F_2)$, $(F_6)$ and $F_{z_j}$($j=1,2$) has exponential critical growth at $\infty$, then $\max\limits_{t>0}g_n(t)<\frac{8-\mu}{16}$ for $n\in \mathbb{N}^+$ large enough.
\end{lemma}
\begin{proof}
By Lemma \ref{attain}, $\max\limits_{t>0}g_n(t)$ is attained at some $t_n>0$ and thus $g'_n(t_n)=0$. By $(F_2)$,
\begin{align}\label{useuse}
  & t_n^2\int_{\mathbb{R}^4}(|\Delta \omega_n^a|^2+|\Delta \omega_n^b|^2)dx=\frac{\mu-8}{4}t_n^{\frac{\mu}{2}-4}\int_{\mathbb{R}^4}(I_\mu*F(t_n\omega_n^a,t_n\omega_n^b))F(t_n\omega_n^a,t_n\omega_n^b)dx\nonumber\\
  &+t_n^{\frac{\mu}{2}-4}\int_{\mathbb{R}^4}(I_\mu*F(t_n\omega_n^a,t_n\omega_n^b))(\frac{\partial F(t_n\omega_n^a,t_n\omega_n^b)}{\partial (t_n\omega_n^a)}t_n\omega_n^a+\frac{\partial F(t_n\omega_n^a,t_n\omega_n^b)}{\partial (t_n\omega_n^b)}t_n\omega_n^b)dx\nonumber\\
  \geq& \frac{4\theta+\mu-8}{4\theta}t_n^{\frac{\mu}{2}-4}\int_{\mathbb{R}^4}(I_\mu*F(t_n\omega_n^a,t_n\omega_n^b))(\frac{\partial F(t_n\omega_n^a,t_n\omega_n^b)}{\partial (t_n\omega_n^a)}t_n\omega_n^a+\frac{\partial F(t_n\omega_n^a,t_n\omega_n^b)}{\partial (t_n\omega_n^b)}t_n\omega_n^b)dx.
\end{align}
By $(F_6)$, for any $\varepsilon>0$. there exists $R_\varepsilon>0$ such that for any $|z_1|,|z_2|\geq R_\varepsilon$,
\begin{equation}\label{ftF}
 F(z)[z\cdot \nabla F(z)]\geq (\varrho-\varepsilon)e^{64\pi^2 |z|^2}.
\end{equation}

{\bf Case $1$:} $\lim\limits_{n\rightarrow\infty}(t_n^2\log n)=0$. Then $\lim\limits_{n\rightarrow\infty}t_n=0$, by \eqref{t omega} and \eqref{l omega}, we have that
\begin{equation*}
\frac{  t_n^2}{2}\int_{\mathbb{R}^4}(|\Delta \omega_n^a|^2+|\Delta \omega_n^b|^2)dx\rightarrow 0,\,\, \text{as $n\rightarrow\infty$}.
\end{equation*}
 Noting that $F(t_n\omega_n^a,t_n\omega_n^b)>0$ by $(F_2)$, so we have
\begin{equation*}
  0<g_n(t_n)\leq\frac{  t_n^2}{2}\int_{\mathbb{R}^4}(|\Delta \omega_n^a|^2+|\Delta \omega_n^b|^2)dx,
\end{equation*}
thus $\lim\limits_{n\rightarrow\infty}g_n(t_n)=0$, and we conclude.

{\bf Case $2$:} $\lim\limits_{n\rightarrow\infty}(t_n^2\log n)=l\in (0,\infty]$. We claim that $l<\infty$. Otherwise, if
$l=\infty$, then $\lim\limits_{n\rightarrow\infty}(t_n\log n)=\infty$.
By \eqref{t omega}-\eqref{deomega} and \eqref{useuse}-\eqref{ftF}, we have
\begin{align*}
  &\frac{128(a^2+b^2)t_n^2}{1+32M_1}\Big(\frac{4+M_3}{4}+\log n+o(\frac{1}{\log ^2n})\Big)
  \\
\geq& \frac{4\theta+\mu-8}{4\theta}t_n^{\frac{\mu}{2}-4} \int_{B_{\frac{1}{n}}(0)}\int_{B_{\frac{1}{n}}(0)}\frac{F(t_n\omega_n^a(y),t_n\omega_n^b(y)) (\frac{\partial F(t_n\omega_n^a(x),t_n\omega_n^b(x))}{\partial (t_n\omega_n^a(x))}t_n\omega_n^a(x)+\frac{\partial F(t_n\omega_n^a(x),t_n\omega_n^b(x))}{\partial (t_n\omega_n^b(x))}t_n\omega_n^b(x)) }{|x-y|^\mu}dxdy\\
\geq& \frac{(4\theta+\mu-8)(\varrho-\varepsilon)^2}{4\theta }t_n^{\frac{\mu}{2}-4}e^{\frac{1024 (a^2+b^2) t_n^2 \log^2 n(1+o(\frac{1}{\log^3 n}))}{1+32M_1}}\int_{B_{\frac{1}{n}}(0)}\int_{B_{\frac{1}{n}}(0)}\frac{dxdy}{|x-y|^\mu}.
\end{align*}
Since $B_{\frac{1}{n}-|x|}(0)\subset B_{\frac{1}{n}}(x)$ for any $|x|\leq \frac{1}{n}$, the last integral can be estimated as follows
\begin{align*}
  \int_{B_{\frac{1}{n}}(0)}\int_{B_{\frac{1}{n}}(0)}\frac{dxdy}{|x-y|^\mu}&=\int_{B_{\frac{1}{n}}(0)}dx\int_{B_{\frac{1}{n}}(x)}\frac{dz}{|z|^\mu}\\
  &\geq\int_{B_{\frac{1}{n}}(0)}dx\int_{B_{\frac{1}{n}-|x|}(0)}\frac{dz}{|z|^\mu}\\
  &=\frac{2\pi^2}{4-\mu} \int_{B_{\frac{1}{n}}(0)}\Big(\frac{1}{n}-|x|\Big)^{4-\mu}dx\\
  &=\frac{4\pi^4}{4-\mu} \int_{0}^{\frac{1}{n}}\Big(\frac{1}{n}-r\Big)^{4-\mu}r^3dr\\
  &=\frac{24\pi^4}{(4-\mu)(5-\mu)(6-\mu)(7-\mu)(8-\mu)n^{8-\mu}}=\frac{C_\mu}{n^{8-\mu}},
\end{align*}
where
\begin{equation*}
  C_\mu=\frac{24\pi^4}{(4-\mu)(5-\mu)(6-\mu)(7-\mu)(8-\mu)}.
\end{equation*}
Consequently, we obtain
\begin{align}\label{tomainuse}
  &\frac{128 (a^2+b^2)t_n^2}{1+32M_1}\Big(\frac{4+M_3}{4}+\log n+o(\frac{1}{\log ^2n})\Big)\nonumber\\
\geq & \frac{(4\theta+\mu-8)C_\mu(\varrho-\varepsilon)^2}{4\theta }t_n^{\frac{\mu}{2}-4}e^{\Big(\frac{1024 (a^2+b^2) t_n^2 \log n(1+o(\frac{1}{\log^3 n}))}{1+32M_1}-(8-\mu)\Big)\log n},
\end{align}
which is a contradiction.
 Thus $l\in(0,\infty)$, and $\lim\limits_{n\rightarrow\infty}t_n=0$, $\lim\limits_{n\rightarrow\infty}(t_n\log n)=\infty$. By \eqref{tomainuse}, letting $n\rightarrow\infty$, we have that
\begin{equation*}
 0< l< \frac{(1+32M_1)(8-\mu)}{1024(a^2+b^2)}.
\end{equation*}
 Otherwise, if $l> \frac{(1+32M_1)(8-\mu)}{1024(a^2+b^2)}$, the right side of \eqref{tomainuse} approaches to infinity as $n\rightarrow\infty$, which is impossible.
 If $l=\frac{(1+32M_1)(8-\mu)}{1024(a^2+b^2)}$, by the definition of $\omega_n^a$ and $\omega_n^b$, we can find that
 \begin{equation*}
   A_n:=\frac{1024 (a^2+b^2) t_n^2 \log n(1+o(\frac{1}{\log^3 n}))}{1+32M_1}-(8-\mu)\rightarrow 0^+,\quad \text{as $n\rightarrow\infty$}
\end{equation*}
and using the Taylor's formula, we have
\begin{equation*}
  n^{A_n}=1+A_n\log n+\frac{A_n^2\log ^2n}{2}+\cdots\geq1.
\end{equation*}
Thus
\begin{equation*}
  \frac{128 (a^2+b^2)t_n^2 \log n}{1+32M_1}\geq \frac{(4\theta+\mu-8)C_\mu(\varrho-\varepsilon)^2}{4\theta } t_n^{\frac{\mu}{2}-4},
\end{equation*}
which is a contradiction.
Therefore,
\begin{align*}
  \lim\limits_{n\rightarrow\infty}g_n(t_n)&\leq  \lim\limits_{n\rightarrow\infty}\frac{  t_n^2}{2}\int_{\mathbb{R}^4}(|\Delta \omega_n^a|^2+|\Delta \omega_n^b|^2)dx\\
  =&\frac{64 (a^2+b^2) l }{1+32M_1}<\frac{8-\mu}{16}.
\end{align*}
This ends the proof.
\end{proof}

\subsection{Palais-Smale sequence}
Similar to the arguments in section \ref{appro0}, we apply Lemma \ref{Ghouss} to construct a $(PS)_{E(a,b)}$ sequence on $\mathcal{P}(a,b)$ for $\mathcal J$. Here, we only state the conclusion without proof.

\begin{lemma}\label{C^1}
Assume that $F$ satisfies $(F_1)$, $(F_2)$, $(F_{5})$ and $F_{z_j}$($j=1,2$) has exponential critical growth at $\infty$, then the functional $\mathcal{I}(u,v)=\mathcal J(\mathcal{H}((u,v),s_{(u,v)}))$ is of class $C^1$ and
\begin{align*}
  \langle \mathcal{I}'(u,v),(\varphi,\psi)\rangle=
  \langle\mathcal J'(\mathcal{H}((u,v),s_{(u,v)})),\mathcal{H}((\varphi,\psi),s_{(u,v)})\rangle
\end{align*}
for any $(u,v)\in \mathcal{S}$ and $(\varphi,\psi)\in T_{(u,v)}\mathcal{S}$.
\end{lemma}

\begin{lemma}\label{pssequencehomo2}
Assume that $F$ satisfies $(F_1)$, $(F_2)$, $(F_{5})$ and $F_{z_j}$($j=1,2$) has exponential critical growth at $\infty$. let $\mathcal{G}$  be a homotopy stable family of compact subsets of $\mathcal{S}$ without boundary (i.e., $B=\emptyset$) and set
\begin{equation*}
  e_{\mathcal{G}}:=\inf\limits_{A\in \mathcal{G}}\max\limits_{(u,v)\in A}\mathcal I(u,v).
\end{equation*}
If $e_{\mathcal{G}}>0$, then there exists a $(PS)_{e_{\mathcal{G}}}$ sequence $\{(u_n,v_n)\}\subset \mathcal{P}(a,b)$ for $\mathcal J$.
\end{lemma}

\begin{proposition}\label{pssequencehomo}
Assume that $F$ satisfies $(F_1)$, $(F_2)$, $(F_{5})$ and $F_{z_j}$($j=1,2$) has exponential critical growth at $\infty$, then there exists a $(PS)_{E(a,b)}$ sequence $\{(u_n,v_n)\}\subset \mathcal{P}(a,b)$ for $\mathcal J$.
\end{proposition}
For the sequence $\{(u_n,v_n)\}$ obtained in Proposition \ref{pssequencehomo}, by $(F_2)$, we know that $\{(u_n,v_n)\}$ is bounded in $\mathcal{X}$, up to a subsequence, we assume that $(u_n,v_n)\rightharpoonup (u_{a},v_{b})$ in $\mathcal{X}$.
Furthermore, by ${\mathcal J}'|_{\mathcal{S}}(u_n,v_n)\rightarrow0$ as $n\rightarrow\infty$ and Lagrange multiplier rule, there exist two sequences $\{\lambda_{1,n}\}, \{\lambda_{2,n}\}\subset \mathbb{R}$ such that
\begin{align}\label{lagrange}
&\int_{\mathbb{R}^4}(\Delta u_n \Delta \varphi+\Delta v_n\Delta \psi)dx-\int_{\mathbb{R}^4}(I_\mu*F(u_n,v_n)) F_{u_n}(u_n,v_n)\varphi dx-\int_{\mathbb{R}^4}(I_\mu*F(u_n,v_n))F_{v_n}(u_n,v_n)\psi dx\nonumber \\
=&\int_{\mathbb{R}^4}(\lambda_{1,n}u_n\varphi+\lambda_{2,n}v_n\psi)dx+o_n(1)\|(\varphi,\psi)\|,
\end{align}
for every $(\varphi,\psi)\in \mathcal{X}$.
\begin{lemma}\label{not01}
Assume that $F$ satisfies $(F_1)-(F_3)$, $(F_5)$, $(F_6)$ and $F_{z_j}$($j=1,2$) has exponential critical growth at $\infty$, then up to a subsequence and up to translations in $\mathbb{R}^4$, $u_a\neq0$ and $v_b\neq0$. 
\end{lemma}
\begin{proof}
If $(u_a,v_b)=(0,0)$,
by ${\mathcal J}(u_n,v_n)=E(a,b)+o_n(1)$, $P(u_n,v_n)=0$ and $(F_2)$, we get
\begin{align*}
  {\mathcal J}(u_n,v_n)-\frac{1}{2}P(u_n,v_n)&=\frac{1}{2}\int_{\mathbb{R}^4}(I_\mu*F(u_n,v_n))\Big[ (u_n,v_n)\cdot\nabla F(u_n,v_n )-(3-\frac{\mu}{4})F(u_n,v_n)\Big]dx\nonumber\\
  &\geq\frac{4\theta+\mu-12}{8}\int_{\mathbb{R}^4}(I_\mu*F(u_n,v_n))F(u_n,v_n)dx.
\end{align*}
    By $\theta>3-\frac{\mu}{4}$, we have
\begin{equation}\label{FF}
 \limsup\limits_{n\rightarrow\infty} \int_{\mathbb{R}^4} (I_\mu*F(u_n,v_n))F(u_n,v_n) dx\leq \frac{8 E(a,b)}{4\theta+\mu-12}.
\end{equation}
This with $P(u_n,v_n)=0$ and the boundedness of $\{(u_n,v_n)\}$ implies that, up to a subsequence, there exists constant $L_0>0$ such that
\begin{equation}\label{L0}
  \int_{\mathbb{R}^4} (I_\mu*F(u_n,v_n))[(u_n,v_n)\cdot \nabla F(u_n,v_n)] dx\leq L_0.
\end{equation}
From Lemma \ref{strong}, we can see
\begin{equation*}
  \int\limits_{\mathbb{R}^4} (I_\mu*F(u_n,v_n))F(u_n,v_n)dx=o_n(1).
\end{equation*}
Hence, by Lemma \ref{control}, we have
\begin{equation*}
 \limsup\limits_{n\rightarrow\infty} (\|\Delta u_n\|_2^2+\|\Delta v_n\|_2^2)\leq 2E(a,b)<\frac{8-\mu}{8}.
\end{equation*}
Up to a subsequence, we assume that $\sup\limits_{n\in \mathbb{N}^+} (\|\Delta u_n\|_2^2+\|\Delta v_n\|_2^2)<\frac{8-\mu}{8}$. Using \eqref{HLS} again, we have
\begin{equation*}
  \int\limits_{\mathbb{R}^4} (I_\mu*F(u_n,v_n))[(u_n,v_n)\cdot \nabla F(u_n,v_n)]dx\leq C\|F(u_n,v_n)\|_{\frac{8}{8-\mu}}\|(u_n,v_n)\cdot \nabla F(u_n,v_n)\|_{\frac{8}{8-\mu}}.
\end{equation*}
Fix $\alpha>32\pi^2$ close to $32\pi^2$ and $m>1$ close to $1$ such that
 \begin{equation*}
   \sup\limits_{n\in \mathbb{N}^+} \frac{8\alpha m \|\Delta (u_n,v_n)\|_2^2}{8-\mu}\leq 32\pi^2.
\end{equation*}
Arguing as \eqref{close01}, for $m'=\frac{m}{m-1}$, we have
\begin{equation*}
  \|F(u_n,v_n)\|_{\frac{8}{8-\mu}}\leq C\Big(\|u_n\|_{\frac{8(\tau+1)}{8-\mu}}^{\frac{8(\tau+1)}{8-\mu}}+\|v_n\|_{\frac{8(\tau+1)}{8-\mu}}^{\frac{8(\tau+1)}{8-\mu}}\Big)^{\frac{8-\mu}{8}}+C\Big(\|u_n\|_{\frac{8(q+1)t'}{8-\mu}}^{\frac{8(q+1)t'}{8-\mu}}+\|v_n\|_{\frac{8(q+1)t'}{8-\mu}}^{\frac{8(q+1)t'}{8-\mu}}\Big)^{\frac{8-\mu}{8t'}}\rightarrow0,\quad\text{as $n\rightarrow\infty$}.
\end{equation*}
Similarly, we can prove that
\begin{equation*}
  \|(u_n,v_n)\cdot \nabla F(u_n,v_n)\|_{\frac{8}{8-\mu}}\rightarrow0,\quad\text{as $n\rightarrow\infty$}.
\end{equation*}
Therefore, we get
\begin{equation*}
  \int\limits_{\mathbb{R}^4} (I_\mu*F(u_n,v_n))[(u_n,v_n)\cdot \nabla F(u_n,v_n)]dx=o_n(1).
\end{equation*}
Recalling that $P(u_n,v_n)=0$, so $\|\Delta u_n\|_2+\|\Delta v_n\|_2=o_n(1)$, which implies $E(a,b)=0$, this is impossible, since $E(a,b)>0$.
Hence $(u_a,v_b)\neq(0,0)$.  From \eqref{lagrange} and Lemma \ref{strong}, we can see that $(u_a,v_b)$ is a weak solution of \eqref{e1.1} with $N=4$. Assume that $u_{a}=0$, then by $(F_3)$, we know that $v_{b}=0$. Similarly, $v_{b}=0$ implies $u_{a}=0$. This is impossible, since $(u_a,v_b)\neq(0,0)$.
\end{proof}

\begin{lemma}\label{lam}
Assume that $F$ satisfies $(F_1)-(F_6)$, and $F_{z_j}$($j=1,2$) has exponential critical growth at $\infty$, then $\{\lambda_{1,n}\}$ and $\{\lambda_{2,n}\}$ are bounded in $\mathbb{R}$. Furthermore, up to a subsequence,
$\lambda_{1,n}\rightarrow\lambda_1<0$ and $\lambda_{2,n}\rightarrow\lambda_2<0$ in $\mathbb{R}$ as $n\rightarrow\infty$. 
\end{lemma}
\begin{proof}
The proof is similar to Lemma \ref{lam0}, so we omit it.
\end{proof}

\subsection{Proof of Theorem \ref{th5}}
\noindent
{\it Proof of Theorem \ref{th5}:} Using \eqref{L0}, Lemmas \ref{strong}, \ref{nonincreasing}, \ref{conclusion}, \ref{not01}, \ref{lam}, we can proceed exactly as the proof of Theorem \ref{th3}.
\qed

\subsection*{Acknowledgments}

The authors were  supported by National Natural Science Foundation of China 11971392.

\end{document}